\documentclass[review,3p,times]{elsarticle}

\usepackage{amsmath,latexsym,amssymb,amsfonts,amsbsy}
\usepackage{amssymb}
\usepackage{amsmath}
\usepackage{amsthm}
\usepackage{cases}
\usepackage{mathrsfs}
\usepackage{graphicx,color}
\DeclareGraphicsExtensions{.eps,.bmp}
\usepackage{ctex}
\CTEXoptions[today=old]

\usepackage{graphicx}
\usepackage{subcaption}
\usepackage{indentfirst}
\usepackage{verbatim}
\usepackage{epstopdf}
\usepackage{tikz}
\usetikzlibrary{decorations.pathreplacing,calligraphy}
\usepackage{booktabs}
\usepackage{tikz-3dplot}
\usepackage{multirow}
\usepackage{algorithm}
\usepackage{algorithmic}

\usepackage[hidelinks]{hyperref}

\usepackage{tikz}
\usetikzlibrary{decorations.pathreplacing,calligraphy,arrows.meta,scopes,calc,intersections}

\newtheorem{theorem}{Theorem}

\newtheorem{lemma}{Lemma}

\newcommand{\hf}{{\frac{1}{2}}}
\newcommand{\ah}{{\frac{1}{6}}}
\newcommand{\bh}{{\frac{1}{3}}}
\newcommand{\fh}{{\frac{2}{3}}}
\newcommand{\gh}{{\frac{5}{6}}}
\newcommand{\bV}{{\bf v}}
\newcommand{\bVi}{{\bf v}^{i}}

\newcommand{\bN}{{\bf n}}
\newcommand{\bT}{{\bf t}}
\newcommand{\bx}{\boldsymbol{x}}

\newcommand{\bxi}{\boldsymbol{\xi}}
\newcommand{\bphi}{\boldsymbol{\phi}}

\newcommand{\ustar}{\tilde{u}}
\newcommand{\U}{\mathbf{U}}

\newcommand{\bU}{\bar{\mathbf{U}}}

\newcommand{\tbF}{\tilde{\mathbf{F}}}
\newcommand{\tbU}{\tilde{\mathbf{U}}}

\newcommand{\bF}{\mathbf{F}}

\newcommand{\bG}{\mathbf{G}}

\newcommand{\cO}{\mathcal{O}}

\newcommand{\cE}{\mathcal{E}}

\newcommand{\tbV}{\tilde{\bf v}}
\newcommand{\tbVs}{\tilde{\bf v}}
\newcommand{\tu}{\tilde{u}}

\newcommand{\tP}{\widetilde{P}}
\newcommand{\tPs}{\widetilde{P}}

\newcommand{\bra}[1]{\left(#1 \right)}
\newcommand{\brb}[1]{\left[#1 \right]}
\newcommand{\brc}[1]{\left\{#1 \right\}}

\newcommand{\avg}{\mathrm{avg}} 
\newcommand{\pt}{\mathrm{\textsc{pv}}} 
\newcommand{\via}{\mathrm{\textsc{via}}}

\newcommand{\dd}{\mathrm{d}}

\newcommand{\pd}[2]{\frac{\partial #1}{\partial #2}}

\newcommand{\tPsi}{\tP^i}
\newcommand{\tPsj}{\tP^j}
\newcommand{\tbVi}{\tbV^i}

\newcommand{\tVi}{\tilde{v}^i}
\newcommand{\tcVs}{\widetilde{\mathcal{V}}}
\newcommand{\cVs}{\mathcal{V}}

\begin{document}
\begin{frontmatter}


\title{ 
A third-order multi-moment cell-centered Lagrangian scheme for hydrodynamics with an accurate 2D nodal solver
}
\author{Xiaoteng Zhang$^{1,2}$}
\ead{zxt2019@pku.edu.cn}

\author{Xun Wang$^{3}$}
\ead{s151025@muc.edu.cn}

\author{Zhijun Shen$^{4,5}$}
\ead{shen_zhijun@iapcm.ac.cn}

\author{Chao Yang\corref{cor1}$^{1,2}$}
\ead{chao_yang@pku.edu.cn}

\cortext[cor1]{Corresponding author.}

\address{ 1. School of Mathematical Sciences, Peking University, Beijing 100871, China\\
2. PKU-Changsha Institute for Computing and Digital Economy, Hunan 410205, China\\
3.     Academy of Mathematics and Systems Science,
Chinese Academy of Sciences,  Beijing, 100190, PR China\\
4. Institute of Applied Physics and Computational Mathematics, LCP,P. O. Box 8009-26, Beijing 100088, China\\
5. HEDPS, Center for Applied Physics and Technology, and College of
Engineering, Peking University, Beijing 100871, China
}

\begin{abstract}
This paper presents a novel high-order cell-centered Lagrangian scheme for 2D compressible hydrodynamics by bridging the multi-moment constrained finite volume method (MCV) \cite{Ii-MCV, Xie-2DMCV, Xie-2DMCVi} with a nodal Riemann solver.
This scheme (denoted by LMCV) not only maintains high-order accuracy as MCV but also inherits the conservation and robust properties of the nodal Riemann solver. 
On the one hand, the MCV  employs and evolves both the point-values (PV) at cell vertexes and the volume-integrated averages (VIA) on computational mesh, which ensures the rigorous numerical conservation and 
establishes an adequate foundation for the computation of Lagrangian fluxes with high accuracy.
On the other hand, we developed a 2D Riemann solver based on EUCCLHYD \cite{Maire-EUCCLHYD}, it takes fully advantage of numerical formulations from high-order scheme and accomplishes the compatibility between the mesh movement and numerical fluxes. 
The main new features of the solver are the introduction of a new set of jump and balance conditions. The jump condition provides a high-accurate formulation linking the surface pressure of each cell to its nodal velocity, while the balance condition ensures nodal conservation and stabilizes the velocity field without losing accuracy. 
More intriguing is that our nodal solver can be regarded as a natural high-order extension of the HLLC and the HLLC-2D \cite{Shen-HLLC2D} solvers. The comparison between these solvers better demonstrates our innovative approach in addressing the difficulties encountered in constructing 2D high-order Lagrangian schemes. A variety of numerical experiments are carried out to illustrate the accuracy and robustness of the algorithm.
\end{abstract}



\begin{keyword}
Lagrangian method; high order accuracy; compressible flow; Multi-moment method

\end{keyword}

\end{frontmatter}

\section{Introduction} \label{chap:Lag:Intro}


\noindent 

The hypersonic flows problems are frequently found in the two main applications: design research on hypersonic flight vehicles and assessment studies of their aerodynamic and aerothermal characteristics. These problems are often too complicated for analytical investigation and present significant challenges for the experimental observation or measurement of physical variables. Numerical simulation provides effective methods and insights for the interpretation of compressible flow phenomena across virtually all scales, ranging from laboratory experiments to astrophysical environments. Since its inception in the late 1950s\cite{Godunov}, Computational Fluid Dynamics (CFD) has evolved significantly. While early methods achieved stability and shock-capturing capability\cite{Despres-GLACE, Maire-EUCCLHYD, Maire-Staggered, morgan2015godunov, Shen-HLLC2D,vilarPositivity2016P2}, they were often hampered by excessive numerical dissipation—a limitation that ultimately spurred the development of higher-order approaches. Before describing high-order Lagrangian method\cite{Cheng-3rd, Liu-subcellLDG, Morgan-LDG,Morgan-LDG-2,Vilar-LDG}, let us briefly give a historical overview of the higher-order finite volume schemes. 

Starting with the monotonic upstream-centered scheme for conservation laws method (MUSCL) \cite{vanLeer}
proposed by Bram van Leer, subsequent high-order schemes have developed in two dimensions: reconstruction of each stencil and degrees of freedom of each cell. 
Representing by essentially non-oscillatory scheme (ENO) \cite{Harten-ENO}, weighted ENO scheme (WENO) \cite{Liu-WENO,Jiang-WENO} and Hermite WENO scheme (HWENO) \cite{Qiu-HWENO-1, Qiu-HWENO-2,Zhu-HWENO-3}, these schemes focus on improving accuracy by applying more effective reconstruction on wide stencils. Other schemes like discontinuous Galaken method (DG) \cite{Cockburn-DG-2,Cockburn-DG-3,Cockburn-DG-4,Cockburn-DG-5}, spectral difference method (SD) \cite{Liu-SD,Wang-SD} and multi-moment constrained finite volume method (MCV) \cite{Ii-MCV}, increase the order of accuracy by employing extra degree of freedom (DOF) at each cell instead of widening the reconstructing stencil. Despite improvements and encouraging results in recent years, a number of issues still remain. Most of high-order schemes are formulated in Euler framework for single-material, which remains a primary challenge: capturing of contact discontinuities in multi-material fluid flows. 


Many interface-capturing methods, such as marker particle and front-tracking approaches, are developed based on Lagrangian techniques \cite{Despres-GLACE, Law, Maire-EUCCLHYD}. In multi-dimensional case, solving Lagrangian formulations could be a challenge. 
Due to one point could be surrounded by a large number of control volumes with different states in multi-dimension, a straightforward usage of 1D Riemann solvers \cite{Roe, Harten, Wendroff} cannot provide the point velocity and surrounding numerical fluxes. In this scenario, the nodal Riemann solver was proposed as a solution, and used as a fundamental part in classic cell-centered Lagrangian schemes including GLACE \cite{Despres-GLACE} and EUCCLHYD \cite{Maire-EUCCLHYD,Maire-HighOrder}. 
Based on these first-order schemes, high-order Lagrangian schemes are developed. Take two-dimension for example, in order to track fluid motion in a velocity field with high-accuracy, most of existing schemes introduce curvilinear meshes, which require extra shape-control points on each edge. As a result, velocities and fluxes need be solved at both cell corners and edges. Several strategies are developed in this case. 

In work by Vilar et al.	\cite{Vilar-LDG}, 
the multidimensional approximate Riemann solver by Maire et al. \cite{Maire-EUCCLHYD,Maire-HighOrder} is used to solve the Riemann problem at the corner, while a 1D Riemann solver is used at the middle of each edge. Morgan et al. used this same approach for quadratic triangular cells \cite{Morgan-LDG} with two different Riemann solvers. A weakness of this approach is that the number of inputs to the Riemann problem varies along each edge. The inconsistency between the Riemann solvers may give rise to some spurious mesh motion or even self intersection. 
The above issues have been mentioned in various articles, some of which use post-processing techniques to cure the ill-conditioned velocity field. In work by Cheng and Shu \cite{Cheng-3rd}, a velocity limitation is introduced to prevent cell degeneration without destroying the third-order accuracy of the numerical scheme. Similarly, a velocity filter is proposed by Morgan et al. \cite{Morgan-LDG-2} to dissipate spurious mesh motion by modifying the velocity reconstructions. The reconstructed velocities are adjusted as a function of the difference between the resulting vertex Riemann velocities and a linear fit to these velocities. The strategy works well for many problems with strong shocks. 
In an alternative perspective, several articles are investigating the treatment of the underlying causes of spurious motion through pre-processing techniques. In work by Xiaodong Liu et al. \cite{Liu-subcellLDG}, each quadratic quadrilateral cell is carefully reconstructed into four quadrilateral subcells by subcell mesh stabilization (SMS). The middle point of each edge is surrounded by four subcells so that it is similar to the vertex at the cell corner, which makes the multidimensional Riemann solver consistently applicable at every vertices. This SMS scheme enables stable mesh motion and accurate solutions in the context of a Lagrangian high-order DG method that is up to third-order with quadratic cells.

In this paper, we develop a novel high-order cell-centered Lagrangian scheme for 2D compressible hydrodynamics based on MCV, which achieves third-order accuracy using an augmented 2D nodal Riemann solver. The MCV discretizes the governing equations by both PVs and VIAs, which are evolved separately by Lax-Friedrichs method and our nodal solver. The main features of this solver are the introduction of a new set of jump and balance conditions, which not only provide a formulation between surface pressure and nodal velocity with high-order accuracy but also stabilizes the nodal velocity field without losing accuracy or breaking nodal conservation. Furthermore, 3rd order TVD Runge-Kutta method and several limiting strategies are employed for accuracy and stability. The combination of MCV and our nodal solver, enables the cell-centered Lagrangian method, for the first time, to acquire third-order accuracy without introducing curvilinear meshes. 

The layout of this paper is as follows. The governing equations are introduced in Section \ref{chap:Lag:Govern} and spatially discretized in Section \ref{chap:Lag:LMCV}. High-order Lagrangian schemes and used Riemann solvers are compared in Section \ref{Chap:Lag:Revisit}. In Section \ref{chap:Lag:Nodal}, our augmented nodal solver is discussed in detail. Additional procedures for our scheme is introduced in Section \ref{chap:Lag:Proc}. Numerical results are demonstrated in Section \ref{chap:Lag:Test}. Concluding remarks are given in Section \ref{chap:Lag:Conclusion}. The accuracy limitation of tranditional Lagrangian schemes is illustrated in Appendix \hyperlink{Apd.2rd}{A}. The analysis of flux accuracy is shown in Appendix \hyperlink{Apd.4th}{B}. The details of reconstruction procedure are illustrated in Appendix \hyperlink{Apd.PV}{C}.

\section{Governing equations} \label{chap:Lag:Govern}
The standard 2D Euler equations can be written as: 
\begin{eqnarray}\label{2DEuler}
\pd{\U}{t} + \nabla \cdot \mathbb{F}(\U) = {\bf 0}, 
\end{eqnarray}
 where the conserved variables $\U$ and   flux $\mathbb{F}$ are
\begin{eqnarray*}
\U=
\begin{bmatrix}
\rho \\
\rho {\bf v} \\ 
\rho E
\end{bmatrix}, \quad 
\mathbb{F}= \begin{bmatrix}
\rho{\bf v}^\top \\
\rho {\bf v}\otimes {\bf v} + P \mathbb{I}_{2} \\ 
(\rho E + P){\bf v}^\top
\end{bmatrix}
\end{eqnarray*}
where $\rho,P,E$ are the fluid density, pressure and total energy respectively,  ${\bf v} = (v_x, v_y)^\top$  is the fluid velocity, $\mathbb{I}_{2}$ is identity matrix of size $2$. 
The above system is closed by an equation of state (EOS)
\begin{eqnarray}
P = \rho (\gamma-1) e= \rho (\gamma-1) \left( E-\hf |{\bf v}|^2\right),
\end{eqnarray}
where $\gamma$ is the specific heat ratio.

According to Reynolds transport theorem, it is easy to recast the system (\ref{2DEuler}) into the following moving control volume formulation
\begin{eqnarray} \label {1.2}
\frac{\mathrm{d}}{\mathrm{dt}} \int _{\omega} {\bf U}\mathrm{d}\omega +\int _{\partial \omega} [\mathbb{F} {\bf n}-({\bf w} \cdot {\bf n}){\bf U}]\mathrm{d}l=0,
\end{eqnarray}
where $\bf {w}$ is the  moving velocity of    the control volume boundary $\partial\omega$, and ${\bf n}=(n_x, n_y)^{\top}$ is the unit outward normal vector on the boundary of $\omega$.  
If ${\bf w} = 0$, the system reduces to a  Eulerian form, and if ${\bf w} = (v_x, v_y)^{\top}$, one obtains a Lagrangian formulation.

In the Lagrangian framework, the 2D Euler equations can be formulated
\begin{eqnarray}
\frac{\dd}{\dd t}\int_{\omega} \rho\, \dd\omega= 0, \label{lag_rho}\\
\frac{\dd}{\dd t}\int_{\omega} \rho {\bf v} \dd\omega + \int_{\partial \omega} P {\bf n}\dd l = \boldsymbol{0}, \label{lag_vel^int} \\
\frac{\dd}{\dd t}\int_{\omega} \rho E \dd\omega + \int_{\partial \omega} P{\bf v}\cdot  {\bf n}\dd l= 0, \label{lag_eng^int}
\end{eqnarray} 

In addition, the variation in time of a control volume needs to satisfy the geometric conservation law (GCL)
\begin{eqnarray}\label{lag_vol^int}
\frac{\dd}{\dd t}\int_{\omega}d \omega- \int_{\partial \omega} {\bf v}\cdot {\bf n}\dd l= 0, 
\end{eqnarray}

\section{Spatial discretization} \label{chap:Lag:LMCV}


\subsection{Notations}

Suppose the computational domain is divided into $I$ non-overlapping quadrilateral cells  $\{\omega^{i},i=1,\cdots,I\}$, each cell $\omega^i$ is uniquely idnetified by its periodic four vertices coordinate in the counterclockwise ordered. For arbitrary quadrilaterals $\omega^i$,  an iso-parametric transformation is introduced  to describe its geometry relations.  A canonical square $\Omega = [-1,1]^2$ is introduced as a reference cell, the transformation is denoted as 

$$\Phi^i(\bxi) : \Omega \to \omega^i,$$ 
with 
\begin{eqnarray}
\Phi^i(\xi,\eta) =  \frac{1}{4} \sum_{r=1}^{4}(1+ \xi_r \xi)(1+ \eta_r \eta) \bx^i_r,   
\end{eqnarray}
where $\bxi=(\xi,\eta)\in[-1,1]^2$, $\Phi^i(\bxi_r) = \bx^i_r$ and $\Phi^i(\boldsymbol{0}) = \bx^i_c$, referring to Fig. \ref{fig_Ci_quad}. 



{\color{black}
\begin{figure}[!htp]
\centering
\begin{tikzpicture}[scale=1]
\coordinate (Oa) at (0,0);
\coordinate (Ob) at (6,-2.3);

\draw [->] (Oa)-- ++(0.7,0) node [right] {$\xi$};
\draw [->] (Oa)-- ++(0,0.7) node [above] {$\eta$};

\draw (Oa)++(-2,-2) node [above right] {$(-1,-1)$}
-- ++ (4,0)     node [above left ] {$(1,-1)$}
-- ++ (0,4)     node [below left ] {$(1,1)$}
-- ++ (-4,0)    node [below right] {$(-1,1)$}
-- cycle;

\draw[-{Stealth[scale=1.5]}] (3.7,0)--(5.3,0);
\draw (4.5,0) node [above] {$\Phi^i(\bxi)$};

\fill (Oa)++(-2,-2) circle (1.5pt) node [left]    {$\bxi_{1}$}
(Oa)++(2,-2)  circle (1.5pt) node [right]   {$\bxi_{2}$}
(Oa)++(2,2)   circle (1.5pt) node [right]   {$\bxi_{3}$}
(Oa)++(-2,2)  circle (1.5pt) node [left]    {$\bxi_{4}$};

\draw [->] (Ob)-- ++(0.7,0) node [right] {$x$};
\draw [->] (Ob)-- ++(0,0.7) node [above] {$y$};

\coordinate (x1) at (1.1,0.6);
\coordinate (x12) at (3.5,1);
\coordinate (x23) at (-1,2.5);
\coordinate (x34) at (-2.7,-0.2);
\coordinate (x5) at (
$(x1)+0.75*(x12)+0.5*(x23)+0.25*(x34)$);

\draw (Ob) ++ (x1) node [left] {$\bx^i_1$}
-- ++ (x12) node [right] {$\bx^i_2$}
-- ++ (x23) node [above] {$\bx^i_3$}
-- ++ (x34) node [left] {$\bx^i_4$}
-- cycle;
\fill (Ob)++(x5) circle (1.5pt) node[below] {$\bx^i_c$};
\fill (Ob)++(x1) circle (1.5pt)
++(x12) circle (1.5pt)
++(x23) circle (1.5pt)
++(x34) circle (1.5pt);
\end{tikzpicture}
\caption{The transformation between reference cell $\Omega$ and computational cell $\omega^i$}
\label{fig_Ci_quad}
\end{figure}
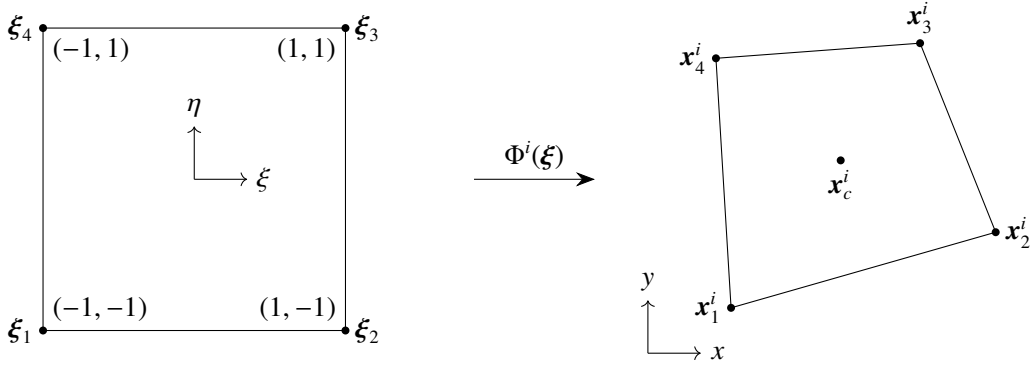
}

To describe the physical field accurately, two kind of discretization moments are defined in our method, i.e. the volume-integrated average (VIA) and point-value (PV) at the vertices for conserved variables $\U$, 

\begin{eqnarray*}
\begin{aligned}
& \text{VIA}:  \U_\avg^i = \frac{1}{|\omega^i|}\int_{\omega^i} \U(\bx)\dd\omega= \frac{1}{|\Omega|}\int_{\Omega} \U^i(\bxi)J \dd \Omega, \label{con1} \\
&  \text{PV} :  \U^i_r = \U(\bx^i_r)=\U^i(\bxi_r), \; r = 1,2,3,4. \label{con2}
\end{aligned}
\end{eqnarray*}
where $J$ is the corresponding Jacobian determinant  $\mathrm{Det}({\bf J})=\mathrm{Det}(\frac{\partial\bx}{\partial\bxi} )$. 

\subsection{High order reconstructions} \label{chap:Lag:LMCV:space}
The reconstruction function \cite{Xie-2DMCV} of $\U^i$ on reference cell $\Omega$ is 
\begin{eqnarray}\label{cell^intp}
\U^i(\bxi) = \psi_{c} \U^i_c + \sum_{r=1}^{4} \psi_{r}\U^i_r 
+ \bra{\psi_{\xi} \U_\xi^i + \psi_{\eta} \U_\eta^i}
+ \bra{\psi_{\xi^2}\U_{\xi^2}^i + \psi_{\eta^2}\U_{\eta^2}^i},
\end{eqnarray}
where $\psi$ are scalar basis functions as below:
\begin{eqnarray*}
\begin{cases}
\psi_r=
\frac{1}{4}\xi\eta\bra{\xi+\xi_r}\bra{\eta+\eta_r}, \quad r=1,2,3,4,
                            &  \\
\psi_c=1-\xi^2\eta^2,                   &            \\
(\psi_{\xi},\psi_{\eta}) = \bra{\xi \bra{1-\eta^2},\eta \bra{1-\xi^2}}, & \\
(\psi_{\xi^2},\psi_{\eta^2}) = \bra{\hf \xi^2 \bra{1-\eta^2}, \hf \eta^2 \bra{1-\xi^2}}, & 
\end{cases}
\end{eqnarray*}
$\U^i_c$, $(\U_\xi^i,\U_\eta^i)$, and $(\U_{\xi^2}^i,\U_{\eta^2}^i)$ are the value, first and second-order derivatives of $\U^i$ at the center of $\Omega$, which are obtained by interpolation and reconstruction as follows. 


By appling the 5-point integral formula 
\begin{eqnarray}
\U_\avg^i = \frac{1}{12 J^i_c}\bra{\sum_{r=1}^{4}J^i_r \U^i_r + 8 J^i_c \U^i_c},
\end{eqnarray}
on the reference cell $\Omega$, we can get the interpolation of $\U^i_c$ as 
\begin{eqnarray}
\U^i_c = \frac{1}{8 J^i_c}\bra{\displaystyle 12 J^i_c \U_\avg^i - \sum_{r=1}^{4}J^i_r\U^i_r},
\end{eqnarray}
where $J^i_r = \mathrm{Det}\bra{\frac{\partial \Phi^i}{\partial \bxi}(\bxi_r)}$ and $J^i_c = \mathrm{Det}\bra{\frac{\partial \Phi^i}{\partial \bxi}(\boldsymbol{0})}$ are determinants of $\Phi^i$. 

The derivatives $(\U_\xi^i,\U_\eta^i)$ and $(\U_{\xi^2}^i,\U_{\eta^2}^i)$ are reconstructed by $\omega^i$ and adjacent cells. 
Firstly, linear reconstruction on each cell surface is performed at computational domain. Take surface $\overline{\bx^i_r\bx^i_{r+1}} = \omega^i \cap \omega^j$ (see Fig. \ref{fig_Eij_recons}) for example, a linear least-square problem 
$$
\min_{\U^i_{r,r+1}(\bx)} |\U^i_{r,r+1}(\bx^i_c) - \U^i_c|^2 +|\U^i_{r,r+1}(\bx^j_c) - \U^j_c|^2 +|\U^i_{r,r+1}(\bx^i_r) - \U^i_r|^2 +|\U^i_{r,r+1}(\bx^i_{r+1}) - \U^i_{r+1}|^2 
$$
is solved, 
where $\U^i_{r,r+1}(\bx)$ is the local linear function defined on computational domain corresponding to surface $\overline{\bx^i_r\bx^i_{r+1}}$. By the transformation $\Phi^i(\bxi)$, $\U^i_{r,r+1}(\bxi) \coloneqq \U^i_{r,r+1}(\Phi^i(\bxi))$ is defined on reference cell $\Omega$, and 
\begin{eqnarray*}
(\U^i_{\xi r},\U^i_{\eta r}) \coloneqq \nabla_{\bxi}\U^i_{r,r+1}\bra{\hf\bxi_r+\hf\bxi_{r+1}}. 
\end{eqnarray*}
The derivatives at cell center is formulated as
\begin{eqnarray}
\label{cell_center_diff}
\begin{cases}\displaystyle
(\U_\xi^i,\U_\eta^i) =
\frac{1}{4} \sum_{r=1}^{4} (\U^i_{\xi r},\U^i_{\eta r}),
\\
\displaystyle
(\U_{\xi^2}^i,\U_{\eta^2}^i) =
\hf (\U^i_{\xi 2} - \U^i_{\xi 4},
\U^i_{\eta 3}-\U^i_{\eta 1}).
\end{cases}
\end{eqnarray}

Without causing ambiguity, we use $\U^i(x)$ to represent the reconstructed physical field on $\omega^i$ as $\U^i(\bx) = \U^i\circ(\Phi^i)^{-1}(\bx)$.

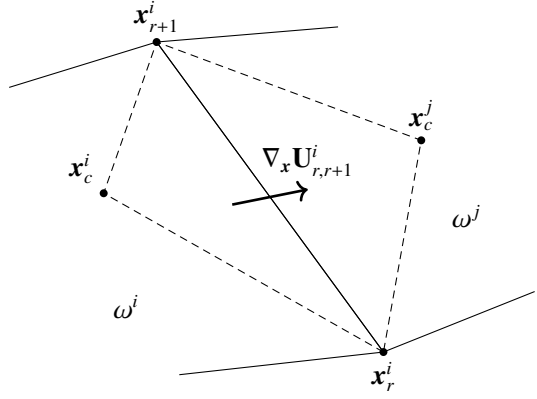
\begin{figure}[!htp]
\centering
\begin{tikzpicture}[scale=1]
\coordinate (Oa) at (0,0);
\coordinate (Ob) at (6,-2.3);

\coordinate (xi) at (-0.5,0.5);
\coordinate (xj) at (3.7,1.2);
\coordinate (xij1) at (3.2,-1.6);
\coordinate (dxi1) at (-2.7, -0.3);
\coordinate (dxi2) at (-1.95,-0.6);
\coordinate (xij2) at (0.2,2.5);
\coordinate (dxj1) at (2,0.8);
\coordinate (dxj2) at (2.4,0.2);

\draw [densely dashed] (xij1) -- (xi) -- (xij2) -- (xj) -- cycle;

\draw (xij1) ++ (dxi1) -- (xij1) -- (xij2) -- ++(dxi2);
\draw (xij1) ++ (dxj1) -- (xij1) -- (xij2) -- ++(dxj2);

\draw (xi) ++ (0.3, -1.5) node {$\omega^i$};
\draw (xj) ++ (0.6, -1) node {$\omega^j$};

\fill (xi) circle (1.5pt) node [above left] {$\bx^i_c$}
(xj) circle (1.5pt) node [above] {$\bx^j_c$}
(xij1) circle (1.5pt) node [below] {$\bx^i_r$}
(xij2) circle (1.5pt) node [above] {$\bx^i_{r+1}$};


\coordinate (xijc) at ($(xij1)!0.5!(xij2)$);
\coordinate (dphi) at (0.5,0.1);

\draw [line width =1pt,->] ($(xijc) - (dphi)$)
-- ($(xijc) + (dphi)$) node [above] {$\nabla_{\bx}\U^i_{r,r+1}$};
\end{tikzpicture}
\caption{gradient reconstruction on cell surface $\overline{\bx^i_r\bx^i_{r+1}}$}
\label{fig_Eij_recons}
\end{figure}

\subsection{Evolution of VIA moment and vertices} \label{chap:Lag:LMCV:VIAflux}
Similar to the EUCCLHYD \cite{Maire-EUCCLHYD}, (\ref{lag_vel^int}-\ref{lag_vol^int}) is discretized as 
\begin{eqnarray}
m^i\bra{\frac{1}{\bar\rho^i}}_t &=& \hf\sum_{r=1}^{4} L^i_{r,r+1}
\bra{\tbVi_{r}+\tbVi_{r+1}} \cdot \bN^i_{r,r+1},
\label{cell_evolve_vol} \\
m^i\bra{\bar{\bV}^i}_t &=&
-\hf\sum_{r=1}^{4} L^i_{r,r+1}
\bra{\tPsi_{r,r+\hf} + \tPsi_{r+\hf,r+1}}
\bN^i_{r,r+1},
\label{cell_evolve_vel} \\
m^i\bra{\bar{E}^i}_t &=&
-\hf\sum_{r=1}^{4} L^i_{r,r+1}
\bra{\tPsi_{r,r+\hf}\tbVi_{r}
+ \tPsi_{r+\hf,r+1}\tbVi_{r+1}}\cdot \bN^i_{r,r+1},
\label{cell_evolve_E}
\end{eqnarray}
where $m^i$ is the constant mass of cell $\omega^i$ according to \eqref{lag_rho}. 
$\bar\rho^i, \bar{\bV}^i, \bar{E}^i$ are obtained by 
\begin{eqnarray*}
\U_\avg^i = \begin{bmatrix}
\bar\rho^i\\
\bar\rho^i\bar\bV^i\\
\bar\rho^i\bar E^i
\end{bmatrix}.
\end{eqnarray*}
$L^i_{r,r+1}$ and $\bN^i_{r,r+1}$ are the length and the unit outward normal vector of surface $\overline{\bx^i_r \bx_{r+1}^i}$. 
Meanwhile, vertices $\bx^i_r$ are moving with local fluid  speed $\tbVi_r$
\begin{eqnarray}
\bra{\bx^i_r}_t = \tbVi_r, \; r=1,2,3,4,
\end{eqnarray}
which is compatible with \eqref{cell_evolve_vol} since the cell area $|\omega^i| = \frac{m^i}{\bar{\rho}^i}$. 
The nodal velocity $\tbVs$ and surface pressure $\tPs$ are solved by an augmented  nodal solver at Section \ref{chap:Lag:Nodal}.

\subsection{Evolution of PV moment} \label{chap:Lag:LMCV:PVflux}
The PV moment at $\bx^i_r$ is updated by solving \eqref{2DEuler} in differential form as 
\begin{eqnarray} \label{node_evolve_1}
\bra{\U^i_r}_t
+ \bra{\bF^i_r}_x - \tVi_{x,r} \bra{\U^i_r}_x
+ \bra{\bG^i_r}_y - \tVi_{y,r} \bra{\U^i_r}_y
= \boldsymbol{0},
\end{eqnarray}
where $[\bF^i_r,\bG^i_r] = \mathbb{F}(\U^i_r) $ and 
$\tbVi_{r} = (\tVi_{x,r},\tVi_{y,r})$. 
With Jacobian matrix of the flux defined by 
$\mathbb{A}  = \frac{\partial \bF}{\partial \U}(\U^i_r) - \tVi_{x,r}\mathbb{I}_4$, 
$\mathbb{B} = \frac{\partial \bG}{\partial \U}(\U^i_r) - \tVi_{y,r}\mathbb{I}_4$, 
\eqref{node_evolve_1} can be locally linearized as 
\begin{eqnarray} \label{node_evolve_linear}
\bra{\U^i_r}_t
+ \mathbb{A}\bra{\U^i_r}_x
+ \mathbb{B}\bra{\U^i_r}_y
= \boldsymbol{0},
\end{eqnarray}
which is discretized by local Lax-Friedrichs flux as 
\begin{eqnarray} \label{node_evolve_numerical}
\bra{\U^i_r}_t
+ \hf\brb{\mathbb{A}\hat{\U}_x(\bx^i_r)
- S_{\mathbb{A}} \bra{\hat{\U}_{x,R}(\bx^i_r)-\hat{\U}_{x,L}(\bx^i_r)}}
+ \hf\brb{\mathbb{B}\hat{\U}_y(\bx^i_r)
- S_{\mathbb{B}} \bra{\hat{\U}_{y,R}(\bx^i_r)-\hat{\U}_{y,L}(\bx^i_r)}}
= \boldsymbol{0},
\end{eqnarray}
where $S_\mathbb{K}$ is maximum norm of eigenvalues of matrix $\mathbb{K}$, approximated derivatives $\hat{\U}_x, \hat{\U}_y$ and one-sided derivatives $\hat{\U}_{x,L}, \hat{\U}_{x,R}, \hat{\U}_{y,L}, \hat{\U}_{y,R}$ at $\bx^i_r$ are reconstructed as {\bf Appendix \hyperlink{Apd.PV}{C} } shown. 

In summary, LMCV defines four point-values (PVs) $\{\U^i_r\; |\; r=1,2,3,4 \}$ at four vertices and one volume-integrated average (VIA) $\U^i_{\avg}$ for each cell $\omega^i$, which are evolved separately by (\ref{cell_evolve_vol}-\ref{cell_evolve_E}) and \eqref{node_evolve_numerical}. Conservative variable $\U$ is reconstructed on cell $\omega^i$ as $\U^i(\bxi)$ by \eqref{cell^intp}, based on the PVs and VIA of adjacent cells.

For each cell $\omega^i$, apart from geometry information, our scheme  requires values of $\U$ at each vertices $\bx^i_r$ and each middle points $\bx^i_{r+\hf} \coloneqq \hf\bx^i_r + \hf\bx^i_{r+1}$ as input, denoted as $\U^i_r$ and $\U^i_{r+\hf}$, for $r=1,2,3,4$. In LMCV, these values are computed by reconstruction function $\U^i(\bxi)$, that is 
$$
\U^i_r = \U^i(\bxi^i_r), \quad \U^i_{r+\hf} = \U^i\bra{\hf\bxi^i_r+\hf\bxi^i_{r+1}}.
$$
We denote the velocity, pressure and acoustic impedance of $\U^i_r$ as 
$\bVi_r$, $P^i_r$ and $\alpha^i_r$. In the same way, $\bV^i_{r+\hf}$, $P^i_{r+\hf}$ and $\alpha^i_{r+\hf}$ are defined for $\U^i_{r+\hf}$. For any other cell-centered Lagrangian schemes, our scheme  is available as long as $\U^i_r$ and $\U^i_{r+\hf}$ can be provided, whether by reconstruction, interpolation or any method else.

\section{Riemann solvers and high-order numerical fluxes} \label{Chap:Lag:Revisit}

\subsection{The HLLC approximate solver}  \label{Chap:Lag:Revisit:1D}
It is well known that the numerical flux in (\ref{cell_evolve_vol}-\ref{cell_evolve_E}) can be obtained from many methods. Most of them is based on the solution of extended 1D Riemann problem along the outward normal direction of cell edges. Let $(\bN,\bT)$ be the unit normal and tangent of edge, then define $(u , v)= (\bV\cdot \bN,  \bV\cdot \bT)$. 
A Riemann problem is set
\begin{eqnarray} \label{2.1}
{\bf U}_{t}+{\bf F}_{x} = 0,
\end{eqnarray}
along $x$-axis with initial values
\begin{eqnarray*}
{\bf U}(x,0)  = \left\{\begin{aligned}
{\bf U}_L = (\rho_L,\rho_L u_L,\rho_L v_L,\rho_L E_L)^\top,\quad x<0,\\
{\bf U}_R = (\rho_R,\rho_R u_R,\rho_R v_R,\rho_R E_R)^\top,\quad x>0.
\end{aligned}
\right.
\end{eqnarray*}
The widely used HLLC solver approximates the Riemann solution by four states as 
\begin{eqnarray}\label{rup_HLLC}
(\U,\bF)(x,t) = \left\{
\begin{array}{lcl}
(\U_L,\bF_L),          &      & {\mbox {if } \ x/t     \leq  S_L},\\
(\tilde{\U}_L,\tbF_L),     &      & {\mbox {if } \  S_L < x/t \leq  \ustar},\\
(\tilde{\U}_R,\tbF_R),     &      & {\mbox {if } \  \ustar < x/t \leq S_R },\\
(\U_R,\bF_R),     &      & {\mbox {if } \ S_R <x/t}.
\end{array} \right.
\end{eqnarray}
with 
\begin{gather*}
\tilde{\U}_K = (\tilde{\rho}_K,\tilde{\rho}_K\ustar,\tilde{\rho}_K v_K,\tilde{\rho}_K \widetilde{E}_K)^\top, \quad 
\tbF_K = \ustar\tilde{\U}_K+(0,\tP,0,\ustar\tP)^\top \quad K=L,R.
\end{gather*}
where $S_L,S_R$ are approximated wave speed \cite{Batten-choice}. 
Based on the Rankine–Hugoniot condition 
\begin{eqnarray}\label{rup_HLLC_rh}
\bF_K - \tbF_K = S_K(\U_K - \tilde{\U}_K).
\end{eqnarray}
Let $\alpha_K = \rho_K c_K$ be the acoustic impedance ($c_L = u_L-S_L,c_R = S_R-u_R $), the velocity and pressure of the contact discontinuity are

\begin{eqnarray}
\label{contavepre1}
\begin{cases} \displaystyle
\ustar = \frac{\alpha_Lu_L+\alpha_Ru_R + P_L -P_R}{\alpha_L + \alpha_R}, \\
\tP = P_L-\alpha_L(\ustar-u_L)  .
\end{cases}
\end{eqnarray}


The rest unknown quantities in $(\tilde{\U}_K,\tbF_K) $ cab be obtained by Eqs. (\ref{rup_HLLC_rh}) and Eqs. (\ref{contavepre1}).



\subsection{Revisit a 2D nodal solver}\label{Chap:Lag:Revisit:Maire}
For comparison purpose, the  original nodal solver from Maire et al. \cite{Maire-EUCCLHYD} is revisited. For any half surface between two cells $e = \overline{\bx^i_r\bx^i_{r+\hf}} = \overline{\bx^j_{k-\hf}\bx^j_{k}} \subset \omega^i\cap \omega^j$, the numerical pressures $\tPs$ on both sides are denoted as $\tPs_{e+} = \tPsi_{r,r+\hf}$ and $\tPs_{e-} = \tPsj_{k-\hf,k}$, and $L_e$ and $\bN_e$ are the length and the unit outward normal vector of half surface $e$, the normal direction is relative to $\omega^i$ (see Fig. \ref{fig_NodalSolver}). 

\noindent
\textbf{Jump condition: } The pressure on each half surface of $\omega^i$ following Rankine–Hugoniot (RH) condition
\begin{eqnarray}\label{maire_pres}
\left\{
\begin{aligned}
& \tPs_{e+} =\tP^i_{r,r+\hf} = P^{i}_{r} 
- \alpha^{i}_{r} \bra{\tbV^i_r - \bV^{i}_{r}}\cdot \bN^i_{r,r+1}
=P_{e+} - \alpha_{e+}\bra{\tbV^i_r - \bV^{i}_{r}}\cdot \bN_{e},
\\
& \tPs_{e-} =\tP^j_{k-\hf,k} = P^{j}_{k} 
- \alpha^{j}_{k} \bra{\tbV^j_k - \bV^{j}_{k}}\cdot \bN^j_{k-1,k}=
P_{e-} + \alpha_{e-}\bra{\tbV^j_k - \bV^{j}_{k}}\cdot \bN_{e}, 
\end{aligned}
\right. 
\end{eqnarray}
where
\begin{gather*}
P_{e+} = P^i_r,\quad P_{e-} = P^j_k,\quad
\alpha_{e+} = \alpha^i_r, \quad \alpha_{e-} = \alpha^j_k.
\end{gather*}

\noindent
\textbf{Balance condition: }
The balance of the pressures  is expressed as
\begin{eqnarray}
\sum_{e\in \cE(\bx^i_r) }L_e\bra{\tPs_{e+}-
\tPs_{e-}}\bN_e=\boldsymbol{0},\label{maire_balance}
\end{eqnarray}
where $\cE(\bx^i_r)$ contains all half surface with $\bx^i_r$ as a vertex. 

Substitute Eqs.\eqref{maire_pres} into Eqs.\eqref{maire_balance}, $\tbVi_r$ is acquired by solving the linear system 
\begin{eqnarray} \label{maire_vel}
\sum_{e\in \cE(\bx^i_r)} L_e
\bra{{\alpha}_{e+}+{\alpha}_{e-}}
\bra{\tbVi_r\cdot\bN_{e}-\cVs_{e}}\bN_{e} = \boldsymbol{0},
\end{eqnarray}
where
\begin{gather} \label{normal_vel}
\cVs_{e} = \frac{
{\alpha}_{e+} \bV^i_r\cdot\bN_{e}
+ {\alpha}_{e-} \bV^j_k\cdot\bN_{e}
+ P_{e+} - P_{e-}
}
{{\alpha}_{e+} + {\alpha}_{e-}},
\end{gather}
which is the normal velocity given by classical 1D acoustic Riemann solver at the endpoint $\bx^i_r$ of edge $\overline{\bx^i_r\bx^i_{r+1}}$. Meanwhile, Eqs.\eqref{maire_vel} can be interpreted as a least square problem
\begin{eqnarray} \label{maire_vel_lsq}
\tbVi_{r} = \underset{\bV\in\mathbb{R}^2}{\arg \min}\sum_{e\in \cE(\bx^i_r)} L_e
\bra{{\alpha}_{e+}+{\alpha}_{e-}}
\bra{\bV\cdot\bN_{e}-\cVs_{e}}^2.
\end{eqnarray}

Noticing that this solver gives $\tbVs^i_r = \bV^i_r$ with smooth physical field, which indicates that the numerical flux of each cell cannot exceed second-order accuracy for straight-edge mesh, as discussed in {\bf Appendix \hyperlink{Apd.2rd}{A} }. 

\subsection{Incompatibility between high-order numerical fluxes and mesh movement}   
In this subsection, we consider the fundamental fluxes and its numerical integrations in cell-centered Lagrangian schemes. 
Consider, for example, any $\omega^i$，
\begin{eqnarray} \label{lag^int}
\frac{\dd}{\dd t} \bU_\avg^i + \frac{1}{m^i}\int_{\partial\omega^i} \bar{\mathbb{F}} \bN \dd l = \boldsymbol{0},
\end{eqnarray}
where 
\begin{eqnarray*}
\bU = \begin{bmatrix}
1/\rho \\ \bV \\ E
\end{bmatrix}, \quad 
\bar{\mathbb{F}} = \begin{bmatrix}
-\bV^\top \\ P \mathbb{I}_2 \\ P\bV^\top
\end{bmatrix}.
\end{eqnarray*}
and $\bU_\avg^i$ reads as the mass averaged value of $\bU$ over $\omega^i$ as 
$\frac{1}{m^i}\int_{\omega^i} \rho \bU \dd \omega$. 
A variety of cell-centered Lagrangian schemes take their starting point from discretizing \eqref{lag^int} and  choosing different the numerical flux $\bar{\mathbb{F}}$ and the treatment of the boundary integral $\int_{\partial\omega^i}$. Distinct from the Euler and ALE methods, the precision of the integral cannot serve as the only factor used to evaluate the correctness of Eqn. \eqref{lag^int} .
The specific volume in the Lagrangian framework depends on  rate of change of each cell area, which means high-accuracy Lagrangian methods also require a more elaborate representation of $\partial\omega^i$.
In general, a standard procedure define cell boundaries by straight line edges, whereas some numerical schemes adopt curvilinear edges to elevate accuracy, which consequently means that extra shape control points are introduced for each edge (see Fig. \ref{Fig.ploycell}). 

{\color{black}
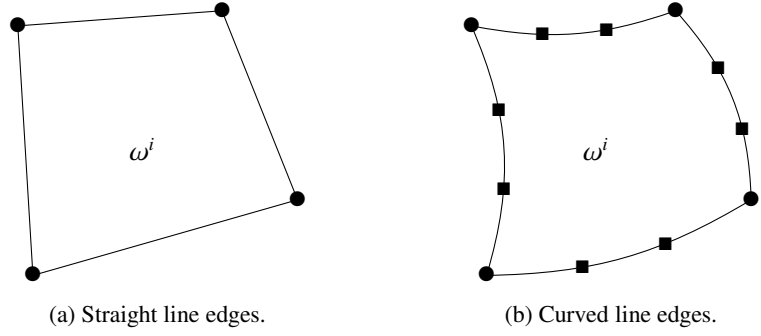
\begin{figure}[!htp]
\centering
\begin{tikzpicture}[scale=1]
\coordinate (Oa) at (0,0);
\coordinate (Ob) at (6,0);

\coordinate (x1) at (1.1,0.6);
\coordinate (x12) at (3.5,1);
\coordinate (x23) at (-1,2.5);
\coordinate (x34) at (-2.7,-0.2);
\coordinate (x5) at (
$(x1)+0.75*(x12)+0.5*(x23)+0.25*(x34)$);

\draw (Oa) ++ (x1) node {\Large$\bullet$}
-- ++ (x12) node {\Large$\bullet$}
-- ++ (x23) node {\Large$\bullet$}
-- ++ (x34) node {\Large$\bullet$}
-- cycle;


\draw (Oa)++(x5) node [below] {$\omega^i$};

\draw (Oa)++(x5)++(0.2,-2.5) node {\small(a) Straight line edges.};

\draw[-] (Ob)++(x1) to[bend right=15] 
node[pos=1/3] {\tiny $\blacksquare$}
node[pos=2/3] {\tiny $\blacksquare$}
++(x12) node {\Large$\bullet$}
to[bend right=20] 
node[pos=1/3] {\tiny $\blacksquare$}
node[pos=2/3] {\tiny $\blacksquare$}
++ (x23) node {\Large$\bullet$}
to[bend left=15]
node[pos=1/3] {\tiny $\blacksquare$}
node[pos=2/3] {\tiny $\blacksquare$}
++ (x34) node {\Large$\bullet$}
to[bend left=20] 
node[pos=1/3] {\tiny $\blacksquare$}
node[pos=2/3] {\tiny $\blacksquare$}
cycle node {\Large$\bullet$}; 

\draw (Ob)++(x5) node [below] {$\omega^i$};

\draw (Ob)++(x5)++(0.2,-2.5) node {\small(b) Curved line edges.};

\end{tikzpicture}
\caption{General cells.}
\label{Fig.ploycell}
\end{figure}
}

By means of quadrature rules, the cell boundary integral is usually discretized as 
\begin{eqnarray} \label{lag^int_2}
\frac{\dd}{\dd t}\int_{\omega^i} {\bU} \dd \omega + \frac{1}{m^i}\sum_{q\in \mathcal{Q}^i} l_q \bar{\mathbb{F}}_q \bN_q = \boldsymbol{0},
\end{eqnarray}
where $\mathcal{Q}^i$ is the integral node set of cell $\omega^i$, $\bN_q$ and $l_q$ are the unit normal vector and weight length at node $q$, $\bar{\mathbb{F}}_q$ is the numerical flux computed at $q$ with certain Riemann solver. Moreover, the shape control points of $\omega^i$ should be advected through 
\begin{eqnarray}
\frac{\dd}{\dd t}\bx_q = \bV_q,\quad \forall q\in \mathcal{P}^i,
\end{eqnarray}
where $\mathcal{P}^i$ is the shape control point set of $\omega^i$. Without loss of generality, it is assumed that $\mathcal{P}^i\subset \mathcal{Q}^i$ so that $\bV_q$ can be acquired by the Riemann solver directly or indirectly. As shown in Fig. \ref{Fig.cpoint}, differences in local mesh topology of $\bx_q$ make traditional 1D Riemann solvers inadequate for all cases. 

\begin{figure}[!htp]
\centering
\begin{tikzpicture}[scale=1]
\coordinate (Oa) at (0,0);
\coordinate (Ob) at (6,0);

\coordinate (x1) at (0,0);
\coordinate (x2) at (3.5,1);
\coordinate (x1a) at (-1,-0.2);
\coordinate (x1b) at (-0.3,-1);
\coordinate (x1c) at (-0.5,1.2);
\coordinate (x2a) at (1,-0.7);
\coordinate (x2b) at (0,-1.5);
\coordinate (x2c) at (-0.5,1);

\draw (x1) to [bend left=35] ++(x1a);
\draw (x1) to [bend left=20] ++(x1b);
\draw (x1) to [bend right=20] ++(x1c);

\draw (x2) to [bend left=15] ++(x2a);
\draw (x2) to [bend left=15] ++(x2b);
\draw (x2) to [bend right=20] ++(x2c);

\node[above] at ($0.5*(x1)+0.5*(x1b)+0.5*(x2)+0.5*(x2b)$) {$\omega^i$};
\node at ($0.5*(x1)+0.5*(x1c)+0.5*(x2)+0.5*(x2c)$) {$\omega^j$};

\draw (x1) 
node {\Large$\bullet$}
to[bend left=30] 
node[pos=1/3] {\color{red}\tiny$\blacksquare$}
node [pos=1/3,below] {$\bx_q$}
node[pos=2/3] {\tiny$\blacksquare$}
++(x2) 
node {\Large$\bullet$};

\coordinate (x3) at (8,0.5);
\coordinate (x3a) at (-1.5,-0.2);
\coordinate (x3b) at (1.2,-0.6);
\coordinate (x3c) at (-0.2,1.3);
\coordinate (x3d) at (-0.4,-1.2);

\draw (x3) to [bend right = 20] 
node[pos=2/3] {\tiny $\blacksquare$}
++(x3a);
\draw (x3) to [bend left = 15] 
node[pos=2/3] {\tiny $\blacksquare$}
++(x3b);
\draw (x3) to [bend right = 15] 
node[pos=2/3] {\tiny $\blacksquare$}
++(x3c);
\draw (x3) to [bend left = 15] 
node[pos=2/3] {\tiny $\blacksquare$}
++(x3d);

\node at (x3) {\color{red}\Large$\bullet$};
\node[above right] at (x3) {$\bx_q$};
\node at ($(x3)+0.7*(x3b)+0.7*(x3d)$) {$\omega^i$};
\node at ($(x3)+0.7*(x3a)+0.7*(x3c)$) {$\omega^k$};
\node at ($(x3)+0.7*(x3a)+0.7*(x3d)$) {$\omega^l$};
\node at ($(x3)+1*(x3c)+1*(x3b)$) {$\omega^j$};

\node at ($0.5*(x1)+0.5*(x2)+(0,-2)$) {\small(a) Face point neighboring cells};
\node at ($(x3)+(0,-2)$) {\small(b) Node neighboring cells};

\end{tikzpicture}
\caption{Control points with neighboring cells.}
\label{Fig.cpoint}
\end{figure}
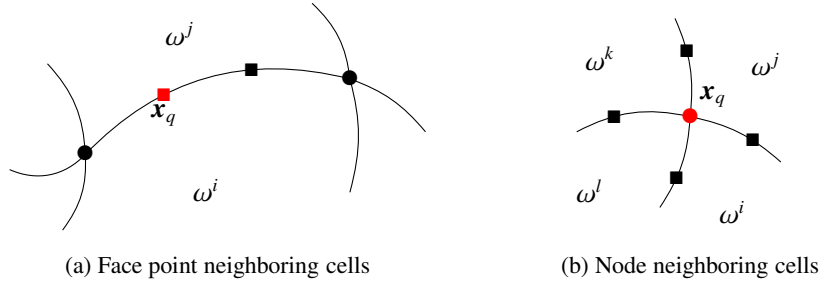

In order to give consistent numerical fluxes and velocities at the cell vertices without losing conservation or other properties, various multidimensional Riemann solvers \cite{Despres-GLACE,Maire-EUCCLHYD,Maire-HighOrder} have been developed for Lagrangian schemes, which directly handle the calculation of nodal velocities on straight-edge meshes. However, complexities arise when addressing curvilinear meshes, where velocities need be solved at both corners and edges.

Throughout schemes \cite{Cheng-3rd, Liu-subcellLDG, Morgan-LDG,Morgan-LDG-2,Vilar-LDG} mentioned in Section \ref{chap:Lag:Intro}, high-order accuracy is achieved by introducing curvilinear mesh to basic schemes \cite{Despres-GLACE,Maire-EUCCLHYD,Maire-HighOrder}, while robustness is secured by post-processing or pre-processing techniques. Nonetheless, in the following, we propose an interesting approach to obtain high accuracy directly through pre-processing technique instead of introducing curvilinear meshes, which also avoids the destruction of robustness. 
\subsection{A new path constructing nodal solver}\label{Chap:Lag:Revisit:ANS}
In the past, on the surface $\overline{\bx^i_r\bx^i_{r+1}}$,
most high-order Lagrangian scheme noticed that only two solved Riemann problem at both ends $\bx^i_r$, $\bx^i_{r+1}$ are not enough to provide enough accuracy. For third-order schemes \cite{Vilar-LDG,Morgan-LDG,Liu-subcellLDG}, an extra Riemann problem is usually introduced at middle point $\bx^i_{r+\hf}$,  but instead of being solved by the standard 1D HLLC solver \cite{Vilar-LDG,Morgan-LDG} or the 2D nodal solver \cite{Liu-subcellLDG}, it is allowed for pressure discontinuities, just like the Riemann solution at the endpoints (see Fig. \ref{fig_pABC}), defined in consistent form as 
\begin{eqnarray}
\left \{ 
\begin{aligned}
    &\tP^i_r =  P^{i}_{r} - \alpha^{i}_{r} \bra{\tbVi_r - \bV^{i}_{r}}\cdot \bN^i_{r,r+1}, \\
    &\tP^i_{r+\hf} = P^{i}_{r+\hf} - \alpha^{i}_{r+\hf} \bra{\hf\tbVi_r+\hf\tbVi_{r+1} - \bV^{i}_{r+\hf}}\cdot \bN^i_{r,r+1},\\
    &\tP^i_{r+1} =  P^{i}_{r+1} - \alpha^{i}_{r+1} \bra{\tbVi_r - \bV^{i}_{r+1}}\cdot \bN^i_{r,r+1}.
\end{aligned}
\right.
\end{eqnarray}

{\color{black}
\begin{figure}[!htp]
\centering
\begin{subfigure}{0.32\linewidth}
\centering
\begin{tikzpicture}[scale=1]
\coordinate (X0) at (0,0);
\coordinate (X1) at (0,3);
\coordinate (XL) at (-1.3,1.5);
\coordinate (XR) at (1.3,1.5);
\coordinate (X0L) at (-1.5,-0.5);
\coordinate (X0R) at (1.5,-0.5);
\coordinate (X1L) at (-1.5,3.5);
\coordinate (X1R) at (1.5,3.5);

\draw (X0) node[yshift=-20] {$\bx^i_r=\bx^j_k$}
-- (X1) node[yshift=20] {$\bx^i_{r+1} = \bx^j_{k-1}$};
\draw (X0L) -- (X0) -- (X0R);
\draw (X1L) -- (X1) -- (X1R);

\draw (XL) node {$\omega^i$};
\draw (XR) node {$\omega^j$};

\fill (X0) circle [radius=0.07]
(X1) circle [radius=0.07];

\coordinate (n) at ($(X0)!1cm!-90:(X1)-(X0)$);
\coordinate (dn0) at ($0.1*(n)$);
\draw ($(X0)!0.03!(X1) + (dn0)$)
--    ($(X0)!0.48!(X1) + (dn0)$);
\node[right] at ($(X0)!0.25!(X1) + (dn0)$)
{$\tP^j_{k}$};
\draw ($(X0)!0.03!(X1) - (dn0)$)
--    ($(X0)!0.48!(X1) - (dn0)$);
\node[left] at ($(X0)!0.25!(X1) - (dn0)$)
{$\tP^i_{r}$};

\draw ($(X0)!0.52!(X1) + (dn0)$)
--    ($(X0)!0.97!(X1) + (dn0)$);
\node[right] at ($(X0)!0.75!(X1) + (dn0)$)
{$\tP^j_{k-1}$};
\draw ($(X0)!0.52!(X1) - (dn0)$)
--    ($(X0)!0.97!(X1) - (dn0)$);
\node[left] at ($(X0)!0.75!(X1) - (dn0)$)
{$\tP^i_{r+1}$};

\draw (0,-1.5) node {\small $\tP^i_{r} \ne \tP^j_k$, $\tP^i_{r+1} \ne \tP^j_{k-1}$};
\end{tikzpicture}
\caption{Schemes using nodal solver from Maire et al.}\label{fig_pA}
\end{subfigure}
\begin{subfigure}{0.32\linewidth}
\centering
\begin{tikzpicture}[scale=1]
\coordinate (X0) at (0,0);
\coordinate (X1) at (0,3);
\coordinate (Xa) at (0.1,0.5);
\coordinate (Xb) at (0.1,2.5);
\coordinate (XL) at (-1.1,1.5);
\coordinate (XR) at (1.5,1.5);
\coordinate (X0L) at (-1.5,-0.5);
\coordinate (X0R) at (1.5,-0.5);
\coordinate (X1L) at (-1.5,3.5);
\coordinate (X1R) at (1.5,3.5);

\draw (X0) -- (Xa);
\draw (X1) -- (Xb);
\node[yshift=-20] at (X0) {$\bx^i_r=\bx^j_k$};
\node[yshift=20] at (X1) {$\bx^i_{r+1} = \bx^j_{k-1}$};

\draw (X0L) to [bend left=8] (X0);
\draw (X0) to [bend left=12] (X0R);
\draw (X1L) to [bend left=-10] (X1);
\draw (X1) to [bend left=-10] (X1R);

\draw (XL) node {$\omega^i$};
\draw (XR) node {$\omega^j$};

\fill (X0) circle [radius=0.07]
(X1) circle [radius=0.07];

\coordinate (n) at ($(X0)!1cm!-90:(X1)-(X0)$);
\coordinate (dn0) at ($0.1*(n)$);
\draw ($(X0)!0.1!(Xa) + (dn0)$)
--    ($(X0)!0.9!(Xa) + (dn0)$);
\node[right] at ($(X0)!0.5!(Xa) + (dn0)$)
{$\tP^j_{k}$};
\draw ($(X0)!0.1!(Xa) - (dn0)$)
--    ($(X0)!0.9!(Xa) - (dn0)$);
\node[left] at ($(X0)!0.5!(Xa) - (dn0)$)
{$\tP^i_{r}$};

\draw [line width = 2] (Xa) to [bend right = 10] (Xb);
\node at ($(XL)!0.72!(XR)$)
{$\tP^j_{k-\hf}$};
\node at ($(XL)!0.28!(XR)$)
{$\tP^i_{r+\hf}$};

\draw ($(Xb)!0.1!(X1) + (dn0)$)
--    ($(Xb)!0.9!(X1) + (dn0)$);
\node[right] at ($(Xb)!0.5!(X1) + (dn0)$)
{$\tP^j_{k-1}$};
\draw ($(Xb)!0.1!(X1) - (dn0)$)
--    ($(Xb)!0.9!(X1) - (dn0)$);
\node[left] at ($(Xb)!0.5!(X1) - (dn0)$)
{$\tP^i_{r+1}$};

\draw (0,-1.5) node {\small $\tP^i_{r} \ne \tP^j_k$, $\tP^i_{r+\hf}=\tP^j_{k-\hf}$, $\tP^i_{r+1} \ne \tP^j_{k-1}$};
\end{tikzpicture}
\caption{Schemes using mixed solvers}\label{fig_pB}
\end{subfigure}
\begin{subfigure}{0.32\linewidth}
\centering
\begin{tikzpicture}[scale=1]
\coordinate (X0) at (0,0);
\coordinate (X1) at (0,3);
\coordinate (XL) at (-1.3,1.5);
\coordinate (XR) at (1.3,1.5);
\coordinate (X0L) at (-1.5,-0.5);
\coordinate (X0R) at (1.5,-0.5);
\coordinate (X1L) at (-1.5,3.5);
\coordinate (X1R) at (1.5,3.5);

\draw (X0) node[yshift=-20] {$\bx^i_r=\bx^j_k$}
-- (X1) node[yshift=20] {$\bx^i_{r+1} = \bx^j_{k-1}$};
\draw (X0L) -- (X0) -- (X0R);
\draw (X1L) -- (X1) -- (X1R);

\draw (XL) node {$\omega^i$};
\draw (XR) node {$\omega^j$};

\fill (X0) circle [radius=0.07]
(X1) circle [radius=0.07];

\coordinate (n) at ($(X0)!1cm!-90:(X1)-(X0)$);
\coordinate (dn0) at ($0.1*(n)$);
\draw ($(X0)!0.03!(X1) + (dn0)$)
--    ($(X0)!0.15!(X1) + (dn0)$);
\node[right] at ($(X0)!0.08!(X1) + (dn0)$)
{$\tP^j_{k}$};
\draw ($(X0)!0.03!(X1) - (dn0)$)
--    ($(X0)!0.15!(X1) - (dn0)$);
\node[left] at ($(X0)!0.08!(X1) - (dn0)$)
{$\tP^i_{r}$};

\draw ($(X0)!0.18!(X1) + (dn0)$)
--    ($(X0)!0.82!(X1) + (dn0)$);
\node[right] at ($(X0)!0.5!(X1) + (dn0)$)
{$\tP^j_{k-\hf}$};
\draw ($(X0)!0.18!(X1) - (dn0)$)
--    ($(X0)!0.82!(X1) - (dn0)$);
\node[left] at ($(X0)!0.5!(X1) - (dn0)$)
{$\tP^i_{r+\hf}$};

\draw ($(X0)!0.85!(X1) + (dn0)$)
--    ($(X0)!0.97!(X1) + (dn0)$);
\node[right] at ($(X0)!0.92!(X1) + (dn0)$)
{$\tP^j_{k-1}$};
\draw ($(X0)!0.85!(X1) - (dn0)$)
--    ($(X0)!0.97!(X1) - (dn0)$);
\node[left] at ($(X0)!0.92!(X1) - (dn0)$)
{$\tP^i_{r+1}$};

\draw (0,-1.5) node {\small $\tP^i_{r} \ne \tP^j_k$, $\tP^i_{r+\hf} \ne\tP^j_{k-\hf}$, $\tP^i_{r+1} \ne \tP^j_{k-1}$};
\end{tikzpicture}
\caption{Schemes using our nodal solver}\label{fig_pC}
\end{subfigure}
\caption{Surface pressures on $\omega^i\cap \omega^j$ defined in different schemes. (a) EUCCLHYD \cite{Maire-EUCCLHYD} defines two different pressures for each half surface. (b) Schemes \cite{Vilar-LDG,Morgan-LDG} using both 2D nodal solver and 1D solver at each surface. The nodal solver is applied at both ends, introducing two pressures for each end. The 1D solver is used to provide both the pressure and the normal velocity at the middle point. The nonlinearity of normal velocities bends the cell surface. (c) Our augmented nodal solver introduces three pressures at each side, which are reassigned into two half surface pressures as Eqn. \eqref{ANS_pres_assign}. The discontinuity of pressure at the middle point allows us to maintain the local linearity of normal velocity. }\label{fig_pABC}
\end{figure}
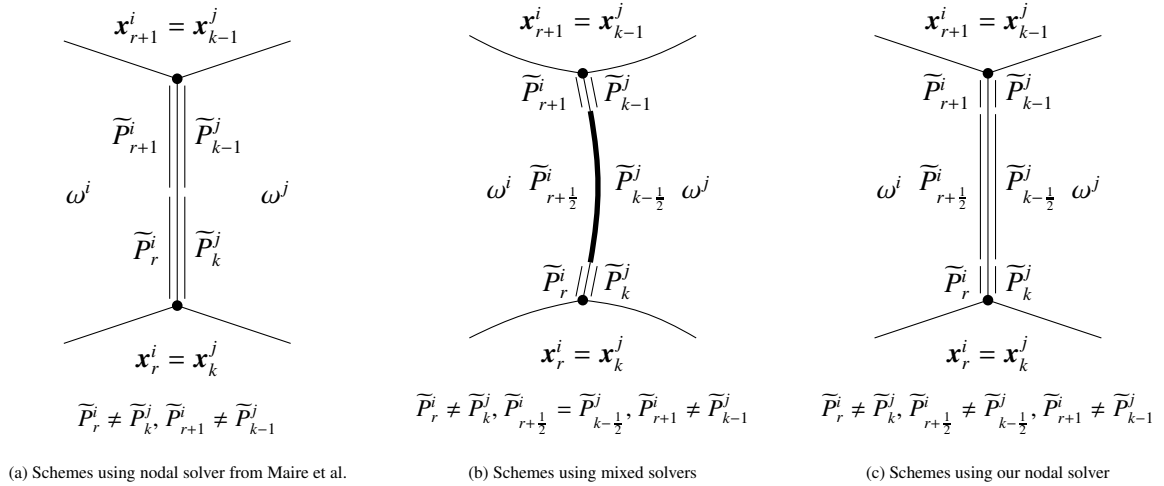 
}

According to Simpson's rule to obtain high accuracy, then we still divided it into two half surfaces

\begin{eqnarray} \label{ANS_pres_assign}
  \ah\tP^i_r + \fh\tP^i_{r+\hf} + \ah \tP^i_{r+1}=   \hf\tP^i_{r,r+\hf} + \hf\tP^i_{r+\hf,r+1}.
\end{eqnarray}
It comes naturally the middle term is written as 
\begin{eqnarray} \label{ANS_pres_split}
    \tP^i_{r+\hf} =  \frac{1}{2}\brb{P^{i}_{r+\hf}- \alpha^{i}_{r+\hf}\bra{\tbVi_r - \bVi_{r+\hf}}\cdot\bN^i_{r,r+1} - \delta \tP^i_{r+\hf}}+
    \frac{1}{2}\brb{P^{i}_{r+\hf}-\alpha^{i}_{r+\hf} \bra{\tbVi_{r+1} - \bV^{i}_{r+\hf}}\cdot \bN^i_{r,r+1}+\delta \tP^i_{r+\hf}},
\end{eqnarray}
with $\delta \tP^i_{r+\hf}$ undetermined. Accordingly, $\tP^i_{r,r+\hf}, \tP^i_{r+\hf,r+1}$ shall follow the form of 
\begin{small}
\begin{eqnarray} \label{ANS_pres_raw}
\left\{
\begin{aligned}
    \tP^i_{r,r+\hf} &= \bra{\frac{1}{3}P^{i}_{r} +\frac{2}{3} P^{i}_{r+\hf}}
-
\bra{\frac{1}{3}\alpha^{i}_{r}+\frac{2}{3}\alpha^{i}_{r+\hf}}\bra{\tbVi_r - \bVi_{r}}\cdot \bN^i_{r,r+1}
- \fh\brb{\alpha^{i}_{r+\hf}\bra{\bVi_{r+\hf} - \bVi_{r}}\cdot \bN^i_{r,r+1} + \delta \tP^i_{r+\hf}},  \\
    \tP^i_{r+\hf,r+1} &=  \bra{\frac{1}{3}P^{i}_{r+1} +\frac{2}{3} P^{i}_{r+\hf}} -\bra{\frac{1}{3}\alpha^{i}_{r+1}+\frac{2}{3}\alpha^{i}_{r+\hf}}\bra{\tbVi_{r+1} - \bVi_{r+1}}\cdot \bN^i_{r,r+1}
- \fh\brb{\alpha^{i}_{r+\hf}\bra{\bVi_{r+\hf} - \bVi_{r+1}}\cdot \bN^i_{r,r+1} - \delta \tP^i_{r+\hf}}.
\end{aligned}
\right.
\end{eqnarray}
\end{small}
Based on symmetry, we choose $\delta\tP^i_{r+\hf}$ to satisfy 
\begin{eqnarray} \label{ANS_dp}
\alpha^{i}_{r+\hf}\bra{\bVi_{r+\hf} - \bVi_{r}}\cdot \bN^i_{r,r+1} + \delta \tP^i_{r+\hf} = \alpha^{i}_{r+\hf}\bra{\bVi_{r+\hf} - \bVi_{r+1}}\cdot \bN^i_{r,r+1} - \delta \tP^i_{r+\hf}, 
\end{eqnarray}
that is $\delta \tP^i_{r+\hf} = \hf \alpha^i_{r+\hf}\bra{\bVi_r - \bVi_{r+1}}\cdot \bN^i_{r,r+1}$. 

For a half surface between two cells $e = \overline{\bx^i_r\bx^i_{r+\hf}} = \overline{\bx^j_{k-\hf}\bx^j_{k}} \subset \omega^i\cap \omega^j$, the numerical pressures $\tPs$ on both sides are denoted as $\tPs_{e+} = \tPsi_{r,r+\hf}$ and $\tPs_{e-} = \tPsj_{k-\hf,k}$, and $L_e$ and $\bN_e$ are the length and the unit outward normal vector of half surface $e$, the normal direction is relative to $\omega^i$ (see Fig. \ref{fig_NodalSolver}). Substituting \eqref{ANS_dp} into \eqref{ANS_pres_raw}, we get 
\noindent
\textbf{Jump condition: } 
\begin{eqnarray}\label{maire_pres_hoo}
\left\{
\begin{aligned}
&\tP_{e+}= \tP^i_{r,r+\hf} =
\bra{\bh P^{i}_{r} +\fh P^{i}_{r+\hf} +\alpha^{i}_{r+\hf} w^{i}_{r+\hf}}
-
\bra{\bh \alpha^{i}_{r}+\fh \alpha^{i}_{r+\hf}}\bra{\tbVi_r - \bV^{i}_{r}}\cdot \bN^i_{r,r+1}              \\
&\tP_{e-} = \tP^j_{k-\hf,k} =
\bra{\bh P^{j}_{k} +\fh P^{j}_{k-\hf} + \alpha^{j}_{k-\hf} w^{j}_{k-\hf}} - 
\bra{\bh \alpha^{j}_{k} + \fh \alpha^{j}_{k-\hf}}
\bra{\tbVi_r - \bV^{j}_{k}}\cdot \bN^j_{k-1,k}
\end{aligned}
\right. ,
\end{eqnarray}
where
\begin{eqnarray*}
w^{i}_{r+\hf} = \fh
\bra{\bV^{i}_{r+\hf}-\hf\bV^{i}_{r} -\hf\bV^{i}_{r+1}}\cdot \bN^i_{r,r+1}, \quad w^{j}_{k-\hf} = \fh
\bra{\bV^{j}_{k-\hf}-\hf\bV^{j}_{k} -\hf\bV^j_{k-1}}\cdot \bN^j_{k-1,k}.
\end{eqnarray*}
In order to maintain conservation, pressures around each vertex should be balanced, that is 

\noindent
\textbf{Balance condition: }
\begin{eqnarray}\label{maire_balance_hoo}
\sum_{e\in \cE(\bx^i_r) }L_e\bra{\tPs_{e+}-
\tPs_{e-}}\bN_e=\boldsymbol{0},
\end{eqnarray}
where $\cE(\bx^i_r)$ contains all half surface with $\bx^i_r$ as a vertex. 
{\color{black}
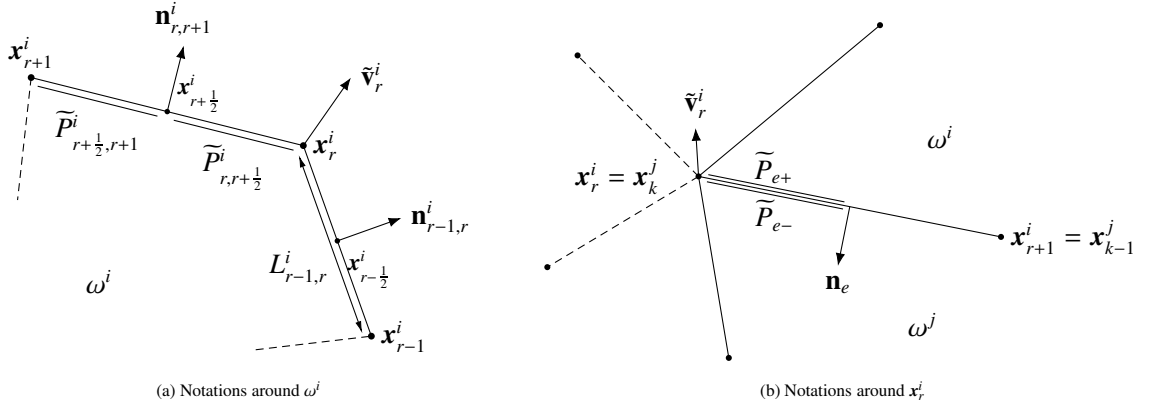
\begin{figure}[!htp]
\centering
\begin{subfigure}{0.48\linewidth}
\centering
\begin{tikzpicture}[scale=0.9]
\coordinate (Mr0p) at (1.3,-1);
\coordinate (Mr0) at (3,-0.8);
\coordinate (Mr) at (2,2);
\coordinate (Mr1) at (-2,3);
\coordinate (Mr1p) at (-2.2,1.2);

\node at (-1,0) {$\omega^i$};
\draw [densely dashed] (Mr0p) -- (Mr0);
\draw  (Mr0) node[right] {$\bx^i_{r-1}$}
-- (Mr)  node[right] {$\bx^i_r$}
-- (Mr1) node[above] {$\bx_{r+1}^i$};
\draw [densely dashed](Mr1)-- (Mr1p);
\fill (Mr0) circle [radius=0.05]
(Mr) circle [radius=0.05]
(Mr1) circle [radius=0.05];
\coordinate (Mc0) at ($(Mr)!0.5!(Mr0)$);
\fill (Mc0) circle [radius=0.04];
\node[font=\small, below right] at (Mc0) {$\bx^i_{r-\hf}$};
\coordinate (Mc1) at ($(Mr)!0.5!(Mr1)$);
\fill (Mc1) circle [radius=0.04];
\node[font=\small, above right] at ($(Mc1)+(0,-0.15)$) {$\bx^i_{r+\hf}$};

\coordinate (Vr) at (0.7,1);
\draw[-{Latex}] (Mr) -- ++(Vr) node[right] {$\tbVi_r$};

\coordinate (dn0) at (-0.1,-0.05);
\draw[{Latex[width=0.08cm]}-{Latex[width=0.08cm]}]
($(Mr)!0.03!(Mr0) + (dn0)$) -- ($(Mr)!0.97!(Mr0) + (dn0)$);
\node[below left] at ($(Mr0)!0.5!(Mr)$) {$L^i_{r-1,r}$};

\coordinate (Mc0) at ($(Mr0)!0.5!(Mr)$);
\coordinate (Mn0) at ($(Mc0)!1cm!90:(Mr0)$);
\draw[-{Latex}] (Mc0) -- (Mn0) node[right] {$\bN^i_{r-1,r}$};

\coordinate (Mc1) at ($(Mr1)!0.5!(Mr)$);
\coordinate (Mn1) at ($(Mc1)!1cm!90:(Mr)$);
\draw[-{Latex}] (Mc1) -- (Mn1) node[above] {$\bN^i_{r,r+1}$};

\coordinate (dn1) at ($(Mc1)!-0.1cm!90:(Mr)-(Mc1)$);
\coordinate (Mrd1) at ($(Mr)+(dn1)$);
\coordinate (Mr1d1) at ($(Mr1)+(dn1)$);
\draw ($(Mrd1)!0.03!(Mr1d1)$) -- ($(Mrd1)!0.47!(Mr1d1)$)
($(Mrd1)!0.53!(Mr1d1)$) -- ($(Mrd1)!0.97!(Mr1d1)$);
\node[below] at ($(Mrd1)!0.25!(Mr1d1)$) {$\tPsi_{r,r+\hf}$};
\node[below] at ($(Mrd1)!0.75!(Mr1d1)$) {$\tPsi_{r+\hf,r+1}$};
\end{tikzpicture}
\caption{Notations around $\omega^i$}\label{fig_Ci_hoo}
\end{subfigure}
\begin{subfigure}{0.48\linewidth}
\centering
\begin{tikzpicture}[scale=0.8]
\coordinate (Mq) at (0,0);
\coordinate (Mk1) at (3,2.5);
\coordinate (Mk) at (5,-1);
\coordinate (Mqk) at (2.5,-0.5);
\coordinate (Mk2) at (-2,2);
\coordinate (Mk3) at (-2.5,-1.5);
\coordinate (Mk4) at (0.5,-3);

\fill (Mq) circle[radius=0.05];
\fill (Mk1) circle[radius=0.05];
\fill (Mk2) circle[radius=0.05];
\fill (Mk3) circle[radius=0.05];
\fill (Mk4) circle[radius=0.05];
\fill (Mk) circle[radius=0.05];
\draw [densely dashed](Mq)--(Mk2) (Mq)--(Mk3);
\draw (Mq) node[left=0.4cm] {$\bx^i_r =\bx^j_{k}$}
--(Mk) node[right] {$\bx_{r+1}^i = \bx^j_{k-1}$}
(Mq)--(Mk1) 
(Mq)--(Mk4);

\coordinate (Ok) at (4,1.9);
\coordinate (Ok1) at (-0.7,2.5);
\coordinate (Ok2) at (3.7,-1.2);
\node [below of=Ok]{$\omega^{i}$};
\node [below of=Ok2]{$\omega^{j}$};

\coordinate (n0) at ($(Mq)!1cm!-90:(Mk)-(Mq)$);
\coordinate (n1) at ($(Mq)!1cm!90:(Mk1)-(Mq)$);
\draw[-{Latex}] ($(Mq)!0.5!(Mk)$) -- ++(n0)
node[below] {$\bN_{e}$};

\coordinate (dn0) at ($0.06*(n0)$);
\coordinate (dn1) at ($0.02*(n0)$);
\draw ($(Mq)!0.03!(Mk) + (dn0)$)
--    ($(Mq)!0.48!(Mk) + (dn0)$);
\node[above] at ($(Mq)!0.25!(Mk) + (dn1)$)
{$\tPs_{e+}$};
\draw ($(Mq)!0.03!(Mk) - (dn0)$)
--    ($(Mq)!0.48!(Mk) - (dn0)$);
\node[below] at ($(Mq)!0.25!(Mk) - (dn1)$)
{$\tPs_{e-}$};

\draw[-{Latex}] (Mq) -- ($(Mq)+(-0.05,0.8)$)
node[above] {$\tbVi_{r}$};
\end{tikzpicture}
\caption{Notations around $\bx^i_r$}\label{fig_Mq_v}
\end{subfigure}
\caption{Notations used in nodal solver.}\label{fig_NodalSolver}
\end{figure}
}

According to both conditions, the solution procedure is as follows.
Firstly, for each half surface $e = \overline{\bx^i_r\bx^i_{r+\hf}} = \overline{\bx^j_{k-\hf}\bx^j_{k}}$, we could rewrite \eqref{maire_pres_hoo} in a simplified form as
\begin{eqnarray}\label{maire_pres_hoo_e}
\left\{
\begin{aligned}
& \tPs_{e+} =
P_{e+} - \alpha_{e+}\bra{\tbVi_r - \bV^{i}_{r}}\cdot \bN_{e} \\
& \tPs_{e-} =
P_{e-} + \alpha_{e-}\bra{\tbVi_r - \bV^{j}_{k}}\cdot \bN_{e}
\end{aligned}
\right. ,
\end{eqnarray}

where
\begin{gather*}
P_{e+} = \frac{1}{3}P^{i}_{r} +\frac{2}{3} P^{i}_{r+\hf} + \alpha^i_{r+\hf} w^i_{r+\hf}, \quad 
P_{e-} = \frac{1}{3}P^{j}_{k} +\frac{2}{3} P^{j}_{k-\hf} + \alpha^j_{k-\hf}w^j_{k-\hf}, \\
\alpha_{e+} = \frac{1}{3}\alpha^{i}_{r}+\frac{2}{3}\alpha^{i}_{r+\hf}, \quad
\alpha_{e-} = \frac{1}{3}\alpha^{j}_{k}+\frac{2}{3}\alpha^{j}_{k-\hf}.
\end{gather*}

Substitute Eqs.\eqref{maire_pres_hoo_e} into Eqs.\eqref{maire_balance_hoo}, we can get $\tbVi_r$ by solving the linear system 
\begin{eqnarray} \label{maire_vel_hoo}
\sum_{e\in \cE(\bx^i_r)} L_e
\bra{{\alpha}_{e+}+{\alpha}_{e-}}
\bra{\tbVi_r\cdot\bN_{e}-\tcVs_{e}}\bN_{e} = \boldsymbol{0},
\end{eqnarray}
where
\begin{gather} \label{normal_vel_e}
\tcVs_{e} = \frac{
\alpha_{e+} \bV^i_r\cdot \bN_e + \alpha_{e-} \bV^j_k\cdot \bN_e
+ P_{e+} - P_{e-}
}
{\alpha_{e+} + \alpha_{e-}}
\end{gather}
which could be regarded as a weighted average of normal velocity given by classical 1D acoustic Riemann solver for half surface $e$, more discussion can be found in Section \ref{chap:Lag:Nodal:Velocity}. Meanwhile, Eqs. \eqref{maire_vel_hoo} can be interpreted as a least square problem
\begin{eqnarray} \label{maire_vel_hoo_lsq}
\tbVi_{r} = \underset{\bV\in\mathbb{R}^2}{\arg \min}\sum_{e\in \cE(\bx^i_r)} L_e
\bra{{\alpha}_{e+}+{\alpha}_{e-}}
\bra{\bV\cdot\bN_{e}-\tcVs_{e}}^2.
\end{eqnarray}


\begin{theorem}
Omitting the boundary conditions, global conservation relations of momentum and energy could be achieved as
\begin{eqnarray*}
\bra{\sum_{i=1}^I m^i\bar{\bV}^i}_t &=& -\hf\sum_{i=1}^I\sum_{r=1}^4 L^i_{r,r+1}
\bra{\tPsi_{r,r+\hf} + \tPsi_{r+\hf,r+1}}
\bN^i_{r,r+1}\\
&=& -\hf \sum_{\bx\in \mathcal{X}}\sum_{e\in\cE(\bx)} L_e\bra{\tPs_{e+}-\tPs_{e-}}\bN_e = \boldsymbol{0},
\end{eqnarray*}
\begin{eqnarray*}
\bra{\sum_{i=1}^I m^i\bar{E}^i}_t &=& -\hf\sum_{i=1}^I\sum_{r=1}^4 L^i_{r,r+1}
\bra{\tPsi_{r,r+\hf}\tbVi_r + \tPsi_{r+\hf,r+1}\tbVi_{r+1}}\cdot
\bN^i_{r,r+1}\\
&=& -\hf \sum_{\bx\in \mathcal{X}} \tbVs(\bx)\cdot \brb{\sum_{e\in\cE(\bx)} L_e\bra{\tPs_{e+}-\tPs_{e-}}\bN_e} = 0,
\end{eqnarray*}
where $\mathcal{X}$ is the mesh vertex set and $\tbVs(\bx)$ is the Lagrangian velocity of mesh vertex $\bx\in \mathcal{X}$.
\end{theorem}

It is worth mentioning that even if the physical field is smooth, the velocity $\tbVi_r$ given by our nodal solver does NOT satisfy $\tbVi_r = \bVi_r$. This property allows the accuracy of numerical flux to exceed the second-order limitation mentioned in {\bf Appendix \hyperlink{Apd.2rd}{A}}, i.e. 
\begin{theorem} \label{remark_hoo}
For smooth physical field $\U$ on computational domain $\omega$, consider a uniformly refined set of quadrilateral meshes $\{M_h,\; h \in \mathbb{R}  \}$.
The numerical flux given by (\ref{maire_pres_hoo}-\ref{maire_balance_hoo}) has 4th order accuracy $\mathcal{O} (h^4)$ at each cell surface, see {\bf Appendix \hyperlink{Apd.4th}{B} } for proof.
\end{theorem}

\section{Features of our nodal solver} \label{chap:Lag:Nodal}
\subsection{Riemann wave structure}
In the local coordinate system, approximate Riemann wave structure at each surface can be described in detail. 
Take surface $e = \overline{\bx^i_r\bx^i_{r+1}} = \overline{\bx^j_{k-1}\bx^j_{k}} = \omega^i\cap \omega^j$ for example, the wave structure is described by $(\U_\zeta,\bF_\zeta^i,\bF_\zeta^j)(x,t)$ 
which are parameterized by $\zeta$ indicating the local 1D Riemann solution at the normal direction at point $\zeta\bx^i_r + (1-\zeta)\bx^i_{r+1}$. It should be noted that the discontinuity of pressure on both sides of $e$ makes it impossible to provide one reasonable numerical flux $\bF$ across both sizes, as is the case with 1D HLLC solvers in Section \ref{Chap:Lag:Revisit:1D}. Following the HLLC-2D solver \cite{Shen-HLLC2D}
 by Shen et.al, two different numerical fluxes $\bF_\zeta^i$ and $\bF_\zeta^j$ are introduced instead, which are associated with $\omega^i$ and $\omega^j$ respectively. That is 
\begin{eqnarray}\label{rup_all}
(\U_\zeta,\bF_\zeta^i,\bF_\zeta^j)(x,t) = \left\{
\begin{array}{lcl}
\bra{\U_L,\bF_L^i,\bF_L^j}(\zeta),          &      & {\mbox {if } \ x/t     \leq  S_L(\zeta)},\\
\bra{\tilde{\U}_L,\tbF_L^i,\tbF_L^j}(\zeta),     &      & {\mbox {if } \  S_L(\zeta) < x/t \leq  \ustar(\zeta)}, \\
\bra{\tilde{\U}_R,\tbF_R^i,\tbF_R^j}(\zeta),     &      & {\mbox {if } \ \ustar(\zeta) < x/t \leq S_R(\zeta)}, \\
\bra{\U_R,\bF_R^i,\bF_R^j}(\zeta),     &      & {\mbox {if } \ S_R(\zeta) <x/t},
\end{array} \right. 
\end{eqnarray}
 where the numerical fluxes $\bF_\zeta^i$ and $\bF_\zeta^j$ follow Rankine–Hugoniot condition: 
\begin{gather}
\left\{\begin{array}{l}
\tbF_L^i(\zeta) - \bF_L^i(\zeta) = S_L(\zeta)\bra{\tbU_L(\zeta) - \U_L(\zeta)} = \tbF_L^j(\zeta) - \bF_L^j(\zeta), \\ 
\tbF_R^i(\zeta) - \tbF_L^i(\zeta) =\;\ustar(\zeta)\; \bra{\tbU_R(\zeta) - \tbU_L(\zeta)} = \tbF_R^j(\zeta) - \tbF_L^j(\zeta), \\ 
\tbF_R^i(\zeta) - \bF_R^i(\zeta) = S_R(\zeta)\bra{\tbU_R(\zeta) - \U_R(\zeta)} = \tbF_R^j(\zeta) - \bF_R^j(\zeta).
\end{array}
\right.\label{rup_rh} 
\end{gather}


Our augmented nodal solver divides surface $e$ into four parts $[0,1]=I_1\sqcup I_2\sqcup I_3 \sqcup I_4$ for four Riemann solutions located at $\bx^i_r$, $\bx^i_{r+\hf}$, $\bx^i_{r+\hf}$ and $\bx^i_{r+1}$ respectively, where 
$$
I_1 = [0,\ah],\; I_2=[\ah,\hf],\; I_3=[\hf,\gh],\; I_4=[\gh,1].
$$
Let $\bN_e,\bT_e$ be the unit normal and tangent of edge $e$, we define $u = \bV\cdot \bN_e$, $v = \bV\cdot \bT_e$ for velocity field of cells, and $\tu = \tbV \cdot \bN_e$ for nodal  contact velocities. 

Firstly, the initial values are directly given by field $\U$ at $\bx^i_r, \bx^i_{r+\hf}, \bx^i_{r+1}$ as  
\begin{gather*}
	\U_L(\zeta) = \left\{
		\begin{array}{ll}
			(\rho^i_r, \rho^i_r u^i_r,\rho^i_r 
v^i_r,\rho^i_r E^i_r)^\top, & \zeta\in I_1, \\
			(\rho^i_{r+\hf}, \rho^i_{r+\hf} u^i_{r+\hf},\rho^i_{r+\hf} 
v^i_{r+\hf},\rho^i_{r+\hf} E^i_{r+\hf})^\top,  & \zeta\in I_2 \cup I_3, \\
(\rho^i_{r+1}, \rho^i_{r+1} u^i_{r+1},\rho^i_{r+1} 
v^i_{r+1},\rho^i_{r+1} E^i_{r+1})^\top, & \zeta\in I_4,
		\end{array}
	\right. \\
	\U_R(\zeta) = \left\{
		\begin{array}{ll}
			(\rho^j_k, \rho^j_k u^j_k,\rho^j_k v^j_k,\rho^j_k E^j_k)^\top, & \zeta\in I_1, \\
			(\rho^j_{k-\hf}, \rho^j_{k-\hf} u^j_{k-\hf},\rho^j_{k-\hf} v^j_{k-\hf},\rho^j_{k-\hf} E^j_{k-\hf})^\top,  & \zeta\in I_2 \cup I_3, \\
			(\rho^j_{k-1}, \rho^j_{k-1} u^j_{k-1},\rho^j_{k-1} v^j_{k-1},\rho^j_{k-1} E^j_{k-1})^\top, & \zeta\in I_4.
		\end{array}
	\right.
\end{gather*}
With the intial values, Eqs. \eqref{rup_rh} is closed with given wave speeds and fluxes as follows. 

The wave speeds at both size are approximated as 
\begin{gather*}
	S_L(\zeta) = \left\{
		\begin{array}{ll}
			u^i_r-\alpha^i_r/\rho^i_r, & \zeta \in I_1, \\
			u^i_{r+\hf}-\alpha^i_{r+\hf}/\rho^i_{r+\hf}, & \zeta\in I_2\cup I_3, \\
			u^i_{r+1}-\alpha^i_{r+1}/\rho^i_{r+1}, & \zeta\in I_4,
		\end{array}
	\right. \quad
	S_R(\zeta) = 
	\left\{
		\begin{array}{ll}
			u^j_k+\alpha^j_k/\rho^j_k, & \zeta \in I_1, \\
			u^j_{k-\hf}+\alpha^j_{k-\hf}/\rho^j_{k-\hf}, & \zeta\in I_2\cup I_3, \\
			u^j_{k-1}+\alpha^j_{k-1}/\rho^j_{k-1}, & \zeta\in I_4,
		\end{array}
	\right.
\end{gather*}
and the velocity of contact discontinuity is splited at the middle as 
\begin{gather*}
	\tu(\zeta) = \left\{
		\begin{array}{ll}
			\tu^i_r, & \zeta \in I_1\cup I_2, \\
			\tu^i_{r+1}, & \zeta\in I_3\cup I_4,
		\end{array}
		\right.
\end{gather*}

As for fluxes $\bF^i_L$ and $\bF^j_R$, it should be noticed that a pressure modification is required by revisiting \eqref{ANS_pres_split}. To be specific, 
\begin{gather*}
	\bF^i_L(\zeta) = \left\{
		\begin{array}{ll}
			\bF\bra{\U_L(\zeta)}, & \zeta \in I_1, \\
			\bF\bra{\U_L(\zeta)}-\delta \bF^i_{r+\hf}(\zeta), & \zeta \in I_2, \\
			\bF\bra{\U_L(\zeta)}+\delta \bF^i_{r+\hf}(\zeta), & \zeta \in I_3, \\
			\bF\bra{\U_L(\zeta)}, & \zeta\in I_4,
		\end{array}
		\right. \quad
	\bF^j_R(\zeta) = \left\{
		\begin{array}{ll}
			\bF\bra{\U_R(\zeta)}, & \zeta \in I_1, \\
			\bF\bra{\U_R(\zeta)}+\delta \bF^i_{k-\hf}(\zeta), & \zeta \in I_2, \\
			\bF\bra{\U_R(\zeta)}-\delta \bF^i_{k-\hf}(\zeta), & \zeta \in I_3, \\
			\bF\bra{\U_R(\zeta)}, & \zeta\in I_4,
		\end{array}
		\right.
\end{gather*}
where $\delta \bF^i_{r+\hf}$ and $\delta \bF^i_{k-\hf}$ are defined by \eqref{ANS_pres_split} as 
\begin{gather*}
	\delta \bF^i_{r+\hf}(\zeta) = \delta \tP^i_{r+\hf} (0,1,0,\tu(\zeta))^\top, \quad \delta \bF^j_{k-\hf}(\zeta) = \delta \tP^j_{k-\hf} (0,1,0,\tu(\zeta))^\top.
\end{gather*}

With given states above, the undetermined states in \eqref{rup_all} can be obtained by solving \eqref{rup_rh}. Especially, the numerical pressures $\tP^i_L,\tP^j_R$ on both size of $e$ are given by 
\begin{gather*}
\tP^i_L(\zeta) = \left\{
	\begin{array}{ll}
		P^i_r-\alpha^i_r(\tu^i_r - u^i_r), & \zeta\in I_1, \\
		P^i_{r+\hf}-\alpha^i_{r+\hf}(\tu^i_r - u^i_{r+\hf})-\delta\tP^i_{r+\hf}, & \zeta\in I_2, \\
		P^i_{r+\hf}-\alpha^i_{r+\hf}(\tu^i_{r+1} - u^i_{r+\hf})+\delta\tP^i_{r+\hf}, & \zeta\in I_3, \\
		P^i_{r+1}-\alpha^i_{r+1}(\tu^i_{r+1} - u^i_{r+1}), & \zeta\in I_4,
	\end{array}
	\right.
\quad
\tP^j_R(\zeta) = 
\left\{
	\begin{array}{ll}
		P^j_k + \alpha^j_k(\tu^j_k - u^j_k), & \zeta\in I_1, \\
		P^j_{k-\hf}-\alpha^j_{k-\hf}(\tu^j_k - u^j_{k-\hf})+\delta\tP^j_{k-\hf}, & \zeta\in I_2, \\
		P^j_{k-\hf}-\alpha^j_{k-\hf}(\tu^j_k - u^j_{k-\hf})-\delta\tP^j_{k-\hf}, & \zeta\in I_3, \\
		P^j_{k-1}-\alpha^j_{k-1}(\tu^j_{k-1} - u^j_{k-1}), & \zeta \in I_4,
	\end{array}
	\right.
\end{gather*}

In our scheme, the interval is divided based on Simpson's rule. In terms of formulation, if the partition is chosen as $I_1 = [0,\hf]$ and $I_4=[\hf,1]$, we will get the HLLC-2D solver, but it should be noticed that the nodal velocity fields $\tbV$ are acquired in different way. $\tbV$ for our nodal solver is solved following Section \ref{Chap:Lag:Revisit:ANS}, while the nodal velocity of HLLC-2D solver is acquired as Section \ref{Chap:Lag:Revisit:Maire}. 

Equivalence of Lagrangian flux (\ref{cell_evolve_vol}-\ref{cell_evolve_E}) and \eqref{rup_all} can be checked that 
\begin{gather*}
\hf \bra{\tbVi_r + \tbVi_{r+1}}\cdot\bN^i_{r,r+1} = \int_0^1 \tu(\zeta)\dd \zeta, \\
\hf\bra{\tP^i_{r,r+\hf}+\tP^i_{r+\hf,r+1}} = \int_0^1 \tP_L^i(\zeta) \dd \zeta, 
\\
\hf\bra{\tP^i_{r,r+\hf}\tbVi_r+\tP^i_{r+\hf,r+1}\tbVi_{r+1}}\cdot \bN^i_{r,r+1} = \int_0^1 \tP_L^i(\zeta) \tu(\zeta) \dd \zeta.
\end{gather*}

\subsection{Property of $\tbV$}\label{chap:Lag:Nodal:Velocity}
In the laboratory coordinate system, the cell-centered Lagrangian schemes developed based on works \cite{Maire-EUCCLHYD, Maire-HighOrder} by Maire and et al. solve the nodal velocity field $\tbV(\bx)$ in the form of
\begin{eqnarray}\label{global_vel_int}
\underset{\tbV(\bx)}{\arg \min} \sum_{f\in \mathcal{F}} \int_{f} (\alpha_{+} + \alpha_{-})\bra{\tbV\cdot\bN - \cVs }^2 \dd l,
\end{eqnarray}
but vary in the discretizations of $\tbV(\bx), \cVs(\bx)$ and the edge integrals. Here $\mathcal{F}$ is the set of cell surfaces, $\cVs$ is the normal velocity at each edge computed by 1D Riemann solver. 
For a more detailed and focused discussion on $f$, new notations are defined for this section as Fig. \ref{fig_notation_f}. 
In order to distinguish the physical fields on both sides of surface, subscript $+$ and $-$ are used here. Let $\alpha_{\pm} = \alpha_+ + \alpha_-$ be the sum of acoustic impedance of both sides for shorter formulation. For any physical field $\varphi$, let $\varphi^f_{\zeta+}$ and $\varphi^f_{\zeta-}$ be its value at point $\bx^f_\zeta = (1-s)\bx^f_0 + s \bx^f_1$ of both size. For example, if normal velocity $\cVs$ is acquired by HLLC solver, then
$$
\cVs^f_\zeta = \frac{\alpha^f_{\zeta+} \bV^f_{\zeta+}\cdot\bN^f + \alpha^f_{\zeta-}\bV^f_{\zeta-}\cdot\bN^f + P^f_{\zeta+}-P^f_{\zeta-}}{\alpha^f_{\zeta+} + \alpha^f_{\zeta-}}.
$$

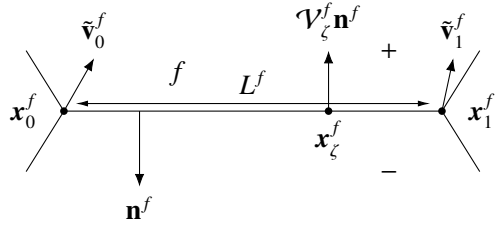
\begin{figure}
    \centering
\begin{tikzpicture}[scale=1]
\coordinate (X0) at (0,0);
\coordinate (X1) at (5,0);
\coordinate (Xs) at (3.5,0);
\coordinate (X0L) at (-0.5,0.8);
\coordinate (X0R) at (-0.5,-0.8);
\coordinate (X1L) at (5.5,0.8);
\coordinate (X1R) at (5.5,-0.8);

\node at (1.5,0.5) {$f$};
\node at (4.3,0.8) {$+$};
\node at (4.3,-0.8) {$-$};

\fill (X0) circle[radius=0.05];
\fill (X1) circle[radius=0.05];
\fill (Xs) circle[radius=0.05];
\draw (X0L) -- (X0) -- (X0R);
\draw (X1L) -- (X1) -- (X1R);
\draw (X0) node[xshift=-15] {$\bx^f_0$} -- (X1) node[xshift=15] {$\bx^f_1$};
\node[below] at (Xs) {$\bx^f_\zeta$};

\coordinate (dn0) at (0,0.1);
\draw[{Latex[width=0.08cm]}-{Latex[width=0.08cm]}]
($(X0)!0.03!(X1) + (dn0)$) -- ($(X0)!0.97!(X1) + (dn0)$);
\node[above] at ($(X0)!0.5!(X1) + (dn0)$) {$L^f$};

\coordinate (n0) at ($(X0)!1cm!-90:(X1)-(X0)$);
\coordinate (n1) at ($(X0)!1cm!90:(X1)-(X0)$);
\draw[-{Latex}] ($(X0)!0.2!(X1)$) -- ++(n0)
node[below] {$\bN^{f}$};

\coordinate (Vr0) at (0.4,0.7);
\draw[-{Latex}] (X0) -- ++(Vr0) node[above] {$\tbV^f_0$};
\coordinate (Vr1) at (0.15,0.7);
\draw[-{Latex}] (X1) -- ++(Vr1) node[above] {$\tbV^f_1$};

\coordinate (Vrs) at (0,0.8);
\draw[-{Latex}] (Xs) -- ++(Vrs) node[above] {$\cVs^f_\zeta \bN^f$};




\end{tikzpicture}
\caption{Notations around surface $f$.}
\label{fig_notation_f}
\end{figure}

As for EUCCLHYD itself, it discretizes $\int_f$ by the trapezoidal rule and rewrites \eqref{global_vel_int} into
\begin{eqnarray}
\underset{\tbV(\bx)}{\arg \min} \sum_{f\in \mathcal{F}} \brb{
\hf L^f \alpha^f_{0\pm} \bra{\tbV^f_0\cdot \bN^f - \cVs^f_0}^2
+ \hf L^f \alpha^f_{1\pm} \bra{\tbV^f_1\cdot \bN^f - \cVs^f_1}^2
}, 
\end{eqnarray}

Consider high-order schemes using quadratic edges such as \cite{Vilar-LDG}, an extra control point is introduced at the middle of each edge, where the normal velocity is given by 1D HLLC Riemann solver. For these schemes, the velocity field can be formulated by applying Simpson's rule on \eqref{global_vel_int}, i.e.
\begin{eqnarray} \label{global_vel_quad}
\underset{\tbV(\bx)}{\arg \min} \sum_{f\in \mathcal{F}} \brb{
\ah L^f \alpha^f_{0\pm} \bra{\tbV^f_0\cdot \bN^f_0 - \cVs^f_0}^2
+ \fh L^f \alpha^f_{\hf\pm} \bra{\tbV^f_\hf \cdot \bN^f_\hf -\cVs^f_\hf}^2
+ \ah L^f \alpha^f_{1\pm} \bra{\tbV^f_1\cdot \bN^f_1 - \cVs^f_1}^2
}.
\end{eqnarray}

In this article, we aim to acquire velocity field which has the comparable accuracy as \eqref{global_vel_quad} without extra control points. Since linear velocity field is compatible with straight-edge mesh, a straightforward idea is replacing $\tbV^f_\hf$ with $\hf(\tbV^f_0+\tbV^f_1)$, namely
\begin{eqnarray} \label{global_vel_linear}
\underset{\tbV(\bx)}{\arg \min} \sum_{f\in \mathcal{F}} \brb{
\ah L^f \alpha^f_{0\pm} \bra{\tbV^f_0\cdot \bN^f - \cVs^f_0}^2
+ \fh L^f \alpha^f_{\hf\pm} \bra{\hf\tbV^f_0 \cdot \bN^f + \hf\tbV^f_1 \cdot \bN^f -\cVs^f_\hf}^2
+ \ah L^f \alpha^f_{1\pm} \bra{\tbV^f_1\cdot \bN^f - \cVs^f_1}^2
}.
\end{eqnarray}
However, the cross term of $\tbV^f_0$ and $\tbV^f_1$ complicates the least square problem \eqref{global_vel_linear} which requires direct or iterative technique for solving linear system of mesh size. The computational cost would be excessive comparing with existing schemes 
\cite{Vilar-LDG,Morgan-LDG,Cheng-3rd,Morgan-LDG-2,Liu-subcellLDG} 
and superfluous in the light of hindsight. To make \eqref{global_vel_linear} directly solvable, the middle term is expanded as
\begin{eqnarray*}
& & \fh L^f \alpha^f_{\hf\pm} \bra{\hf\tbV^f_0 \cdot \bN^f + \hf\tbV^f_1 \cdot \bN^f -\cVs^f_\hf}^2 \\
&=& \frac{1}{3}L^f \alpha^f_{\hf\pm} \bra{\tbV^f_0 \cdot \bN^f + \hf \delta \tbV^f \cdot \bN^f -\cVs^f_\hf}^2
+ \frac{1}{3}L^f \alpha^f_{\hf\pm} \bra{\tbV^f_1 \cdot \bN^f - \hf \delta \tbV^f \cdot \bN^f -\cVs^f_\hf}^2,
\end{eqnarray*}
where $\delta \tbV^f$ is defined by $\delta \tbV^f = \tbV^f_1 - \tbV^f_0$ and approximated by
\begin{eqnarray} \label{normal_dvel_approx}
\delta \tbV^f\cdot \bN^f \approx \frac{\alpha^f_{\hf+}\bra{\bV^f_{1+} - \bV^f_{0+}} + \alpha^f_{\hf-}\bra{\bV^f_{1-} - \bV^f_{0-}}}{\alpha^f_{\hf+} + \alpha^f_{\hf-}}\cdot\bN^f \eqqcolon \delta \cVs^f.
\end{eqnarray}
With this approximation, \eqref{global_vel_linear} is uncoupled to
\begin{eqnarray} \label{global_vel_linear_approx_origin}
\underset{\tbV(\bx)}{\arg \min} \sum_{f\in \mathcal{F}}
\left[ \ah L^f \alpha^f_{0\pm} \bra{\tbV^f_0\cdot \bN^f - \cVs^f_0}^2
+ \frac{1}{3}L^f \alpha^f_{\hf\pm} \bra{\tbV^f_0 \cdot \bN^f + \hf \delta \cVs^f - \cVs^f_\hf}^2 \right. \notag \\
+
\left. \ah L^f \alpha^f_{1\pm} \bra{\tbV^f_1\cdot \bN^f - \cVs^f_1}^2
+ \frac{1}{3}L^f \alpha^f_{\hf\pm} \bra{\tbV^f_1 \cdot \bN^f - \hf \delta \cVs^f - \cVs^f_\hf}^2 \right],
\end{eqnarray}
which is equivalent to
\begin{eqnarray} \label{global_vel_linear_approx}
\underset{\tbV(\bx)}{\arg \min} \sum_{f\in \mathcal{F}} \brb{
\bra{\ah \alpha^f_{0\pm} + \frac{1}{3}\alpha^f_{\hf\pm}} L^f
\bra{\tbV^f_0\cdot \bN^f - \tcVs^f_0}^2
+
\bra{\ah \alpha^f_{1\pm} + \frac{1}{3}\alpha^f_{\hf\pm}} L^f
\bra{\tbV^f_1\cdot \bN^f - \tcVs^f_1}^2
},
\end{eqnarray}
where
\begin{eqnarray*}
\tcVs^f_0 = \frac{\alpha^f_{0\pm}\cVs^f_0 + 2 \alpha^f_{\hf\pm}\bra{\cVs^f_\hf - \hf \delta \cVs^f}}{\alpha^f_{0\pm} + 2 \alpha^f_{\hf\pm}}, \qquad
\tcVs^f_1 = \frac{\alpha^f_{1\pm}\cVs^f_1 + 2 \alpha^f_{\hf\pm}\bra{\cVs^f_\hf + \hf \delta \cVs^f}}{\alpha^f_{1\pm} + 2 \alpha^f_{\hf\pm}}.
\end{eqnarray*}
By dividing the surface $f$ into two half surface $e_0 =\overline{\bx^f_0\bx^f_\hf}$ and $e_1 = \overline{\bx^f_\hf\bx^f_1}$, it can be directly checked that $\tcVs^f_0$ and $\tcVs^f_1$ are nothing but normal velocity $\tcVs_e$ defined in \eqref{normal_vel_e},
\begin{gather*}
\tcVs^f_0 = \tcVs_{e_0},\quad \tcVs^f_1 = \tcVs_{e_1},
\end{gather*}
which means \eqref{global_vel_linear_approx} is equivalent to \eqref{maire_vel_hoo_lsq}.

\section{Procedures} \label{chap:Lag:Proc}
This section introduces the additional procedures required to complete this scheme.

\subsection{Limiting procedure} \label{chap:Lag:LMCV:limiter}
In order to suppress the numerical oscillations, limiting procedure is defined as follows.

Firstly, a smoothness indicator is introduced for each cell $\omega^i$ to measure the discontinuities around the cell. To that end, we use the total boundary variations (TBV) across cell boundaries
\begin{eqnarray}
\mathrm{TBV}^i(\varphi) = \frac
{\displaystyle
\sum_{j\in \mathcal{N}^i}\brb{
\frac{1}{ |e_{ij}| }
\int_{e_{ij}}( \varphi^i - \varphi^j )\, \mathrm{d} s }^\beta
}
{\displaystyle
\sum_{j\in \mathcal{N}^i}\bra{\bar\varphi^i - \bar\varphi^j}^\beta
},
\end{eqnarray}
where $\mathcal{N}^i$ is the index set of neighbor cells of $\omega^i$, $e_{ij} = \omega^i\cap \omega^j$,  $\varphi$ is a scalar physical field used to detect discontinuities, and power $\beta$ is used to control the indicator sensitivty. In our test, we choose $\varphi = \rho$ and $\beta = 4$ for better numerical preformance. Similar to \cite{charest2015high,ivan2014high}, the smoothness indicator is defined as
\begin{eqnarray}\label{tvb_si}
S^i = \frac{1 - \mathrm{TBV}^i(\varphi)}{\max\brc{\mathrm{TBV}^i(\varphi), \epsilon}},
\end{eqnarray}
with $\epsilon = 10^{-16}$ to prevent zero-division. It is noted that $S^i$ increases to infinite for $\mathrm{TBV}^i(\varphi) \to 0$ in smooth region, and remains small for discontinuities occured. As a result, a problem-dependent cutoff number $S_c$ is used as threshold and the cell $\omega^i$ with $S^i < S_c$ is considered as non-smooth cell, where the linear reconstruction is applied.

Secondly, for each non-smooth cell $\omega^i$, a local linear reconstruction is applied as
\begin{eqnarray} \label{cell_tbv}
\U^i(\bx) \leftarrow \U_\avg^i + (\nabla_{\bx}\U)^i \cdot (\bx - \bx^i_b),
\end{eqnarray}
where $\bx^i_b$ is centroid of $\omega^i$, so the linear reconstruction is conserved. The gradient $(\nabla_{\bx}\U)^i$ is acquired by solving
\begin{eqnarray} \label{cell_tbv2}
\underset{(\nabla_{\bx}\U)^i}{\arg \min}\sum_{j\in\mathcal{N}^i} \left|\U_\avg^i  -\U_\avg^j + (\nabla_{\bx}\U)^i \cdot (\bx^j_{b} - \bx^i_b)\right|^2.
\end{eqnarray}

Thirdly, MLP \cite{park2010multi} reconstruction is used to suppress numerical oscillations causing by false local extremes. For each cell $\omega^i$ and conserved scalar $\varphi$, we update $\varphi^i(\bx)$ by
\begin{eqnarray} \label{cell_mlp}
\varphi^i(\bx) \leftarrow \bar\varphi^i + \lambda_{\mathrm{MLP}}\bra{\varphi^i(\bx) - \bar\varphi^i},
\end{eqnarray}
with $\lambda_{\mathrm{MLP}} = \min_{r=1}^4 \lambda_r$ and
\begin{eqnarray}
\lambda_r=
\begin{cases}
f_{\lim}\bra{
\frac{\varphi_{\max,r} - \bar\varphi^i}
{\varphi^i_r - \bar\varphi^i}},
& \varphi^i_r - \bar\varphi^i > 0,
\\
f_{\lim}\bra{
\frac{\varphi_{\min,r} - \bar\varphi^i}
{\varphi^i_r - \bar\varphi^i}},
& \varphi^i_r - \bar\varphi^i < 0,
\\
1, & \varphi^i_r - \bar\varphi^i = 0,
\end{cases}
\end{eqnarray}
where $\varphi_{\max, r}$ and $\varphi_{\min,r}$ are the maximum and minimum of $\bar\varphi^j$ among the adacent cells of $\bx^i_r$. In this work, we use the Michalak-Gooch limiter \cite{michalak2009accuracy} as
\begin{eqnarray}
f_{\lim}(x) = \begin{cases}
-\frac{4}{27} x^3 + x, & x < \frac{3}{2},   \\
1,                     & x \ge \frac{3}{2}.
\end{cases}
\end{eqnarray}
It should be noticed that the MLP reconstruction is related to the selection of conserved scalars $\varphi$.
In order to maintain symmetry, we choose $\varphi\in \{\rho, \rho \bV \cdot \mathbf{e}_n, \rho \bV \cdot \mathbf{e}_t, \rho E\}$ for $\omega^i$ with $\{\mathbf{e}_n, \mathbf{e}_t\}$ as orthonormal basis, where $\mathbf{e}_n \mathrel{/\negmedspace/} \bar{\bV}^i$.

\subsection{Time discretization} \label{chap:Lag:LMCV:time}
The third-order SSP RK (Strong Stability Perserving Runge Kutta) method \cite{Shu-TVD} 
is used for time discretization in order to achieve consistent 3rd order accuracy. Since the VIA and PV use different control equations, their time integrations are rewriten separately in order to avoid ambiguity. We denote (\ref{cell_evolve_vol}-\ref{cell_evolve_E}) as $(\U_\via)_t = L_{\via}(\U,\bx)$ and \eqref{node_evolve_numerical} as $(\U_{\pt})_t= L_{\pt}(\U,\bx)$, where $\bx$ represents the quadrilateral mesh, $\U_{\via}$ and $\U_{\pt}$ are two column vectors of $\{\bU_\avg^i,\; i=1,2,\cdots,I.\}$ and $\{\U^i_r,\; i=1,2,\cdots,I,\, r=1,2,3,4. \}$, respectively. With physical field $\U$ and mesh $\bx$, let $\tbVs(\U,\bx)$ be the nodal velocity field given by \eqref{node_evolve_numerical}, we can use the SSP RK method at $n$-th time step as
\begin{align*}
\mbox{Stage 1)} \quad                       & \bx^{(1)} = \bx^n
+ \Delta t^n \tbVs(\U^n,\bx^n),             &                                                                                                                               \\
                        & \U_{\via}^{(1)} = \U_{\via}^n
+ \Delta t^n L_{\via}(\U^n,\bx^n),          &                                                                                                                               \\
                        & \U_{\pt}^{(1)} = \U_{\pt}^n
+ \Delta t^n L_{\pt}(\U^n,\bx^n),           &                                                                                                                               \\
\mbox{Stage 2)} \quad                       & \bx^{(2)} = \frac {3}{4} \bx^n
+ \frac{1}{4}\brb{\bx^{(1)}
+ \Delta t^n \tbVs(\U^{(1)},\bx^{(1)})},    &                                                                                                                               \\
                        & \U_{\via}^{(2)} = \frac{3}{4} \U_{\via}^n
+ \frac{1}{4} \brb{ \U_{\via}^{(1)}
+ \Delta t^n L_{\via}(\U^{(1)},\bx^{(1)})}, &                                                                                                                               \\
\quad                                       & \U_{\pt}^{(2)} = \frac {3}{4} \U_{\pt}^n
+ \frac{1}{4}\brb{\U_{\pt}^{(1)}
+ \Delta t^n L_{\pt}(\U^{(1)},\bx^{(1)})},  &                                                                                                                               \\
\mbox{Stage 3)}\quad                        & \bx^{n+1} = \frac{1}{3} \bx^{n} + \frac{2}{3} \brb{\bx^{(2)} + \Delta t^n \tbVs(\U^{(2)},\bx^{(2)})},                       & \\
\quad                                       & \U_{\via}^{n+1} = \frac{1}{3}  \U_{\via}^{n} + \frac{2}{3} \brb{\U_{\via}^{(2)} + \Delta t^n L_{\via}(\U^{(2)},\bx^{(2)})}, & \\
\quad                                       & \U_{\pt}^{n+1} = \frac{1}{3} \U_{\pt}^{n} + \frac{2}{3} \brb{\U_{\pt}^{(2)} + \Delta t^n L_{\pt}(\U^{(2)},\bx^{(2)})},      &
\end{align*}
where time step $\Delta t^n $ satisfies CFL condition
\begin{eqnarray} \label{pp_cfl}
\Delta t^n \le \sigma_e
\dfrac{m^i}{\displaystyle \sum_{r=1}^4 L^i_{r,r+1} \bra{\ah\alpha^i_r + \frac{2}{3}\alpha^i_{r+\hf} + \ah\alpha^i_{r+1}}},\; \forall i = 1,2,3,\cdots,I,
\end{eqnarray}
with CFL conffcient $\sigma_e=0.2$.

\subsection{Summary of solution procedure}
The whole solution procedure at $n$-th time step could be summarized as follows:
\begin{enumerate}
\item According to Section \ref{chap:Lag:LMCV:space} and \ref{chap:Lag:LMCV:limiter}, physical field $\U^n(\bx)$ is reconstructed and limited with VIA and PV values $\U^n = (\U^n_{\avg},\U^n_{\pt})$ and mesh $\bx^n$ at time $t^n$.
\item With physical field $\U^n(\bx)$ on mesh $\bx^n$, time step $\Delta t^n$ is determined by \eqref{pp_cfl}.
\item Let $k=0, \U^{(0)} = \U^n, \bx^{(0)} = \bx^n$, the Runge-Kutta stages are performed successively as follows:
\begin{enumerate}
\item Physical field $\U^{(k)}(\bx)$ is reconstructed and limited based on $\U^{(k)}, \bx^{(k)}$;
\item $L_{\avg}(\U^{(k)},\bx^{(k)}),L_{\pt}(\U^{(k)},\bx^{(k)})$ and $\tbVs(\U^{(k)},\bx^{(k)})$ are computed based on Section \ref{chap:Lag:LMCV:VIAflux} and \ref{chap:Lag:LMCV:PVflux}.
\item $\U^{(k+1)}$ and $\bx^{(k+1)}$ are updated by third-order SSP RK method as Section \ref{chap:Lag:LMCV:time} shows, let $ k = k+1$.
\end{enumerate}
\item The result of SSP RK method is denoted as $\U^{n+1} = \U^{(k)}, \bx^{n+1} = \bx^{(k)}$, and the time is updated by $t^{n+1} = t^n + \Delta t^n$.
\end{enumerate}

\section{Numerical tests} \label{chap:Lag:Test}
In this section, widely used benchmark tests are performed to verify the accuracy and robustness of our method. For purpose of numerical accuracy measurement, 
average $L_2$ error is defined for any physical field $\varphi$ as 
\begin{eqnarray}
\sqrt{\frac{ \sum_{i=1}^I |\omega^i| (\varphi_\avg^i -\varphi^i_{\mathrm{ref}} )^2}{\sum_{i=1}^I |\omega^i|}},\quad 
\end{eqnarray}
where the reference value $\varphi^i_{\mathrm{ref}} = \int_{\omega^i} \varphi_{\mathrm{ref}}\,\mathrm{d}\bx = \int_{\Omega} \varphi_{\mathrm{ref}}\, \mathrm{Det}\bra{\frac{\partial \bx^i}{\partial \bxi}} \mathrm{d}\bxi$ are computed by applying 3-point Guass-Legendre internal on each dimension of $\Omega$ for sufficient accuracy. Unless otherwise specified, tests are based on ideal gas with specific heat ratio $\gamma$.

\subsection{Isentropic vortex problem}
The isentropic vortex \cite{shuEssentially1997} is a steady smooth flow with strong nonlinearity. The initial distribution is defined by
\begin{eqnarray*}
\rho_0 = T^{\frac{1}{\gamma-1}},\quad \bV_0 = \frac{\epsilon}{2\pi}\exp\bra{\frac{1-r^2}{2}}(-y,x)^{\top},\quad P_0=T^{\frac{\gamma}{\gamma-1}},
\end{eqnarray*}
where $r = \sqrt{x^2+y^2}$, the vortex parameter $\epsilon=5$ and $\gamma = 1.4$. The computational domain is taken as $[-10, 10]^2$ with natural boundaries.

The local pressure fields at $t = 2$ and $ t=3$ using LMCV are shown in Fig. \ref{Fig.T1} with resolution of $100\times 100$. The mesh distortion accumulates over time while the pressure field is preserved well via our high order scheme.
\begin{figure}[!htbp]
	\centering  
	\begin{subfigure}{0.48\linewidth}
		\centering
		\includegraphics[width=\textwidth]{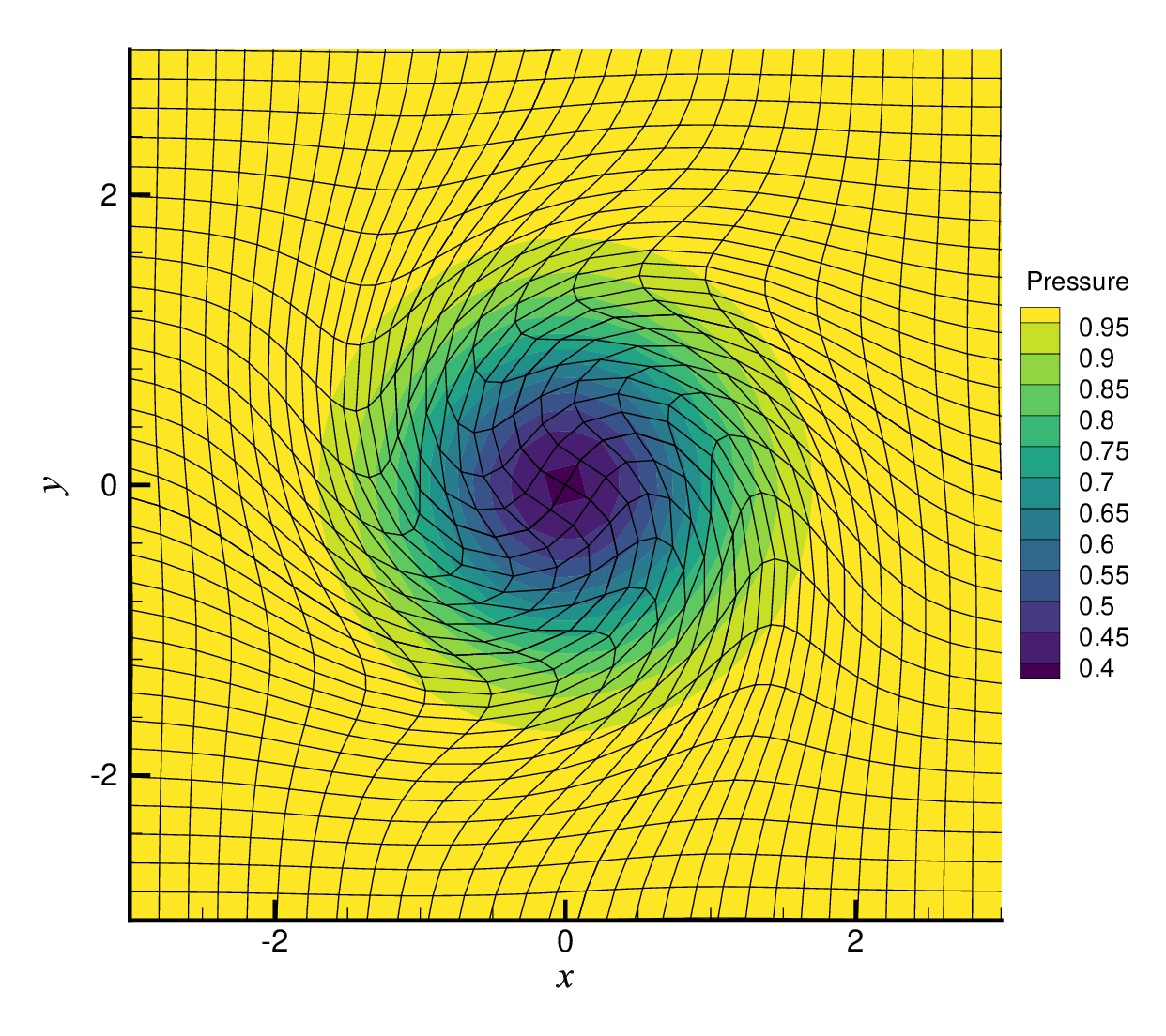}
	\end{subfigure}
	\begin{subfigure}{0.48\linewidth}
		\centering
		\includegraphics[width=\textwidth]{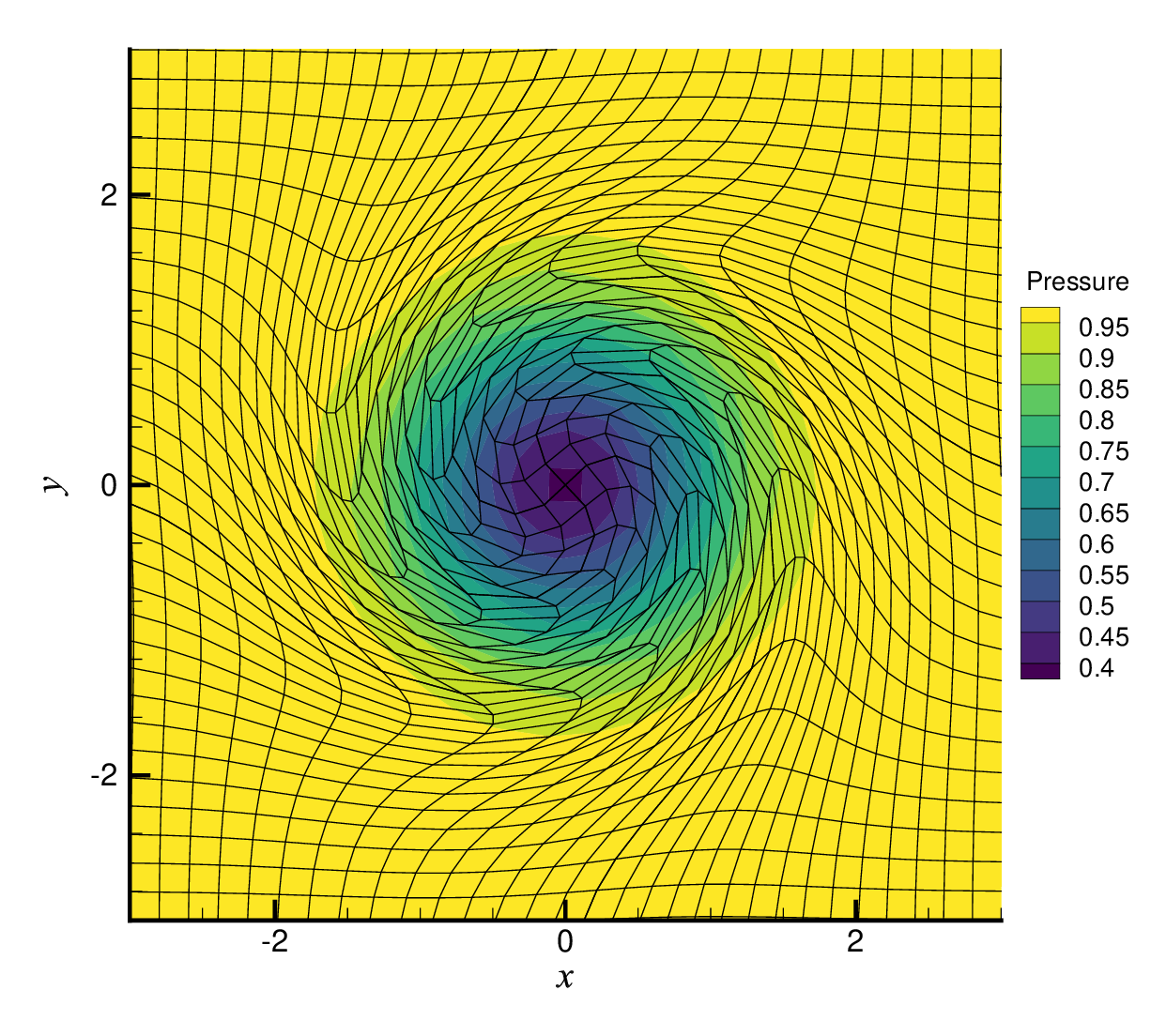}
	\end{subfigure}
	\caption{The close-up view of pressure contours with mesh using LMCV at $t=2$ (left) and $t=3$ (right) are shown for the isentropic vortex. The initial mesh are $100\times 100$ uniform mesh.}
	\label{Fig.T1}
\end{figure}

The order of accuracy of LMCV is assessed and shown in Table \ref{tab:acc}, where these tests are performed on a series of uniform initial meshes. It is worth noting that the convergence order increases as the mesh becomes finer, because larger cells accumulate more distortions over time, which magnify the error.
\begin{table}[H]
\centering
\caption{Numerical error and convergence order for the isentropic vortex at time $t=1$.}
\label{tab:acc}
\begin{small}
\begin{tabular}{lllllllll}
\hline
                & \multicolumn{2}{l}{Density} &  & \multicolumn{2}{l}{Momentum} &  & \multicolumn{2}{l}{Internal energy} \\ \cline{2-3} \cline{5-6} \cline{8-9} 
Mesh            & $L_2$ Error     & Order     &  & $L_2$ Error      & Order     &  & $L_2$ Error         & Order         \\ \hline
$50\times 50$   & 1.17E-03        & -         &  & 2.00E-03         & -         &  & 4.18E-03            & -             \\
$100\times 100$ & 1.64E-04        & 2.83      &  & 2.86E-04         & 2.81      &  & 5.79E-04            & 2.85          \\
$200\times 200$ & 1.99E-05        & 3.05      &  & 3.65E-05         & 2.97      &  & 6.91E-05            & 3.07          \\
$400\times 400$ & 2.28E-06        & 3.12      &  & 4.61E-06         & 2.98      &  & 7.62E-06            & 3.18          \\ \hline
\end{tabular}

\end{small}
\end{table}

\subsection{Taylor-Green vortex problem}
The Taylor-Green vortex \cite{vilar2012cell,burton2015reduction,dobrev2012high,morgan2015godunov}
is a smooth flow with initial field
\begin{eqnarray*}
\rho_0 = 1,\quad \bV_0 = (\sin(\pi x)\cos(\pi y),-\cos(\pi x)\sin(\pi y))^{\top},\quad P_0=\frac{1}{4}\brb{\cos(2\pi x)+\cos(2\pi y)}+1,
\end{eqnarray*}
which is steadied by a source term
\begin{eqnarray}
S = \bra{0,0,0,\frac{\pi}{4(\gamma - 1)}\brb{\cos(3\pi x)\cos(\pi y) - \cos(3\pi y)\cos(\pi x)}}^{\top}.
\end{eqnarray}
where $\gamma = 1.4$. The computational domain is $[0,1]^2$ with rigid boundaries.

Firstly, this test is carried out with $25\times 25$ uniform initial mesh.
As shown in Fig. \ref{Fig.TG}, cells close to boundaries are stretched over time. Nevertheless, maximum point of the pressure field are still located at corners consistent with exact solution at time $t=0.5$, but shift at $t=0.75$ as stretching becomes more severe. Similar result also appeared using other Lagrangian high order schemes \cite{Liu-subcellLDG}.
\begin{figure}[!htbp]
	\centering  
	\begin{subfigure}{0.48\linewidth}
		\centering
		\includegraphics[width=\textwidth]{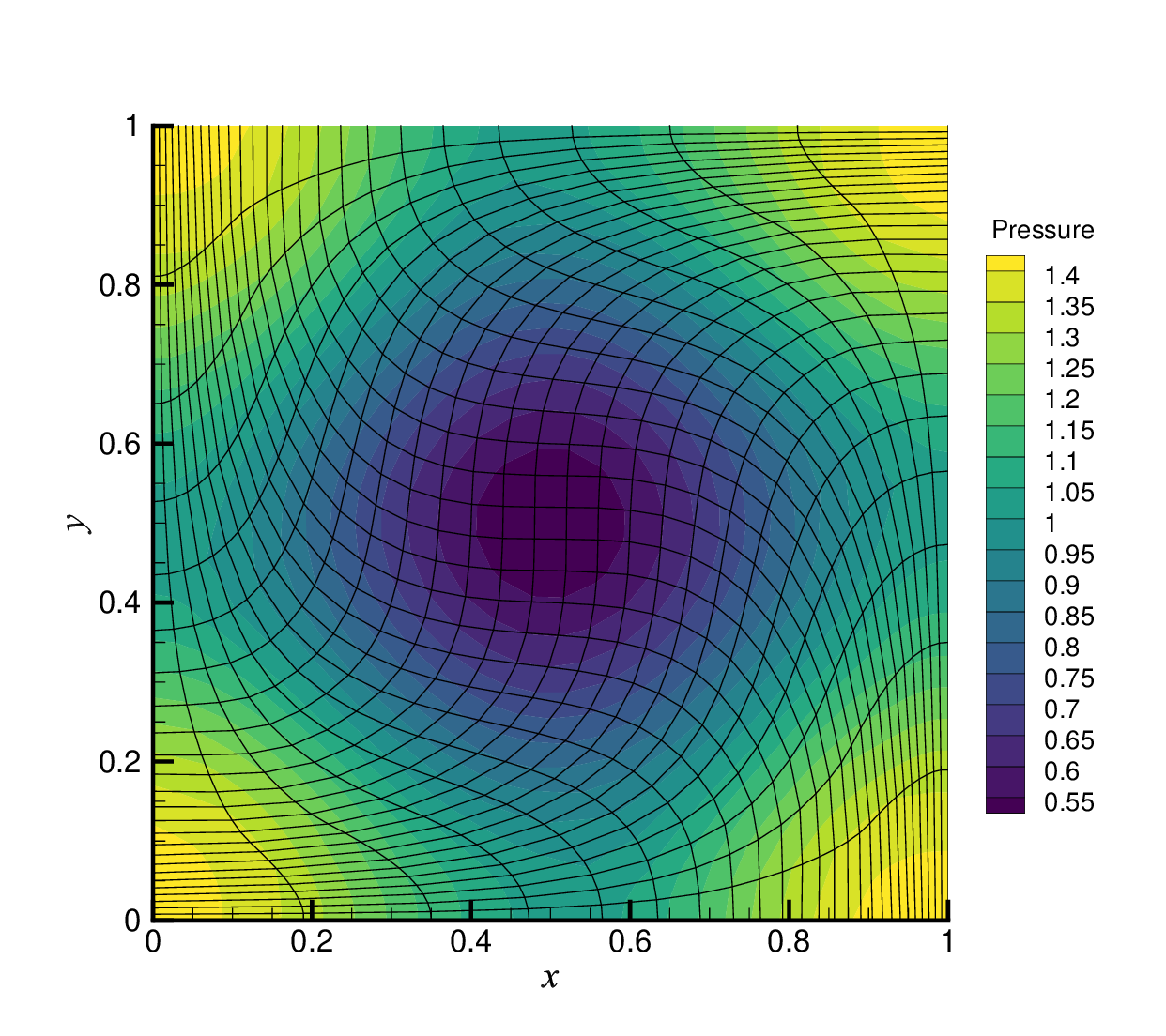}
	\end{subfigure}
	\begin{subfigure}{0.48\linewidth}
		\centering
		\includegraphics[width=\textwidth]{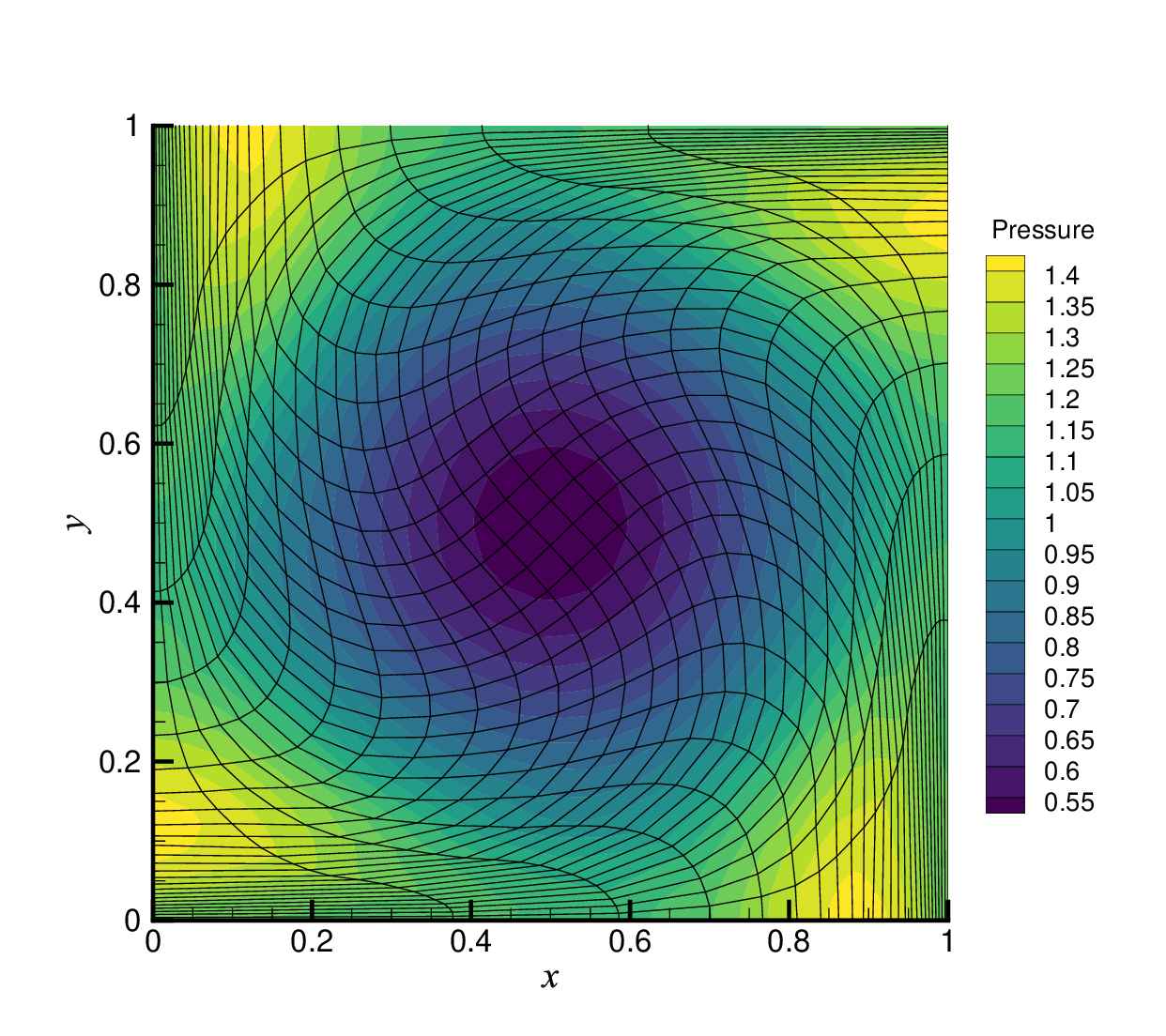}
	\end{subfigure}
	\caption{The global view of pressure contours with mesh using LMCV at $t=0.5$ (left) and $t=0.75$ (right) are shown for the Taylor-Green vortex. The initial mesh are $25\times 25$ uniform mesh. }
	\label{Fig.TG}
\end{figure}

In order to demonstrate the accuracy of LMCV on stretched meshes, two series of initial meshes $M_0(N)$ and $M_{\hf}(N)$ are introduced, where $M_0(N)$ is uniform mesh with size $N\times N$, and stretched mesh $M_{\hf}(N)$ is obtained by appling velocity field of Taylor-Green vortex to $M_0(N)$ for 0.5 time unit. The results of accuracy tests on both series are shown in Table \ref{tab:acc_tg1} and \ref{tab:acc_tg2}, respectively. Comparing the two tables, it can be seen that the mesh distortion amlifies numerical error, but our method yield the expected third-order accuracy for all variables, even on highly distorted meshes.

\begin{table}[H]
\centering
\caption{Numerical error and convergence order for Taylor-Green vortex at time $t=0.1$ with $M_0$ as initial meshes.}
\label{tab:acc_tg1}
\begin{small}
\begin{tabular}{lllllllll}
\hline
                & \multicolumn{2}{l}{Density} &  & \multicolumn{2}{l}{Momentum} &  & \multicolumn{2}{l}{Internal energy} \\ \cline{2-3} \cline{5-6} \cline{8-9} 
Mesh            & $L_2$ Error     & Order     &  & $L_2$ Error      & Order     &  & $L_2$ Error         & Order         \\ \hline
$25\times 25$   & 4.45E-05        & -         &  & 7.01E-05         & -         &  & 1.34E-04            & -             \\
$50\times 50$   & 5.05E-06        & 3.14      &  & 8.39E-06         & 3.06      &  & 1.51E-05            & 3.15          \\
$100\times 100$ & 5.41E-07        & 3.22      &  & 9.89E-07         & 3.08      &  & 1.82E-06            & 3.05          \\
$200\times 200$ & 6.82E-08        & 2.99      &  & 1.23E-07         & 3.01      &  & 2.30E-07            & 2.98          \\ \hline
\end{tabular}

\end{small}
\end{table}
\begin{table}[H]
\centering
\caption{Numerical error and convergence order for Taylor-Green vortex at time $t=0.1$ with $M_{\hf}$ as initial meshes.}
\label{tab:acc_tg2}
\begin{small}
\begin{tabular}{lllllllll}
\hline
                & \multicolumn{2}{l}{Density} &  & \multicolumn{2}{l}{Momentum} &  & \multicolumn{2}{l}{Internal energy} \\ \cline{2-3} \cline{5-6} \cline{8-9} 
Mesh            & $L_2$ Error     & Order     &  & $L_2$ Error      & Order     &  & $L_2$ Error         & Order         \\ \hline
$25\times 25$   & 6.63E-04        & -         &  & 1.95E-03         & -         &  & 2.52E-03            & -             \\
$50\times 50$   & 1.00E-04        & 2.73      &  & 2.52E-04         & 2.95      &  & 3.98E-04            & 2.66          \\
$100\times 100$ & 1.27E-05        & 2.98      &  & 3.15E-05         & 3.00      &  & 5.10E-05            & 2.97          \\
$200\times 200$ & 1.63E-06        & 2.97      &  & 3.91E-06         & 3.01      &  & 6.56E-06            & 2.96          \\ \hline
\end{tabular}

\end{small}
\end{table}

\subsection{Gresho vortex problem}
The Gresho vortex \cite{Vilar-LDG} 
is a steady smooth vortical flow similar to the isentropic vortex. The initial field is
\begin{eqnarray*}
\rho_0 = 1, \quad
\bV_0 = \frac{v_0}{r} g_v\bra{\frac{r}{r_v}} (-y,x)^\top, \quad
P_0 = P_0 + v_0^2 \, h_v\bra{\frac{r}{r_v}},
\end{eqnarray*}
where $\gamma = 1.4$, $r = \sqrt{x^2+y^2}$, $v_0 = 1$, $P_0 = 5$ and the radius of the vortex $r_v=0.4$. The shape functions $g_v$ and $h_v$ are defined as
\begin{eqnarray*}
g_v(s) = \begin{cases}
2^{2n}s^n(1-s)^n, & s \le 1, \\
0,                & s > 1,
\end{cases} \quad
h_v(s) = \int_{0}^{s} \frac{g_v(\xi)^2}{\xi} \mathrm{d} \xi,
\end{eqnarray*}
with $n=6$. Unlike the isentropic vortex, the computational domain for this test case is a disk with $r\le 0.52$ with Dirichlet boundaries. In order to discretize the domain with quadrilateral cells, uniform polar mesh with $N_r \times N_\theta$ size is introduced here, divides the disk into $N_r$ parts equally along the radius and $N_{\theta}$ parts equally along the angular direction. Small arcs in each cell are approximated by line segments with the same endpoints, which makes accuracy test more challenging. In addition, a small hole with radius $10^{-6}$ is placed at the origin to avoid degenerate quadrilateral cells.

First, this problem is solved by LMCV on mesh of $36 \times 36$ size, and the pressure contours at time $t=0.2$ and $t=0.5$ are shown in Fig. \ref{Fig.Gresho}. Our method gives reliable results even when cells are so distorted that they tend to self-intersect at $t=0.5$. Soon after this, self-intersecting cells are inevitably appear, making the solving procedure unable to continue.
\begin{figure}[!htbp]
	\centering  
	\begin{subfigure}{0.43\linewidth}
		\centering
		\includegraphics[width=\textwidth]{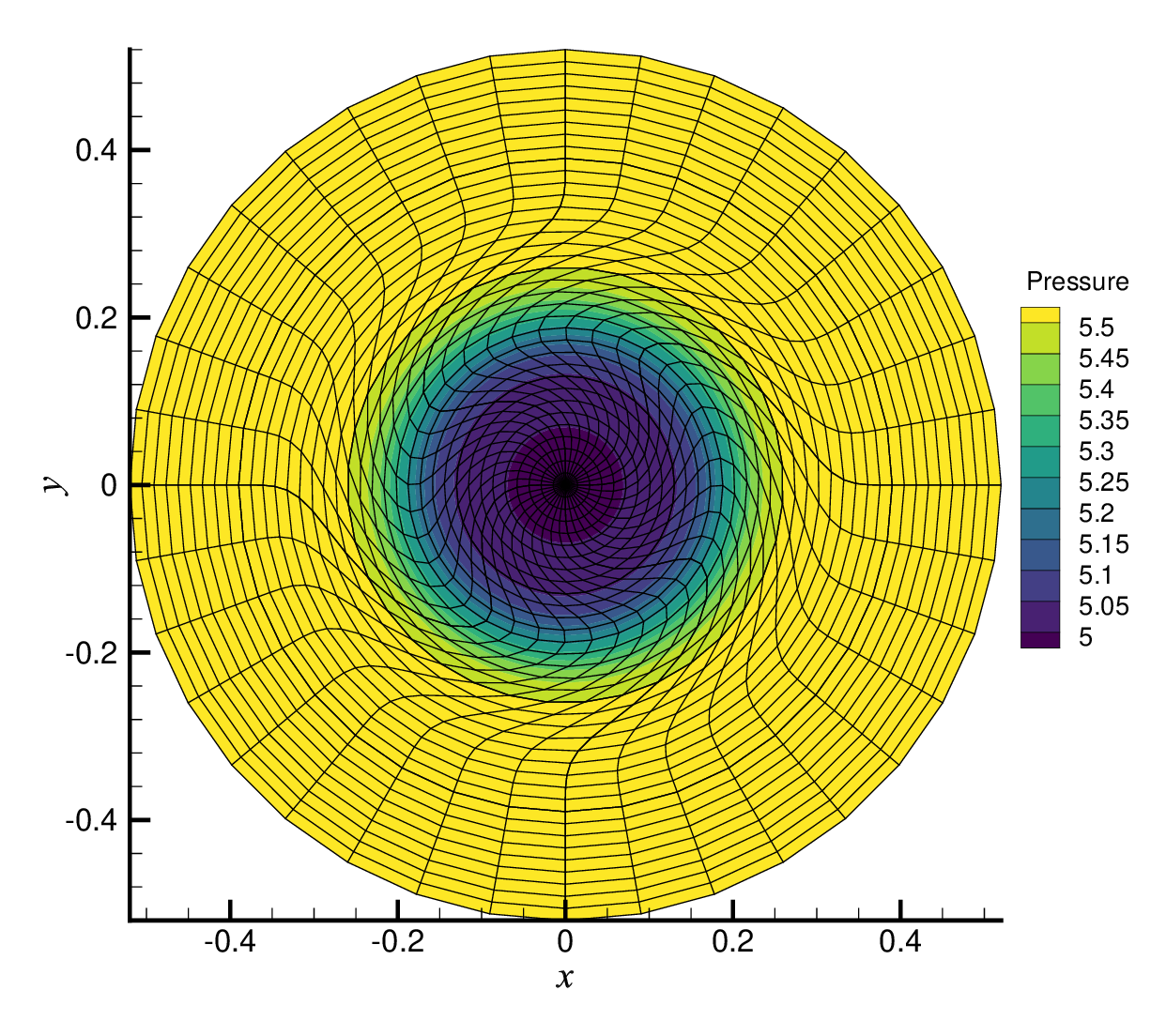}
	\end{subfigure}
	\begin{subfigure}{0.43\linewidth}
		\centering
		\includegraphics[width=\textwidth]{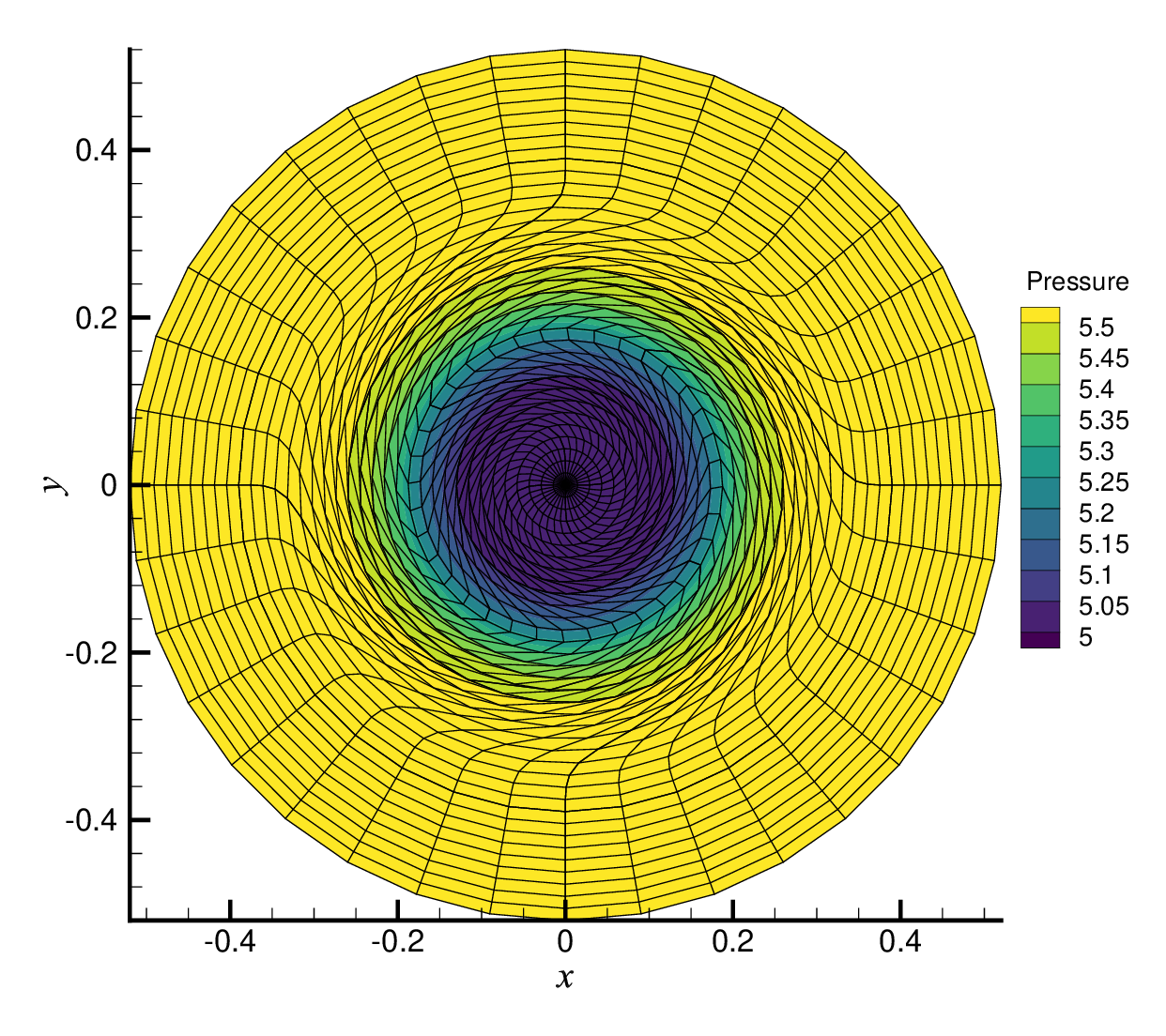}
	\end{subfigure} \\
	\begin{subfigure}{0.43\linewidth}
		\centering
		\includegraphics[width=\textwidth]{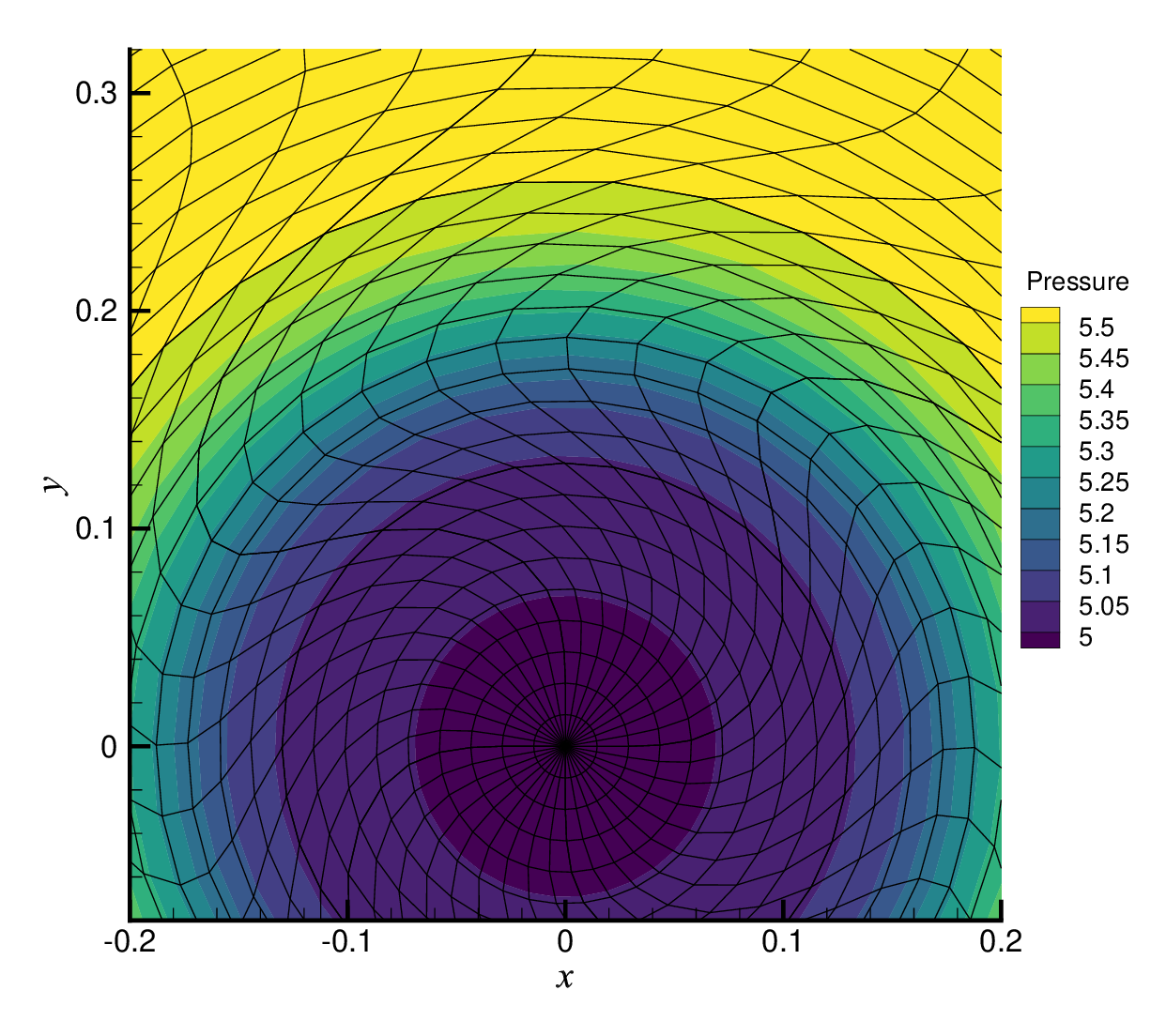}
	\end{subfigure}
	\begin{subfigure}{0.43\linewidth}
		\centering
		\includegraphics[width=\textwidth]{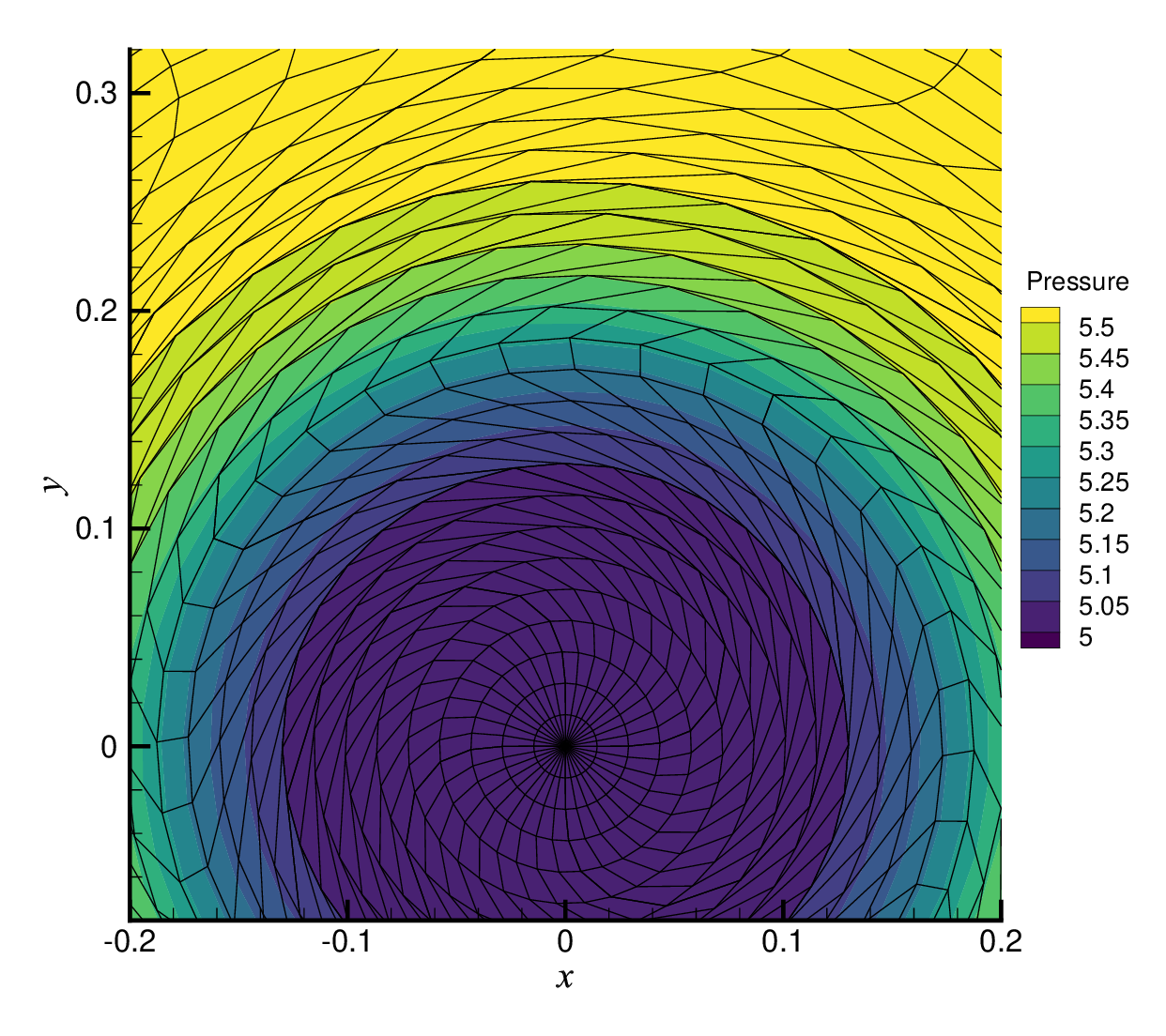}
	\end{subfigure} \\
	\caption{The global view (top) and close-up view (below) of pressure contours with mesh using LMCV at $t=0.2$ (left) and $t=0.5$ (right) are shown for Gresho vortex with mesh of $36\times 36$ size. }
	\label{Fig.Gresho}
\end{figure}

Then, the numerical error and order of accuracy for Gresho vrotex is presented in Table \ref{tab:acc_gresho}. Although the polar coordinate transformation is not used here to remove the nonliearity of the velocity field, the third-order convergence rate is delivered as expected.
\begin{table}[H]
\centering
\caption{Numerical error and convergence order for Gresho vortex at time $t=0.1$.}
\label{tab:acc_gresho}
\begin{small}
\begin{tabular}{lllllllll}
\hline
                & \multicolumn{2}{l}{Density} &  & \multicolumn{2}{l}{Momentum} &  & \multicolumn{2}{l}{Internal energy} \\ \cline{2-3} \cline{5-6} \cline{8-9} 
Mesh            & $L_2$ Error     & Order     &  & $L_2$ Error      & Order     &  & $L_2$ Error         & Order         \\ \hline
$25\times 25$   & 5.25E-04        & -         &  & 1.18E-02         & -         &  & 1.15E-02            & -             \\
$50\times 50$   & 7.85E-05        & 2.74      &  & 1.87E-03         & 2.65      &  & 1.68E-03            & 2.78          \\
$100\times 100$ & 1.01E-05        & 2.95      &  & 2.50E-04         & 2.91      &  & 2.15E-04            & 2.96          \\
$200\times 200$ & 1.26E-06        & 3.00      &  & 3.17E-05         & 2.98      &  & 2.68E-05            & 3.00          \\ \hline
\end{tabular}

\end{small}
\end{table}

\subsection{Sod problem}
This well-known problem \cite{Sod1} consists of a shock tube of unity length. The computational domain is rigid box $[0,1]^2$ filled with two kind of static gas where the interface is located at $x=0.5$. The gas on the left is characterized by density $\rho_L = 1$ and pressure $P_L = 1$, while the right state is defined by $\rho_R = 0.125$ and $P_R = 0.1$. Both sides follow gamma gas law with $\gamma = 1.4$. The initial mesh is a uniform Cartesian grid of $100\times 2$ size. The numerical results obtained with EUCCLHYD scheme (labeled as FV) and LMCV are presented in Fig. \ref{Fig.Sod} as density distributions, which show the compatibility with 1D case and the improvement of high-order solution. The limiting procedure with $S_c=50$ is used here.
\begin{figure}[!htbp]
	\centering  
	\begin{subfigure}{0.43\linewidth}
		\centering
		\includegraphics[width=\textwidth]{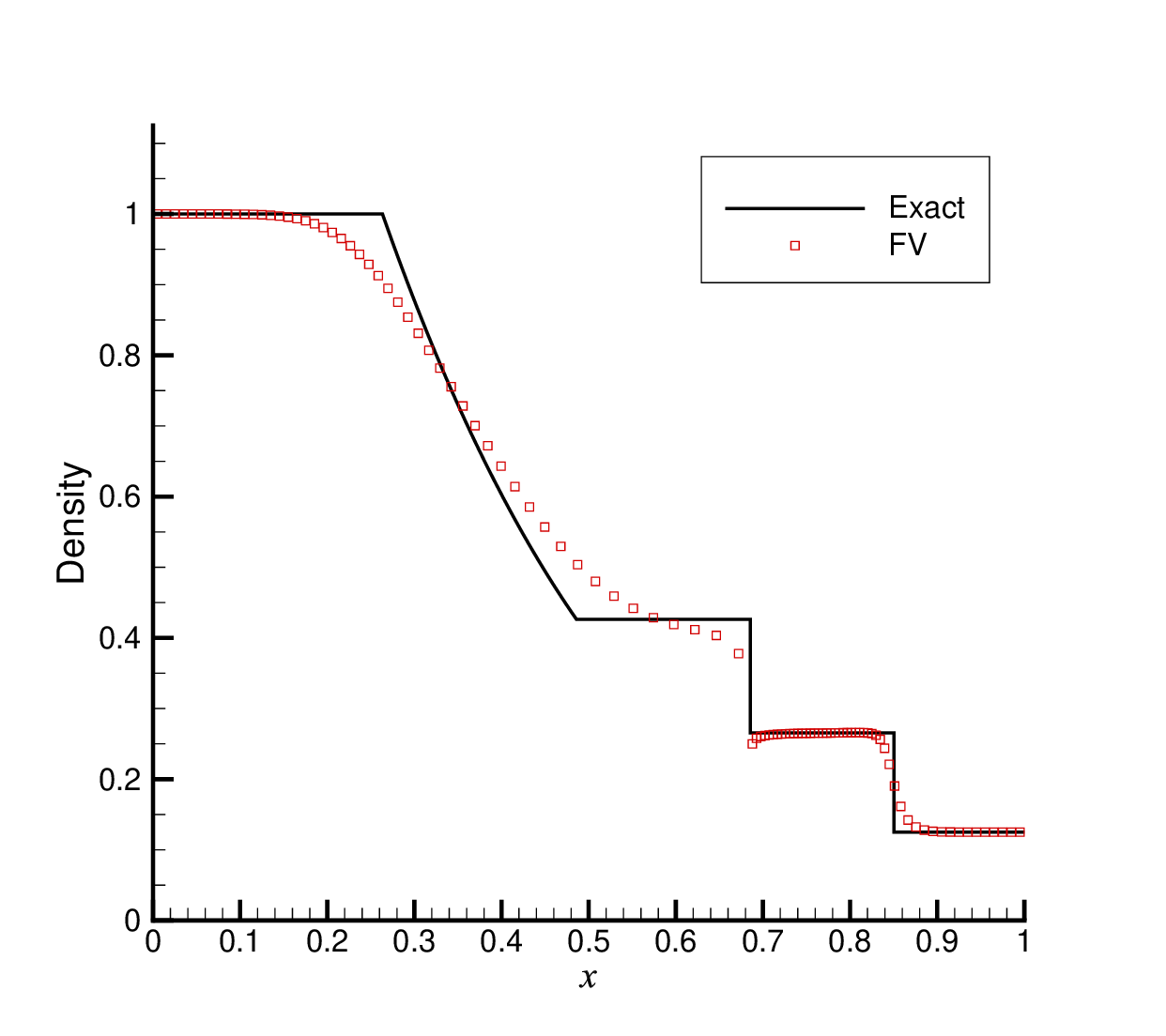}
	\end{subfigure}
	\begin{subfigure}{0.43\linewidth}
		\centering
		\includegraphics[width=\textwidth]{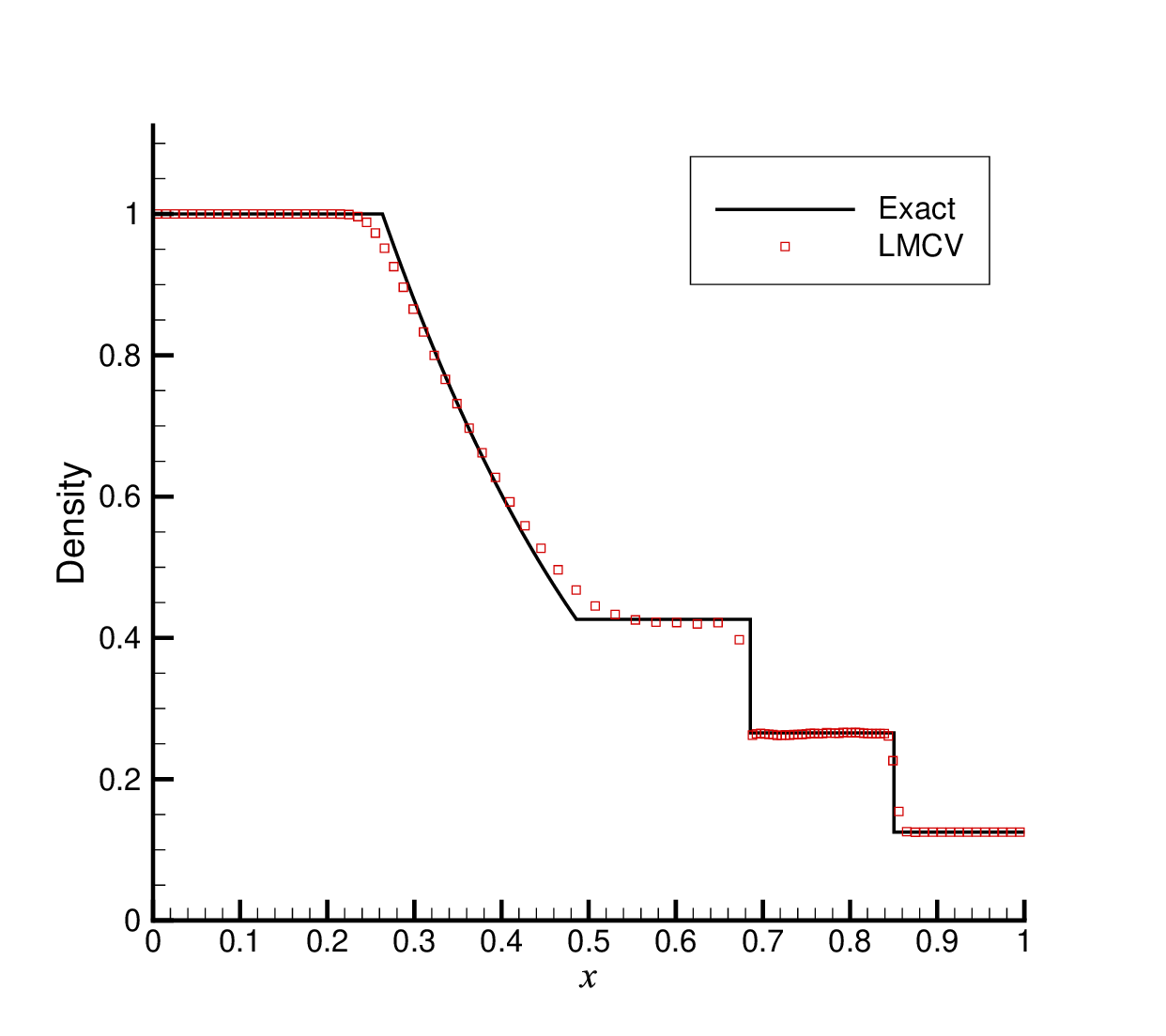}
	\end{subfigure}
	\caption{The scatter plots of density at $t=0.2$ for the Sod problem with EUCCLHYD (left) and LMCV (right). The black solid line represents the exact density distribution. }
	\label{Fig.Sod}
\end{figure}

\subsection{Sedov blast problem}
The Sedov test case \cite{sedov2018similarity,pederson2016sedov}
describes an outward traveling blast wave that is triggered by a point energy source. The 2D case is initialized with background as
\begin{eqnarray}
\rho_0 = 1, \quad \bV_0 = \boldsymbol{0},\quad P_0 = 10^{-6}.
\end{eqnarray}
To place the energy source, the pressure $P_0$ of cell $\omega^i$ containing the origin is set to $(\gamma - 1) \rho_0 \frac{E_0}{|\omega^i|}$, where $\gamma = 1.4$ and total energy $E_0 = 0.244816$ so that the shock front is located at $r=1$ when $t=1$ in the exact solution. The computational domain is a square $[0,1.2]^2$ with all rigid boundaries. The limiting procedure of LMCV is applied in this test case with $S_c = 10^5$ since the presence of shock waves. Initially uniform mesh is used here with $50\times 50$ resolution, and the final time is $t=1$. The final meshes and density contours are shown in Fig. \ref{Fig.Sedov-mesh}. For comparison purpose, the results from EUCCLHYD scheme (labeled as FV) are also demonstrated. 
\begin{figure}[!htbp]
	\centering  
	\begin{subfigure}{0.43\linewidth}
		\centering
		\includegraphics[width=\textwidth]{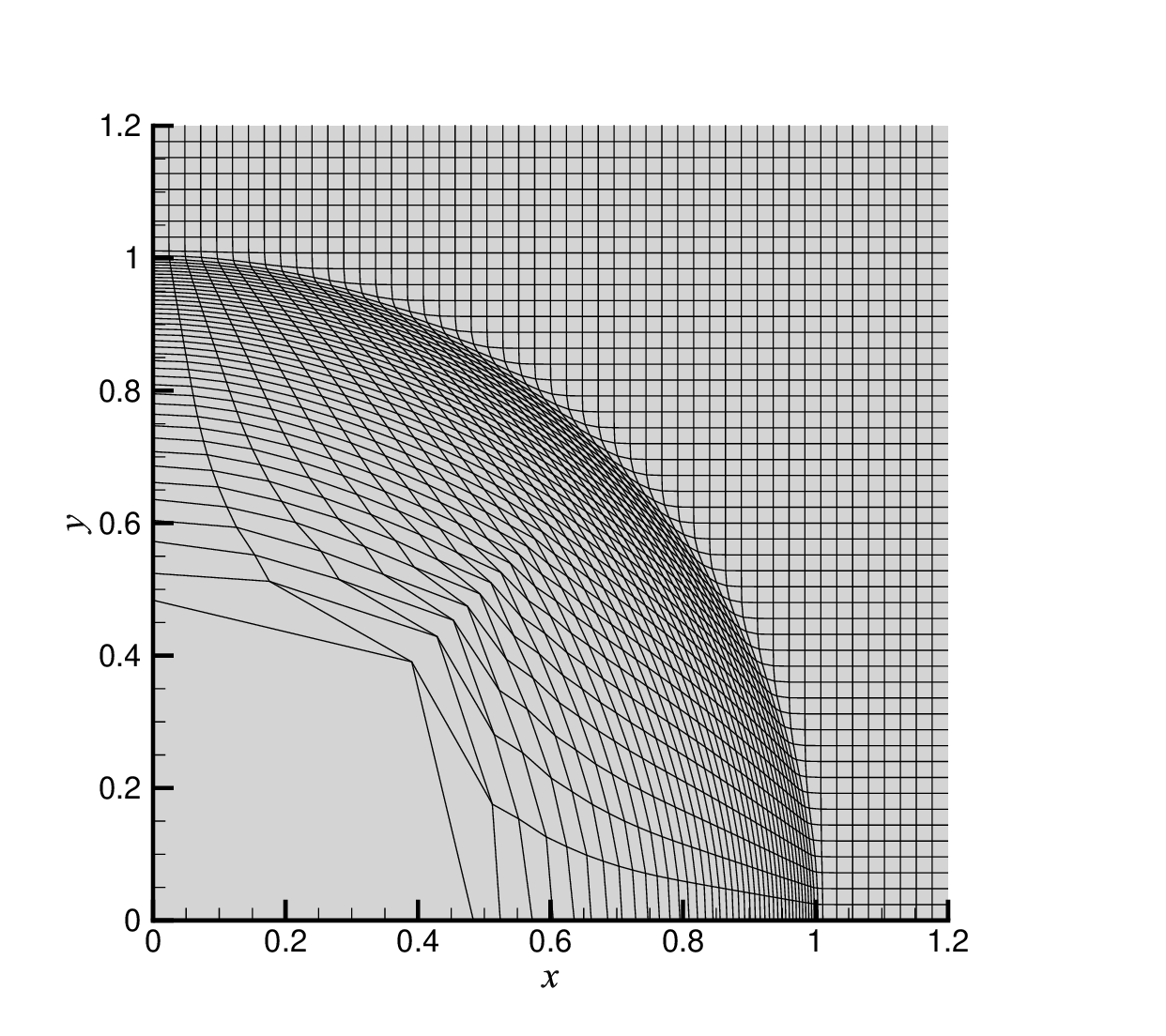}
		\caption{The final mesh for FV}
	\end{subfigure}
	\begin{subfigure}{0.43\linewidth}
		\centering
		\includegraphics[width=\textwidth]{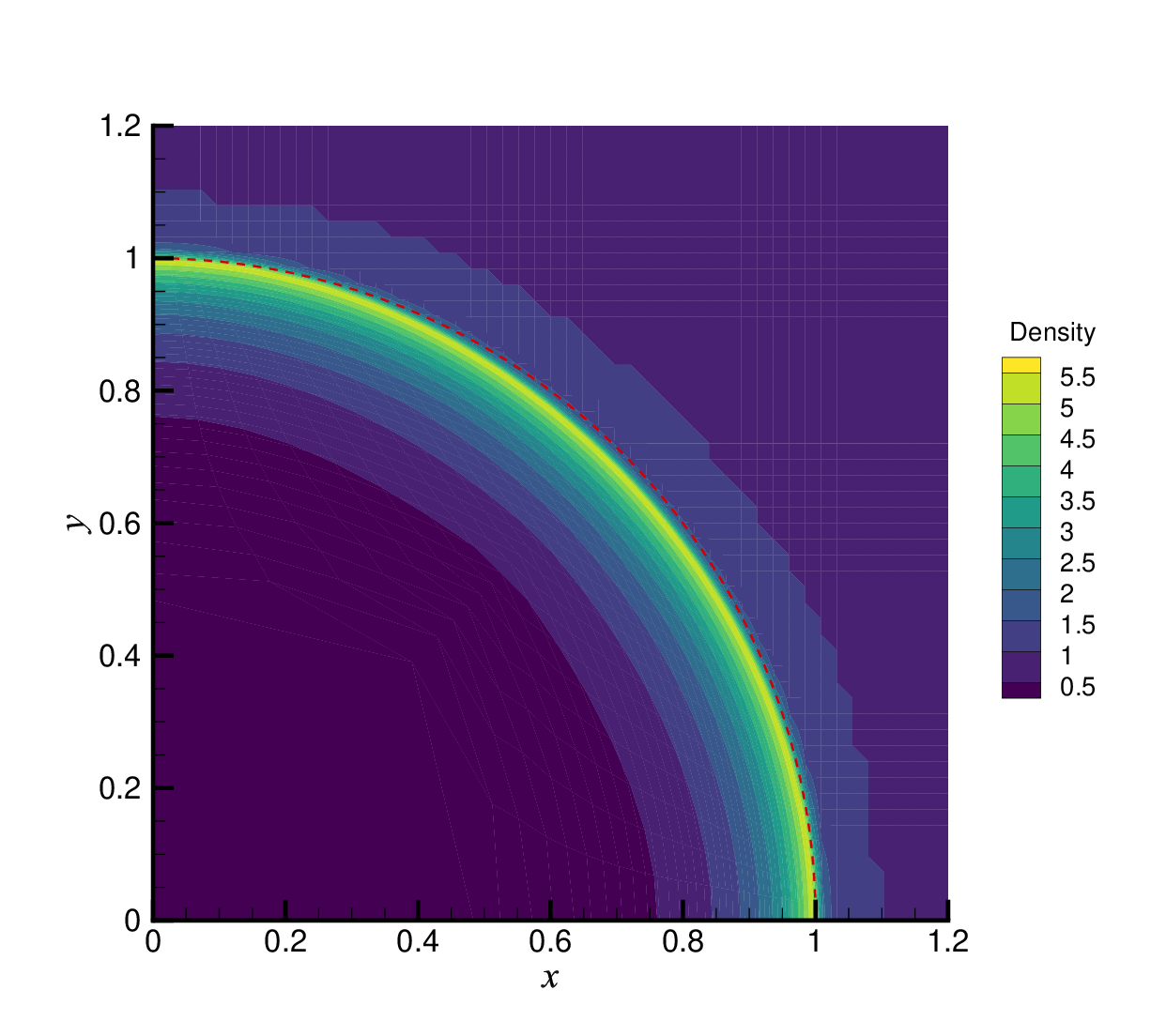}
		\caption{The final density contours for FV}
	\end{subfigure}
	\\
	\begin{subfigure}{0.43\linewidth}
		\centering
		\includegraphics[width=\textwidth]{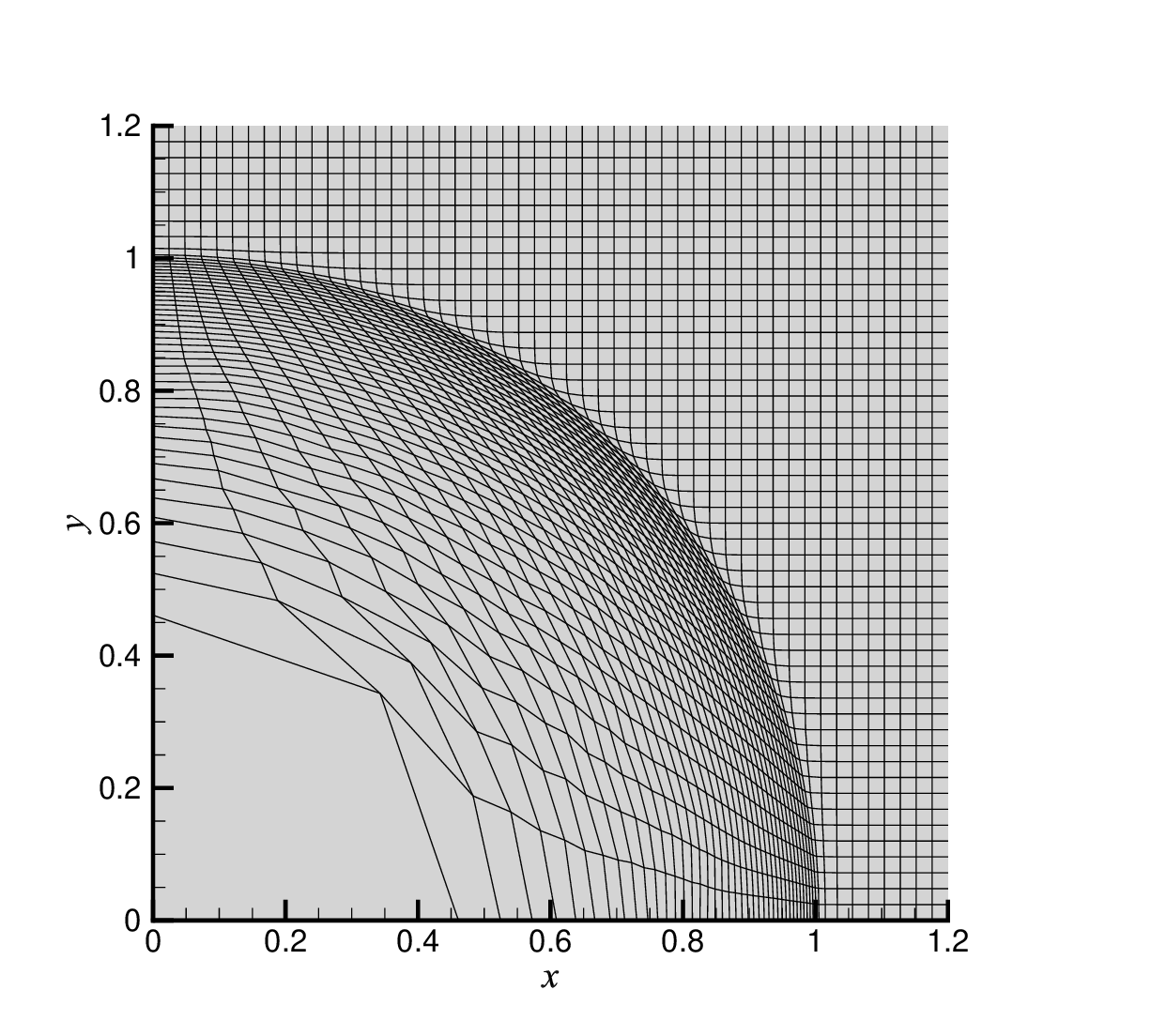}
		\caption{The final mesh for LMCV}
	\end{subfigure}
	\begin{subfigure}{0.43\linewidth}
		\centering
		\includegraphics[width=\textwidth]{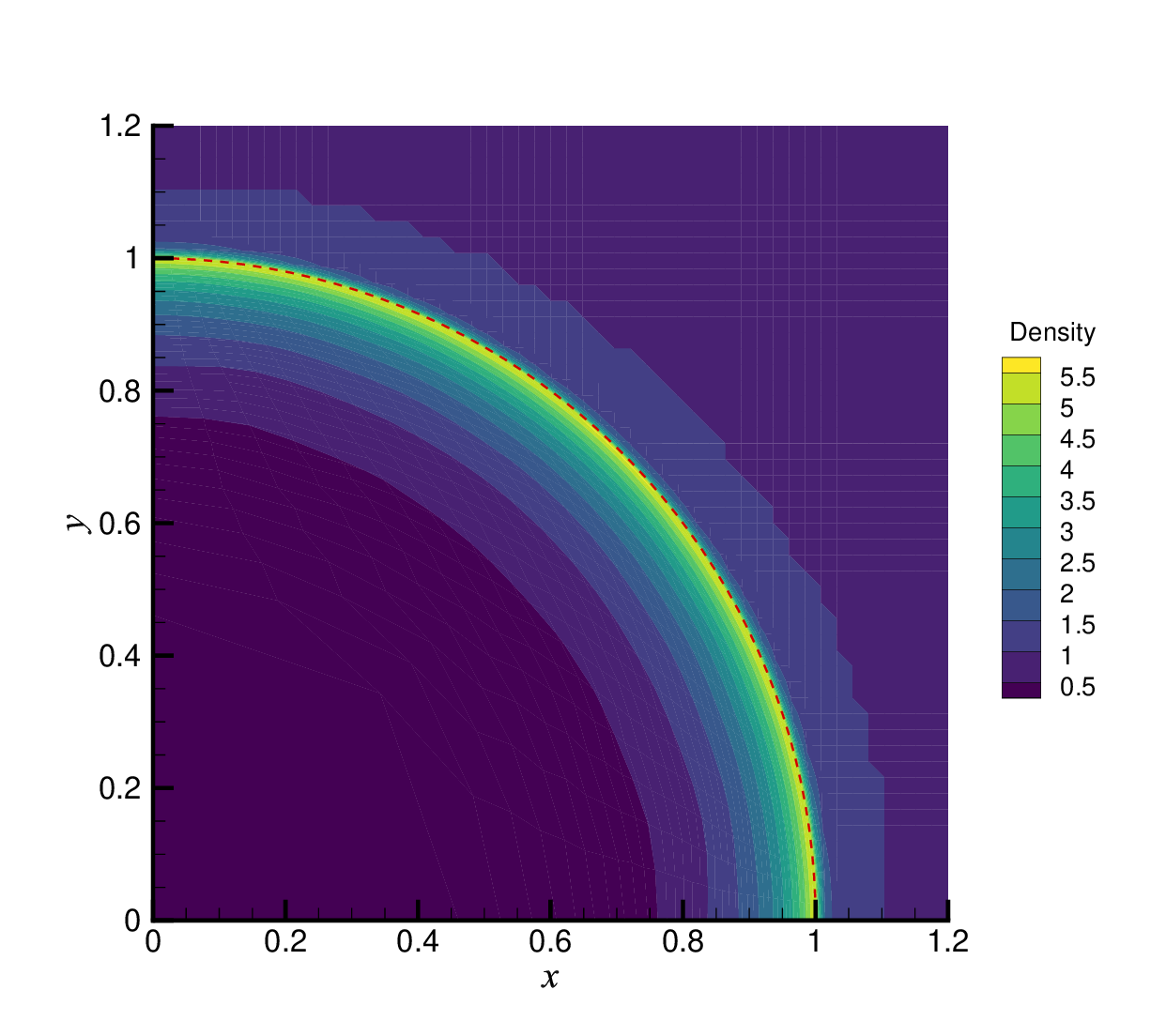}
		\caption{The final density contours for LMCV}
	\end{subfigure}  \\
	\caption{The meshes and density contours for the Sedov blast problem at time $t=1$ with initially uniform mesh of $50\times 50$ size. The red dashed line is the theoretical shock front location. }
	\label{Fig.Sedov-mesh}
\end{figure}

Further more, the scatter plots of density at the final time are shown in Fig. \ref{Fig.Sedov-rad}, in which the exact solution \cite{sedov2018similarity} is drawn with soild line. With comparison, it can be seen that the radial symmetry and sharpness of shock wave are well-preserved with LMCV, which highlights the accuracy of our method without losing robustness in the face of discontinuities.

\begin{figure}[!htbp]
	\centering  
	\begin{subfigure}{0.43\linewidth}
		\centering
		\includegraphics[width=\textwidth]{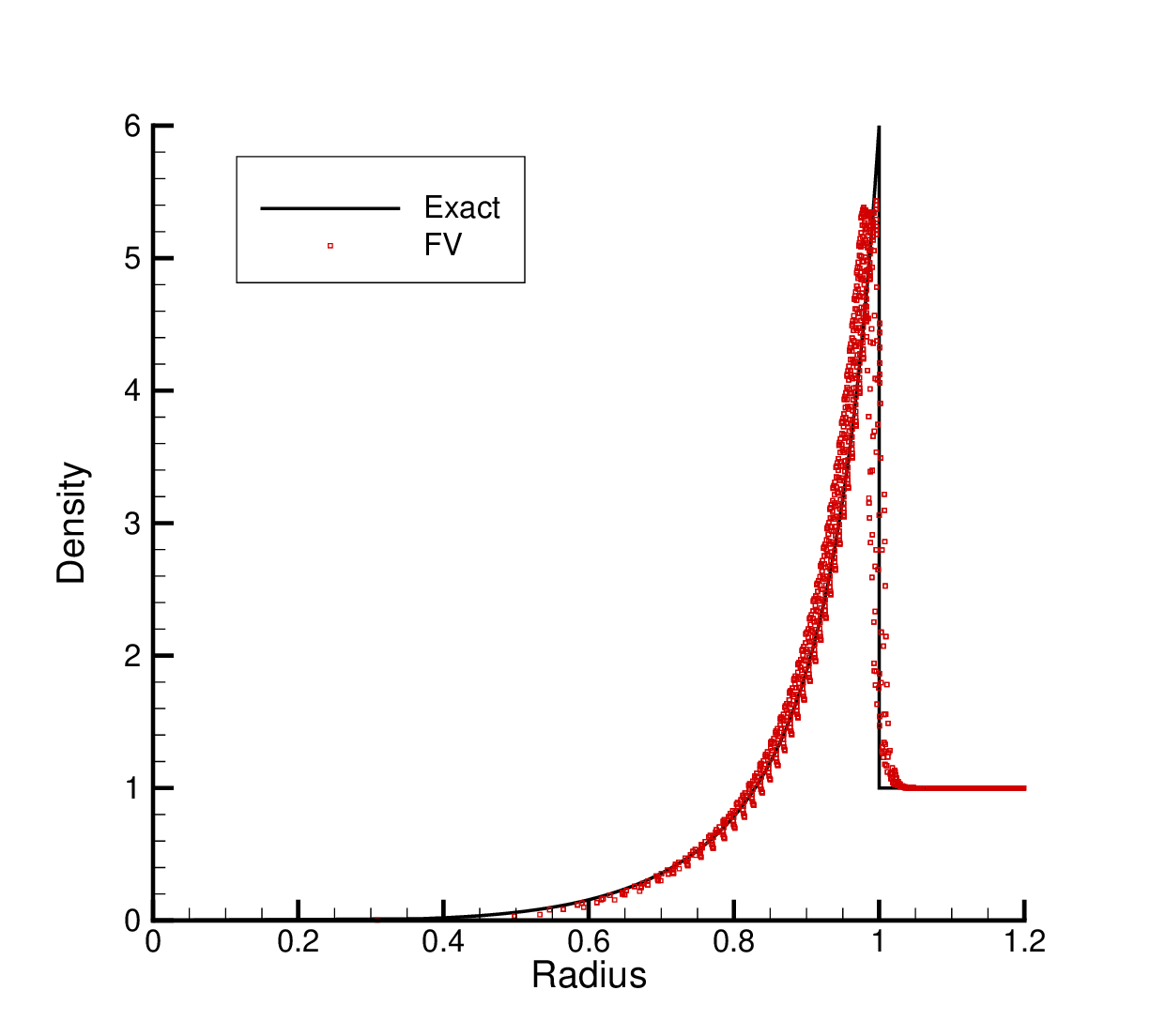}
	\end{subfigure}
	\begin{subfigure}{0.43\linewidth}
		\centering
		\includegraphics[width=\textwidth]{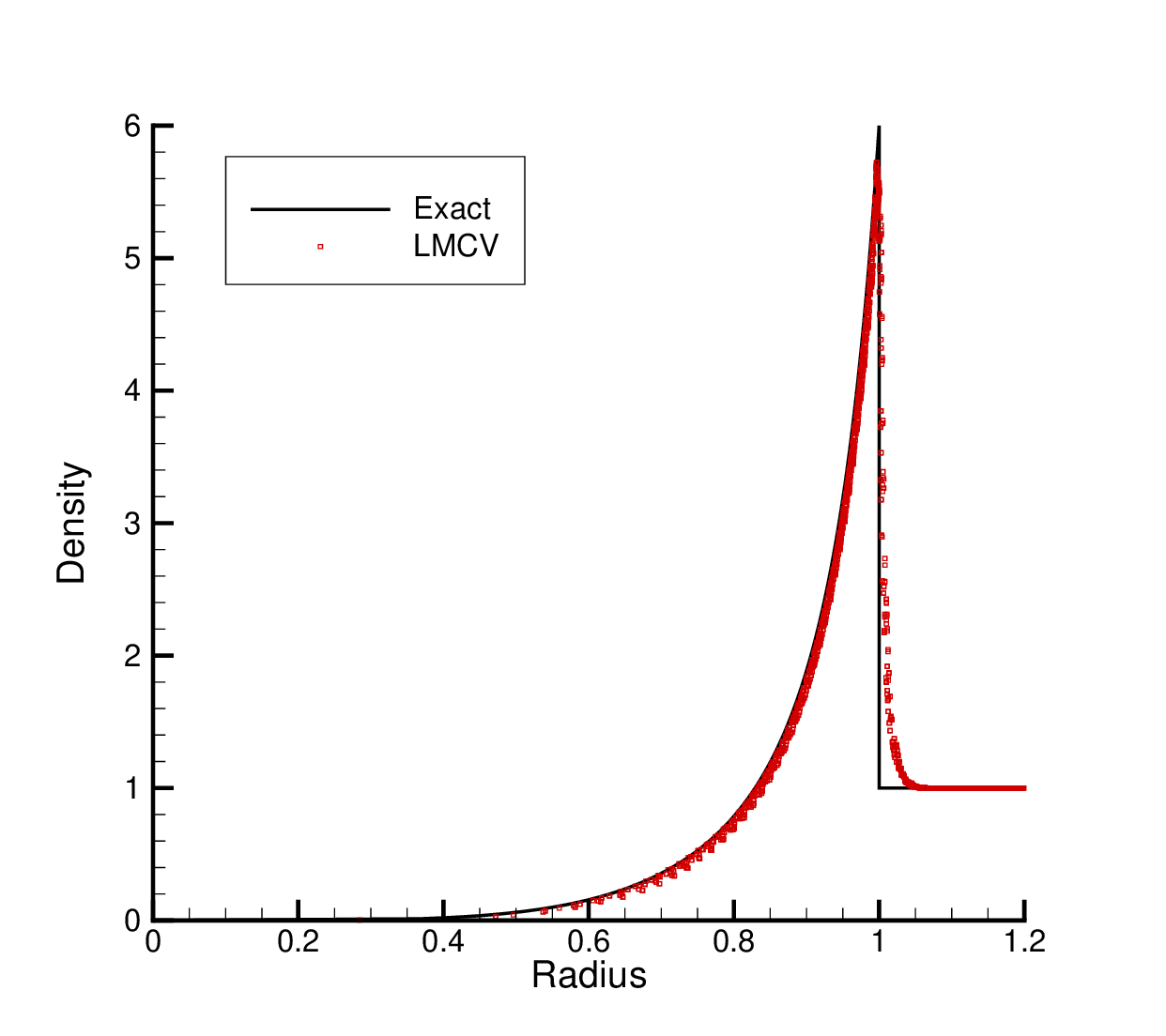}
	\end{subfigure}
	\caption{The scatter plots of density at $t=1$ for the Sedov blast problem with EUCCLHYD (left) and LMCV (right). The black solid line represents the exact density distribution from \cite{sedov2018similarity}. }
	\label{Fig.Sedov-rad}
\end{figure}

\subsection{Noh problem}
Similar to Sedov blast problem, Noh problem \cite{Noh1987} is also often used to test the robustness of Lagrangian schemes. It contains a shock caused by the convergence of a uniform gas towards the origin. More specifically, the ideal gas with $\gamma = 5/3$ is initialized as
\begin{eqnarray*}
\rho_0 = 1,\quad \bV_0 = \frac{u_0}{\sqrt{x^2+y^2}}(-x,-y)^{\top},\quad P_0 = 10^{-6}.
\end{eqnarray*} with convergence velocity $u_0 = 1$.
The shock wave propagates radially outward from the origin at a speed of $\hf (\gamma-1) u_0 = \frac{1}{3}$, leaves a steady platform with a density of $\rho_0\bra{\frac{\gamma+1}{\gamma-1}}^2 = 16$ behind. In this test, the computational domain is defined as $[0,1]^2$ and discretized by uniform mesh of size $50 \times 50$. The presence of shock requires the limiting procedure with $S_c = 10^5$. Results at time $t=0.6$ given by EUCCLHYD and LMCV are shown in Fig. \ref{Fig.Noh-mesh}-\ref{Fig.Noh-rad}. From comparison in Fig. \ref{Fig.Noh-rad}, LMCV gives more accurate density distribution with less scatter along different radial directions. In Fig. \ref{Fig.Noh-local}, our method avoids the false mesh distortion \cite{vilarPositivity2016P2} 
near the middle of the shock front.

\begin{figure}[!htbp]
	\centering  
	\begin{subfigure}{0.48\linewidth}
		\centering
		\includegraphics[width=\textwidth]{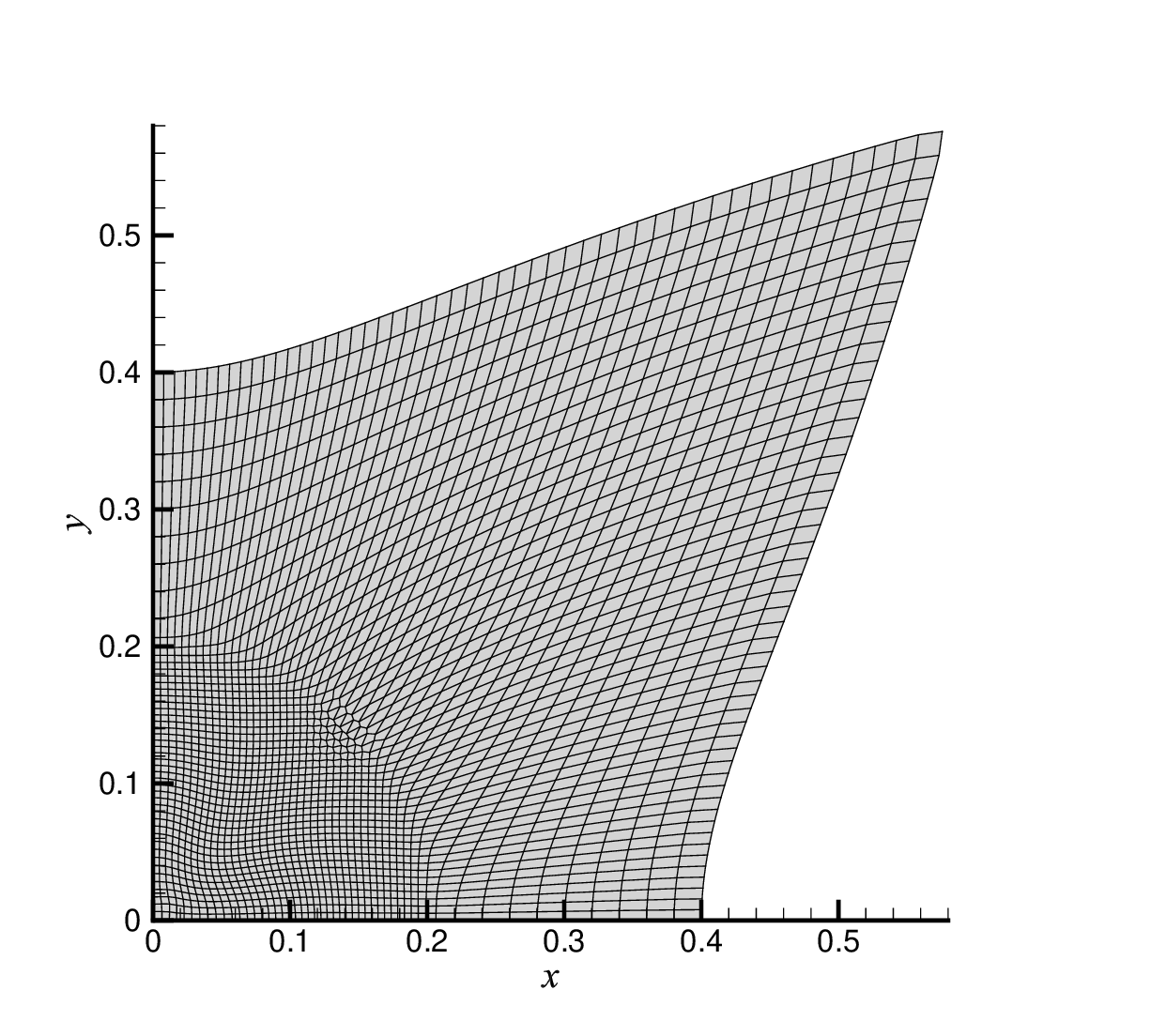}
		\caption{The final mesh for FV}
	\end{subfigure}
	\begin{subfigure}{0.48\linewidth}
		\centering
		\includegraphics[width=\textwidth]{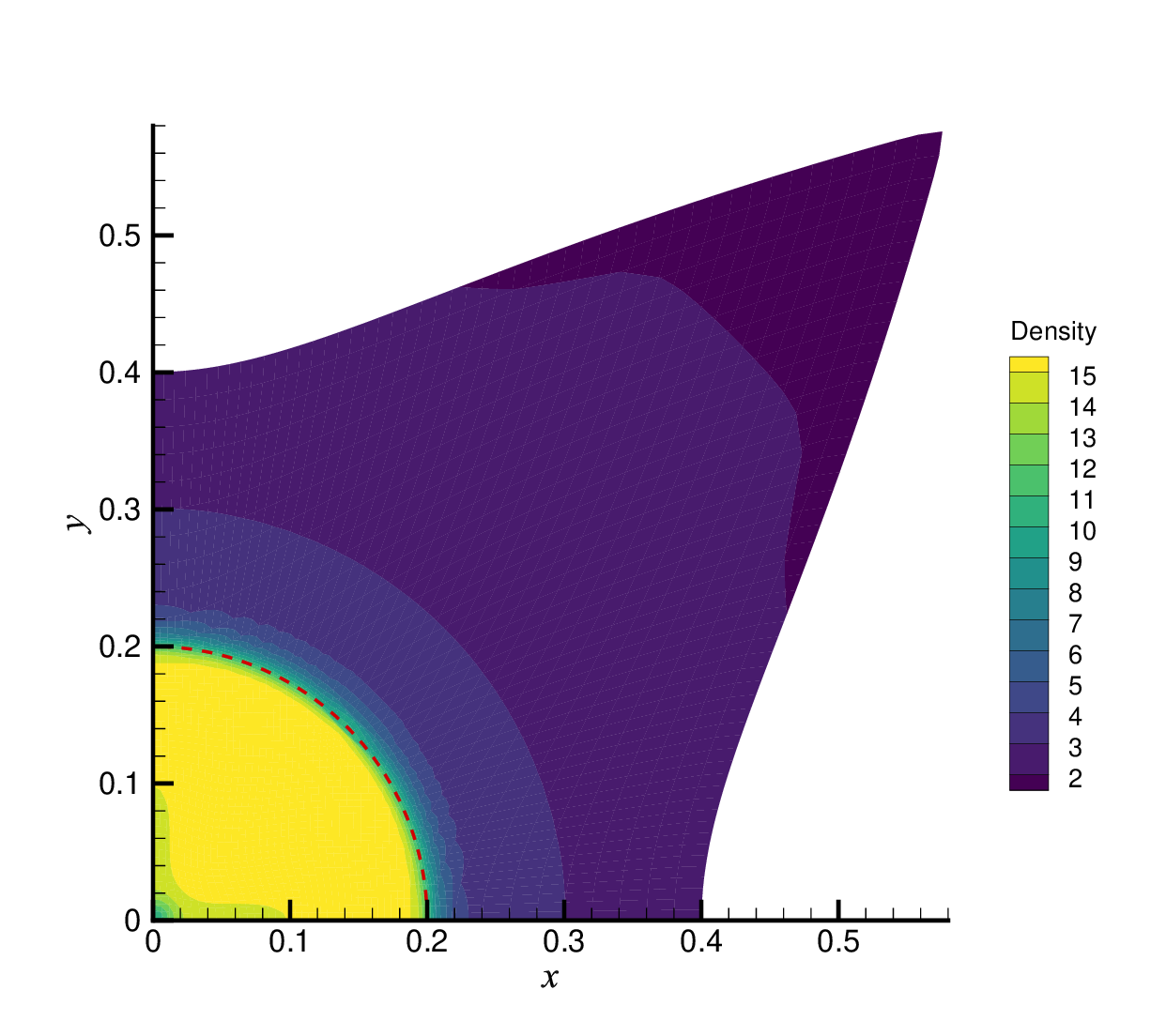}
		\caption{The final density contours for FV}
	\end{subfigure}
	 \\
	 \begin{subfigure}{0.48\linewidth}
		\centering
		\includegraphics[width=\textwidth]{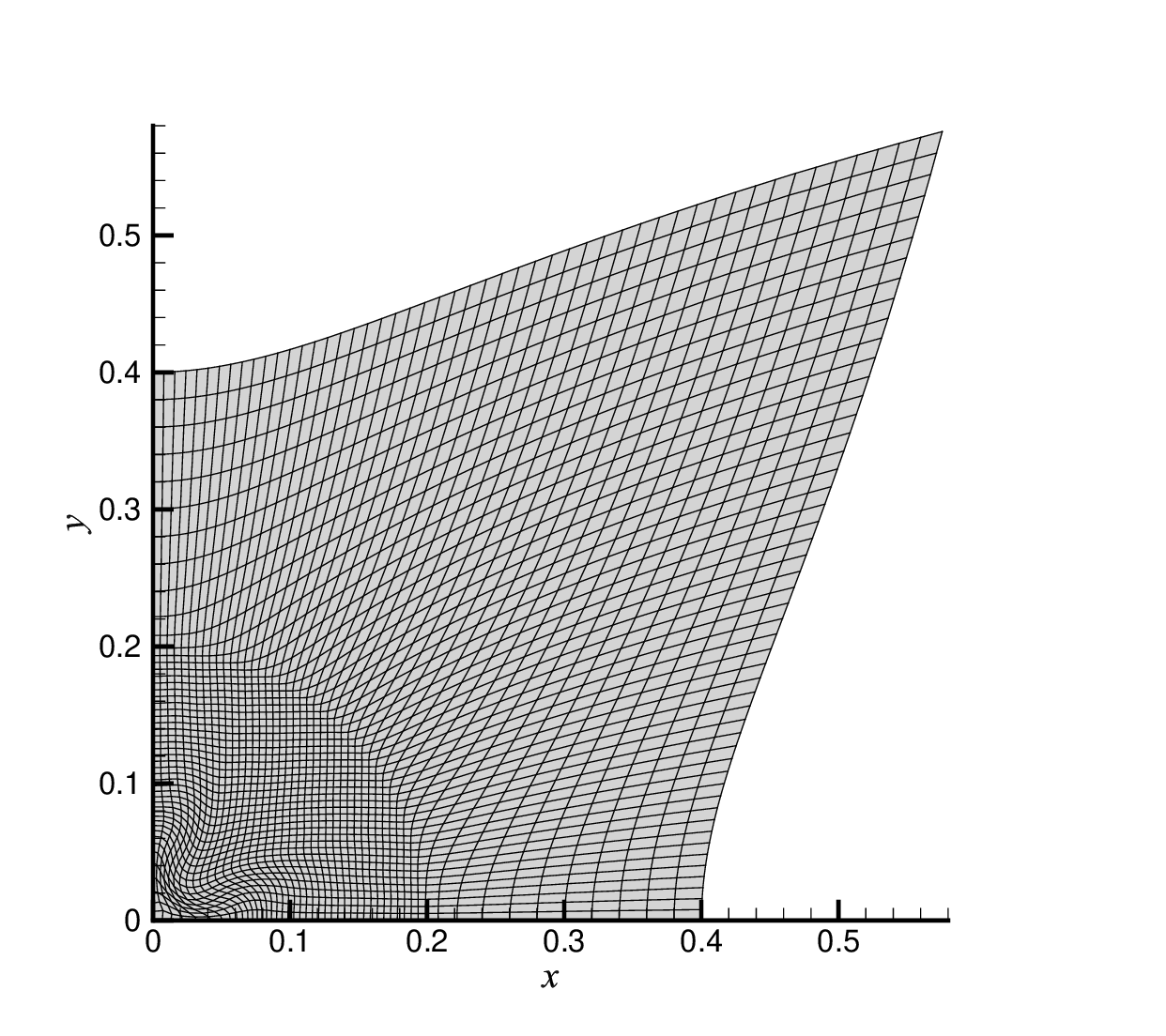}
		\caption{The final mesh for LMCV}
	\end{subfigure}
	\begin{subfigure}{0.48\linewidth}
		\centering
		\includegraphics[width=\textwidth]{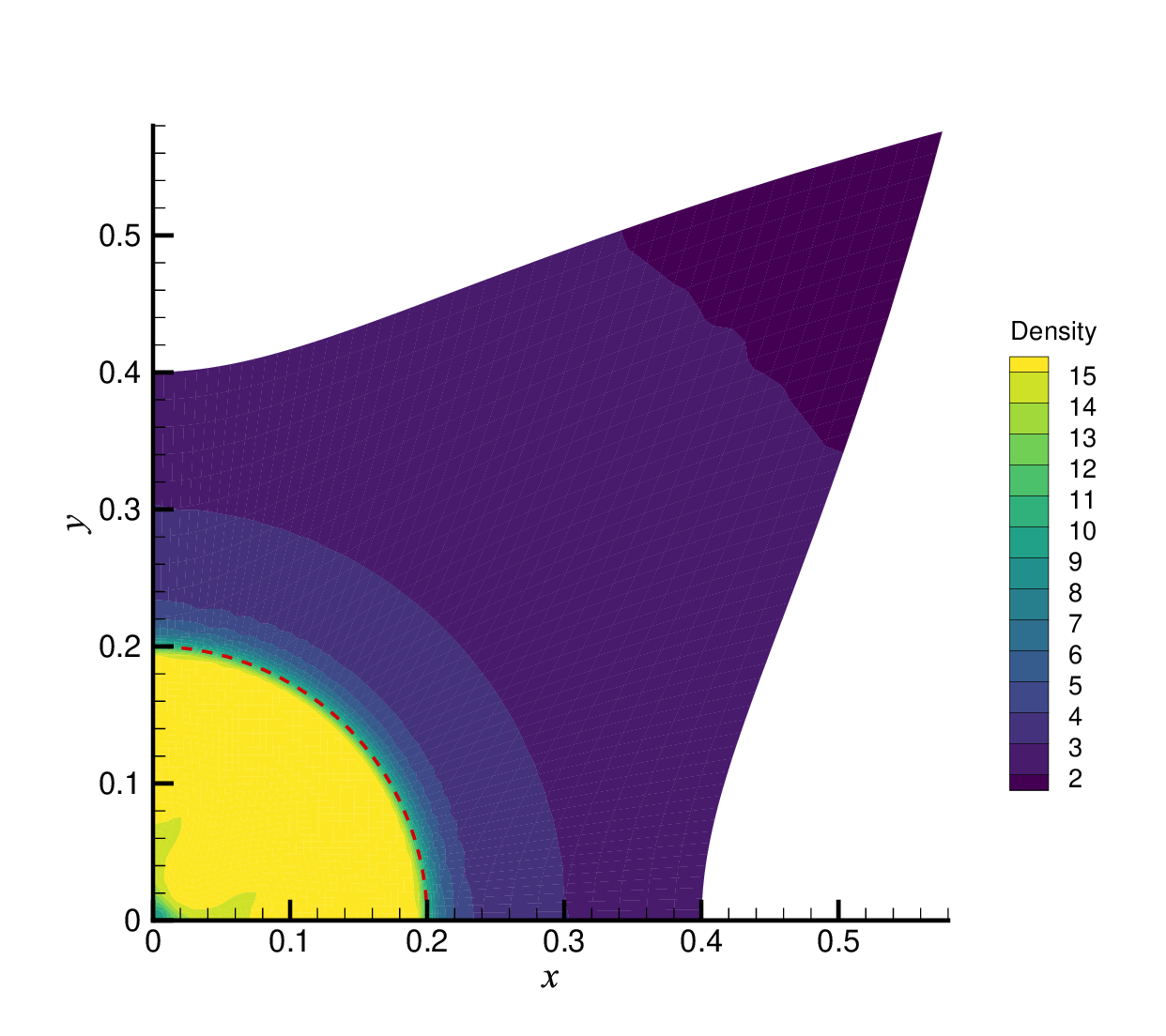}
		\caption{The final density contours for LMCV}
	\end{subfigure}
	 \\
	\caption{The meshes and density contours for Noh problem at time $t=0.6$ with initially uniform mesh of $50 \times 50$ size. The red dashed line is the theoretical shock front location. }
	\label{Fig.Noh-mesh}
\end{figure}

\begin{figure}[!htbp]
	\centering  
	\begin{subfigure}{0.48\linewidth}
		\includegraphics[width=\textwidth]{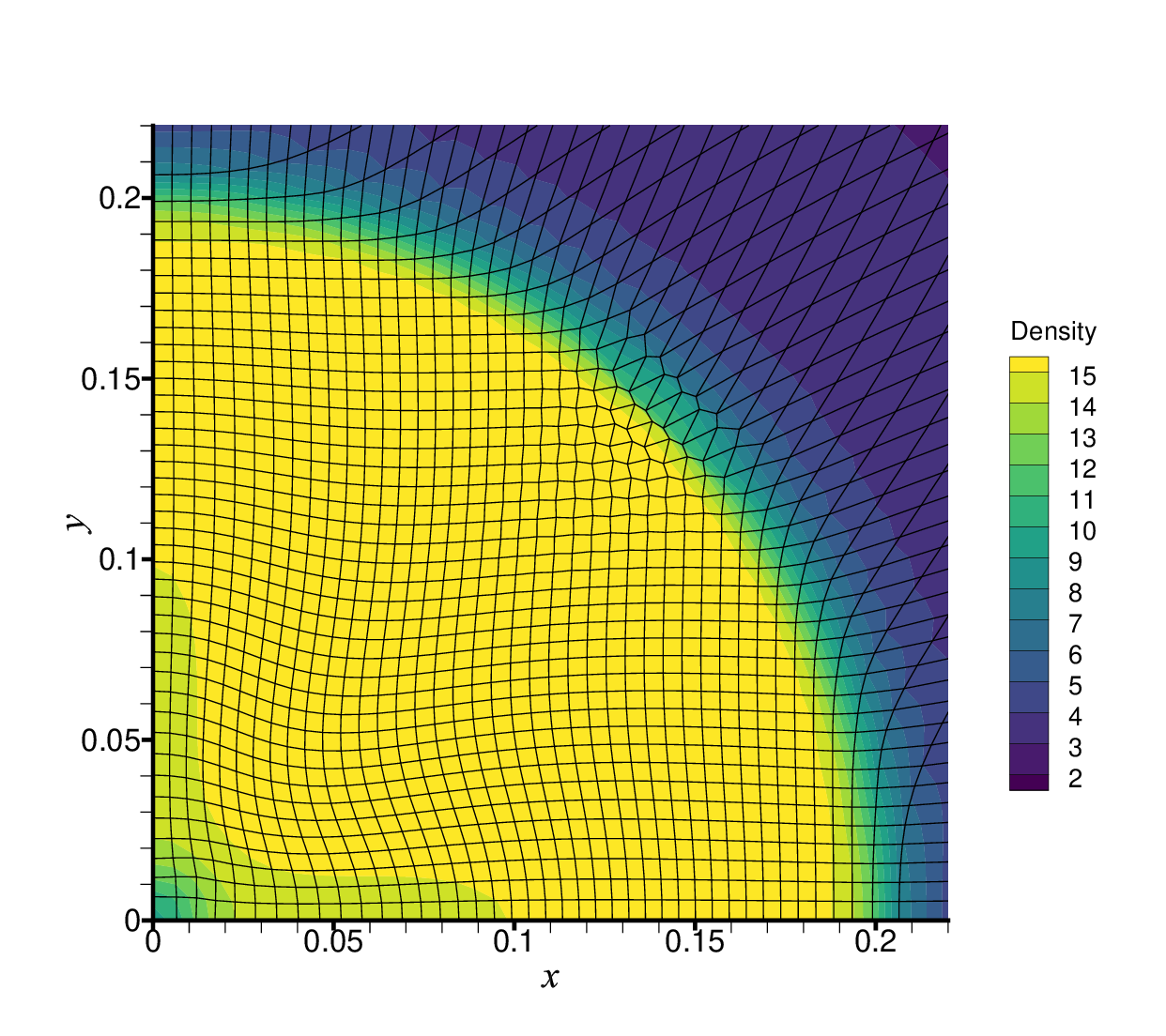}
		\caption{The close-up of final mesh for FV}
		\centering
	\end{subfigure}
	\begin{subfigure}{0.48\linewidth}
		\centering
		\includegraphics[width=\textwidth]{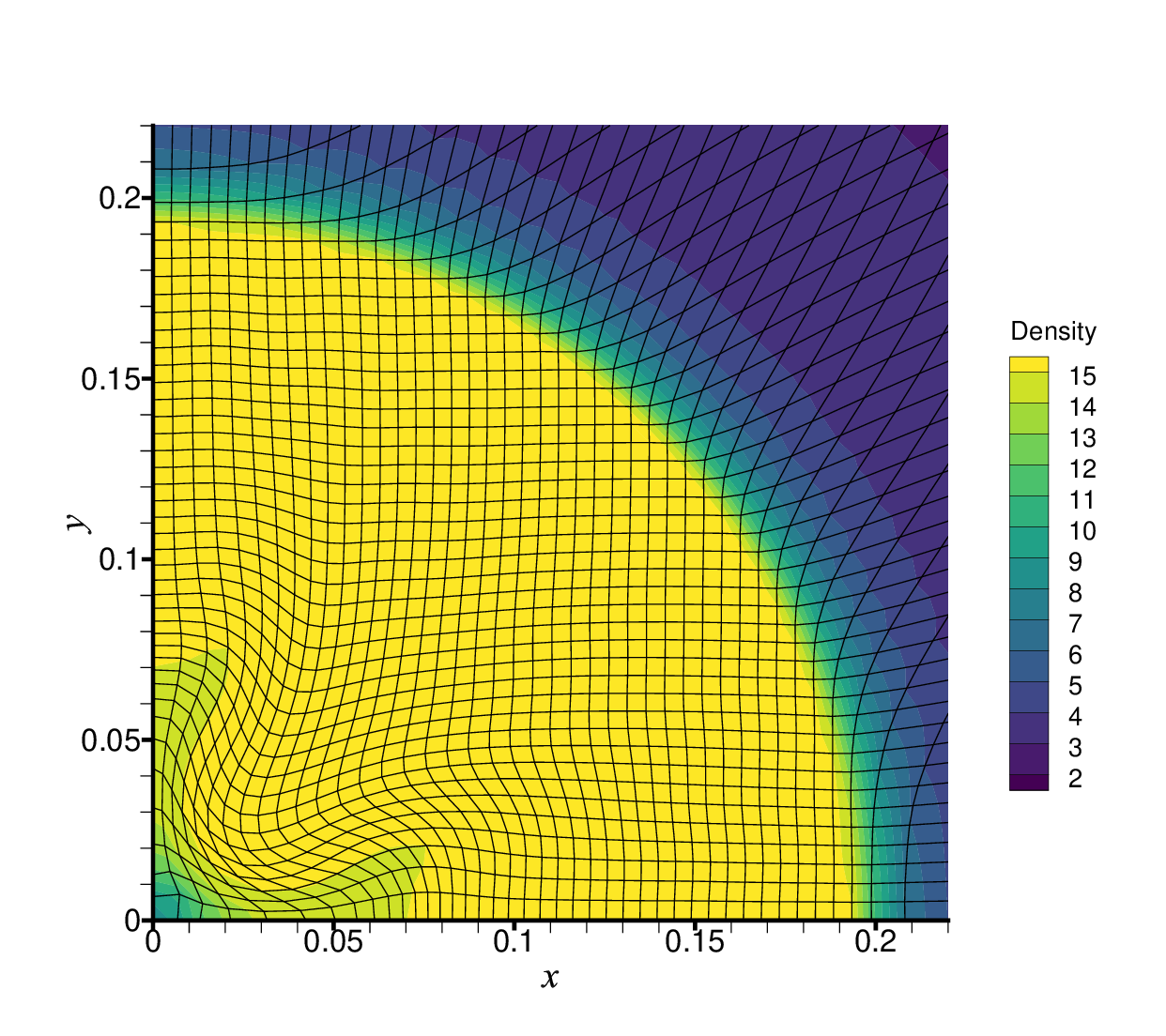}
		\caption{The close-up of final mesh for LMCV}
	\end{subfigure}
	\caption{The close-up of the meshes behind the shock at $t = 0.6$ for the Noh problem. }
	\label{Fig.Noh-local}
\end{figure}

\begin{figure}[!htbp]
	\centering  
	\begin{subfigure}{0.48\linewidth}
		\centering
		\includegraphics[width=\textwidth]{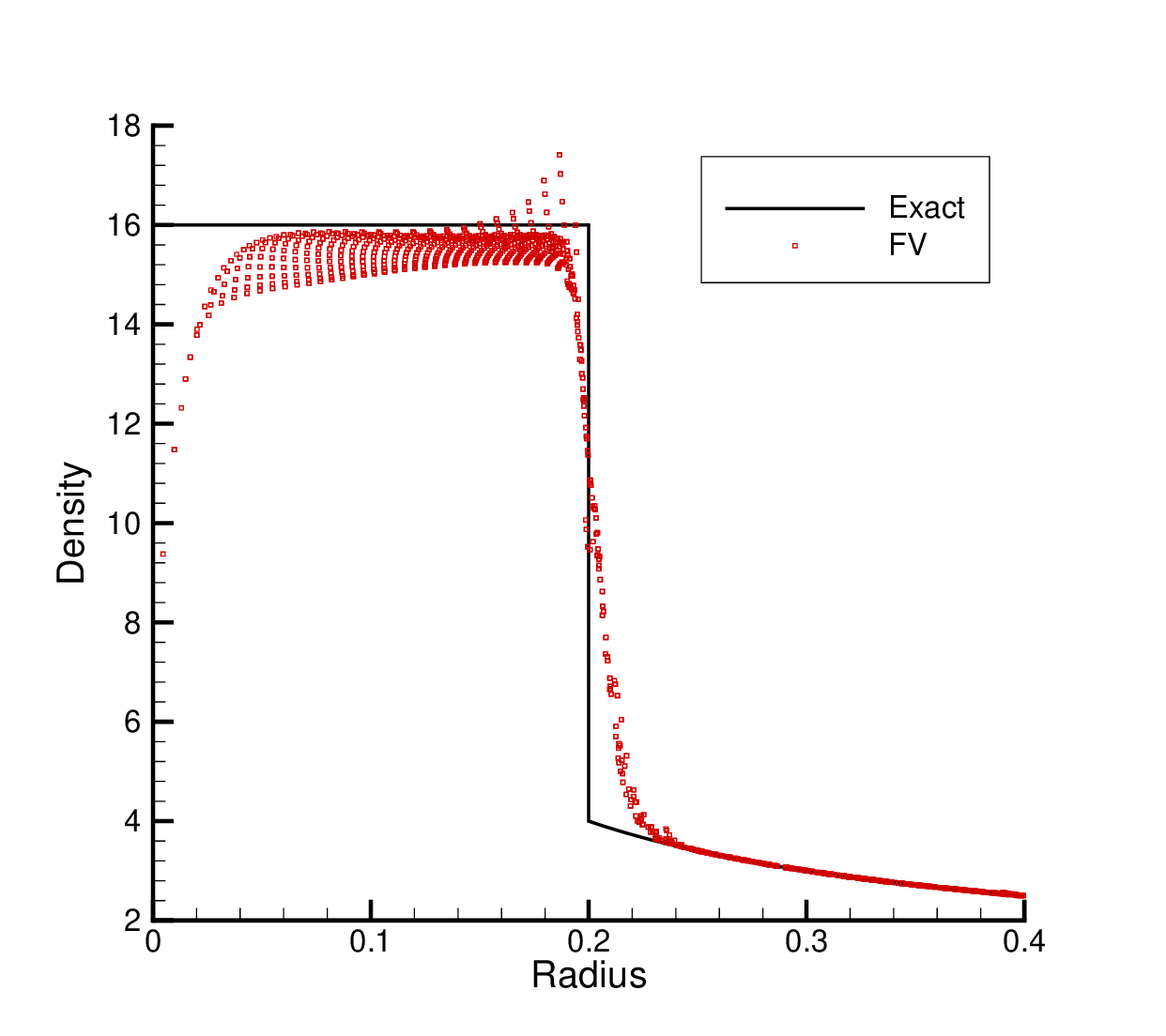}
	\end{subfigure}
	\begin{subfigure}{0.48\linewidth}
		\centering
		\includegraphics[width=\textwidth]{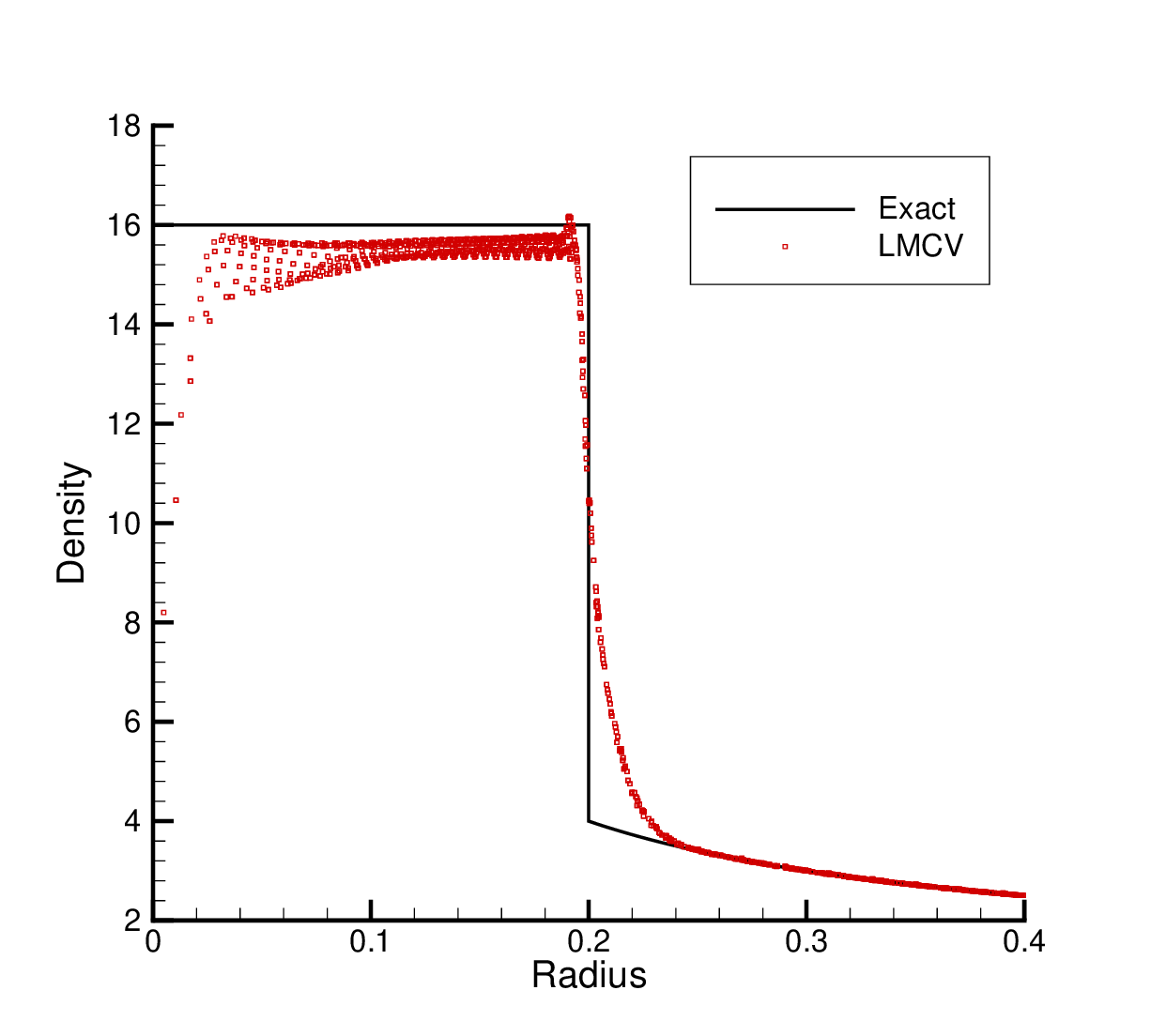}
	\end{subfigure}
	\caption{The scatter plots of density at $t=0.6$ for Noh problem  with EUCCLHYD (left) and LMCV (right). }
	\label{Fig.Noh-rad}
\end{figure}

\subsection{Saltzman problem}
Saltzman problem \cite{dukowicz1992vorticity,margolin1988centered} 
is a challenging problem which describes a piston-driven planar shock passing through a misaligned mesh displayed in Fig. \ref{Fig.Saltzman-m0}, which is given by transforming the uniform $100\times 10$ mesh of rectangle $[0,1]\times[0,0.1]$ with the following mapping
\begin{eqnarray*}
\boldsymbol{\phi}(x,y) = (x+(0.1-y)\sin(\pi x), y).
\end{eqnarray*}
The physical domain is initially filled with ideal gas with $\gamma = 5/3$ and state
\begin{eqnarray}
\rho_0 =1, \quad \bV_0 = \boldsymbol{0},\quad P_0 = 10^{-6}.
\end{eqnarray} Boundaries are fixed except the left boundary on which a unity inward normal velocity is prescribed.
Since the compression raises a shock wave, the limiting procedure with $S_c = 5 \times 10^5$ is adopted in this test. From start,
the shock forms at the left boundary and propagates to the right at a velocity of 4/3. The shock compresses the gas behind to a desity equal to 4, and first hits the right boundary at $t = 0.75$. After reflecting from the right boundary, the shock travels to the left and leave a steady field with density $\rho =10$ behind until $t=0.9$, when the shock hits the left boundary. The meshes and density contours at $t = 0.6$, $t = 0.75$ and $t = 0.9$ are shown in Fig. \ref{Fig.Saltzman-mesh}. Comparing with results given by EUCCLHYD, meshes from LMCV move in a more stable and smooth manner with less distortion near the left boundary. Moreover ,the density scatter plot of LMCV agrees well with the analytical solution in Fig. \ref{Fig.Saltzman-x}.

\begin{figure}[!htbp]
	\centering  
	\begin{subfigure}{\linewidth}
		\centering
		\includegraphics[width=\textwidth]{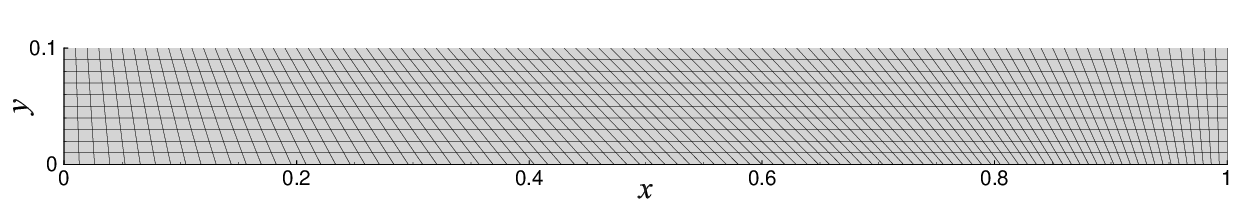}
	\end{subfigure}
	\caption{The initial mesh for Saltzman problem.}
	\label{Fig.Saltzman-m0}
\end{figure}

\begin{figure}[!htbp]
	\centering  
	\begin{subfigure}{0.6\linewidth}
		\centering
		\includegraphics[width=\textwidth]{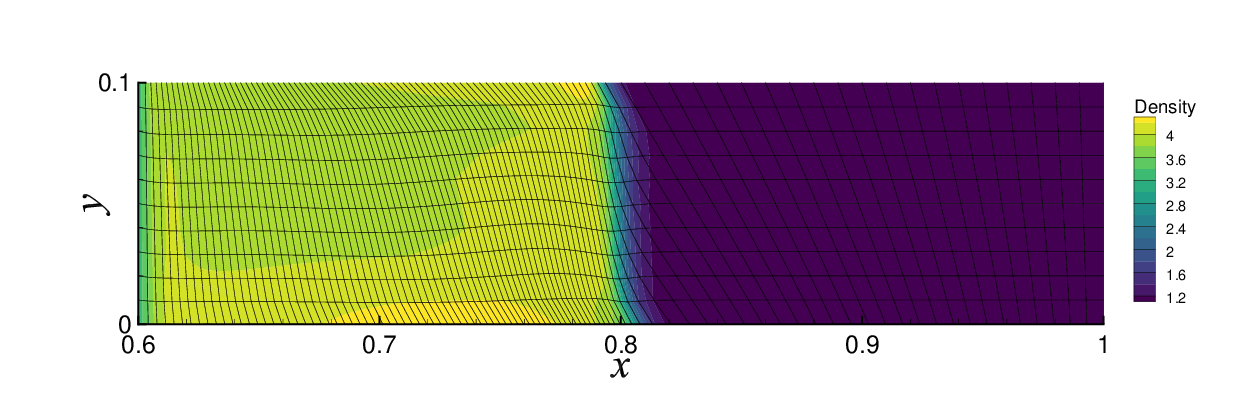}
		\caption{FV at $t=0.6$}
	\end{subfigure}
	\begin{subfigure}{0.39\linewidth}
		\centering
		\includegraphics[width=\textwidth]{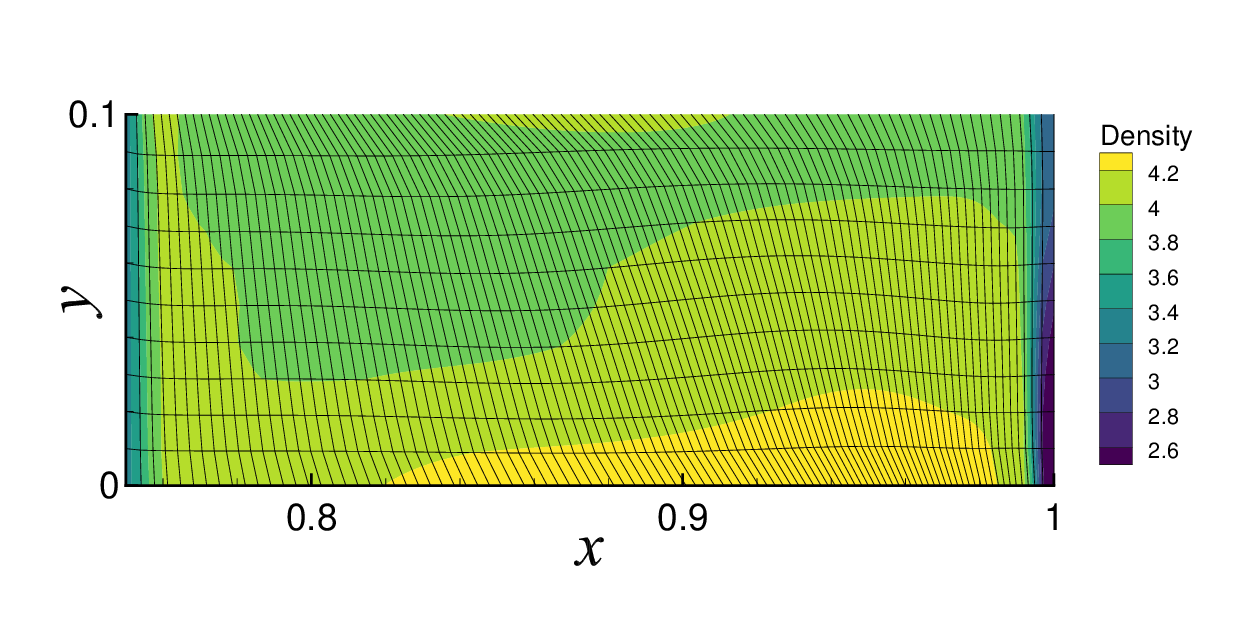}
		\caption{FV at $t=0.75$}
	\end{subfigure}
	\begin{subfigure}{0.6\linewidth}
		\centering
		\includegraphics[width=\textwidth]{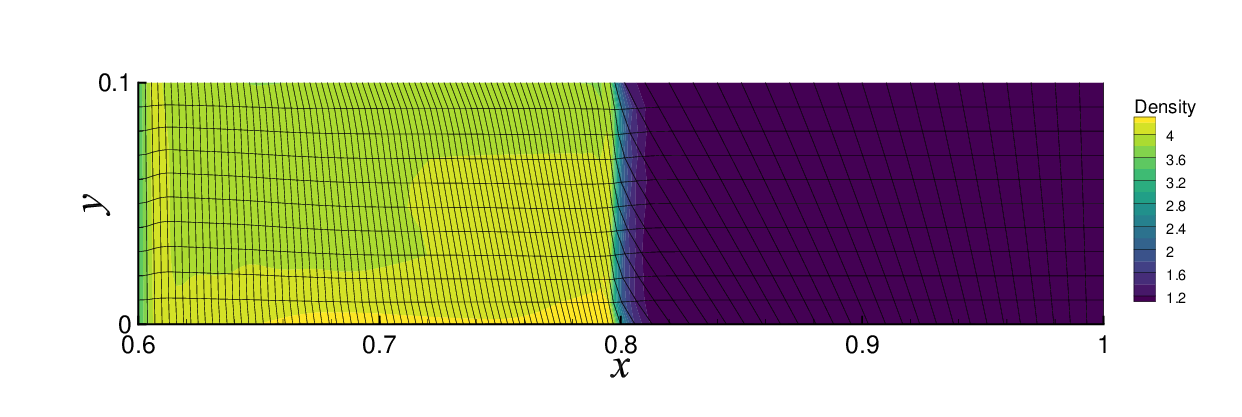}
		\caption{LMCV at $t=0.6$}
	\end{subfigure}
	\begin{subfigure}{0.39\linewidth}
		\centering
		\includegraphics[width=\textwidth]{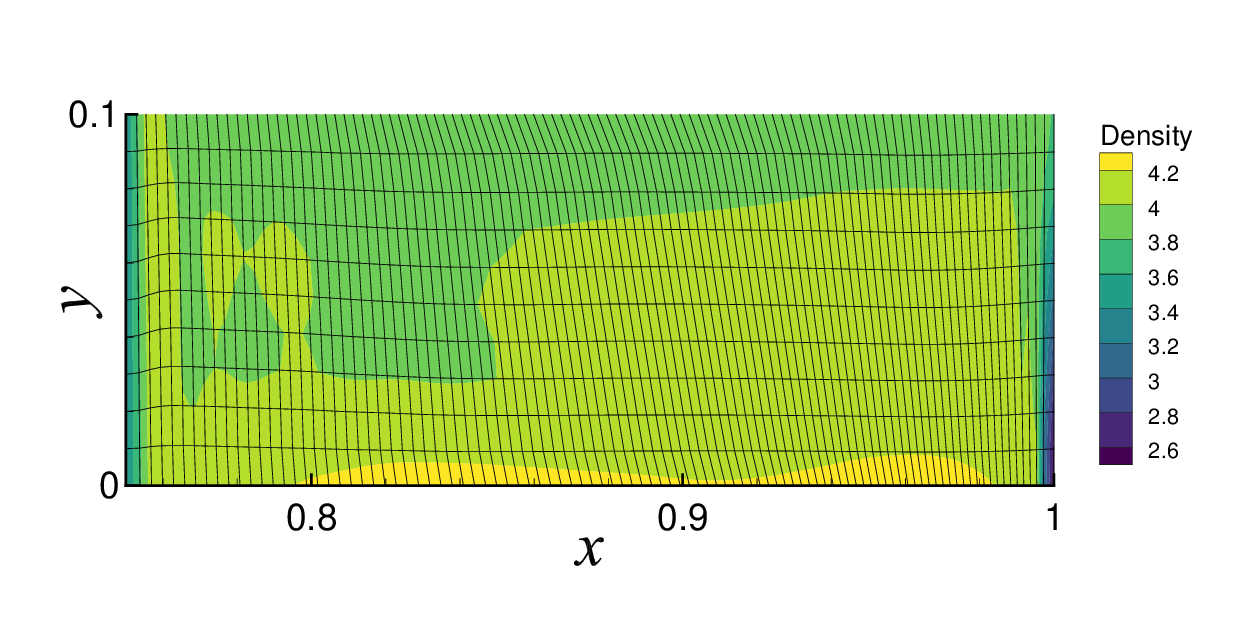}
		\caption{LMCV at $t=0.75$}
	\end{subfigure}
	\begin{subfigure}{0.45\linewidth}
		\centering
		\includegraphics[width=\textwidth]{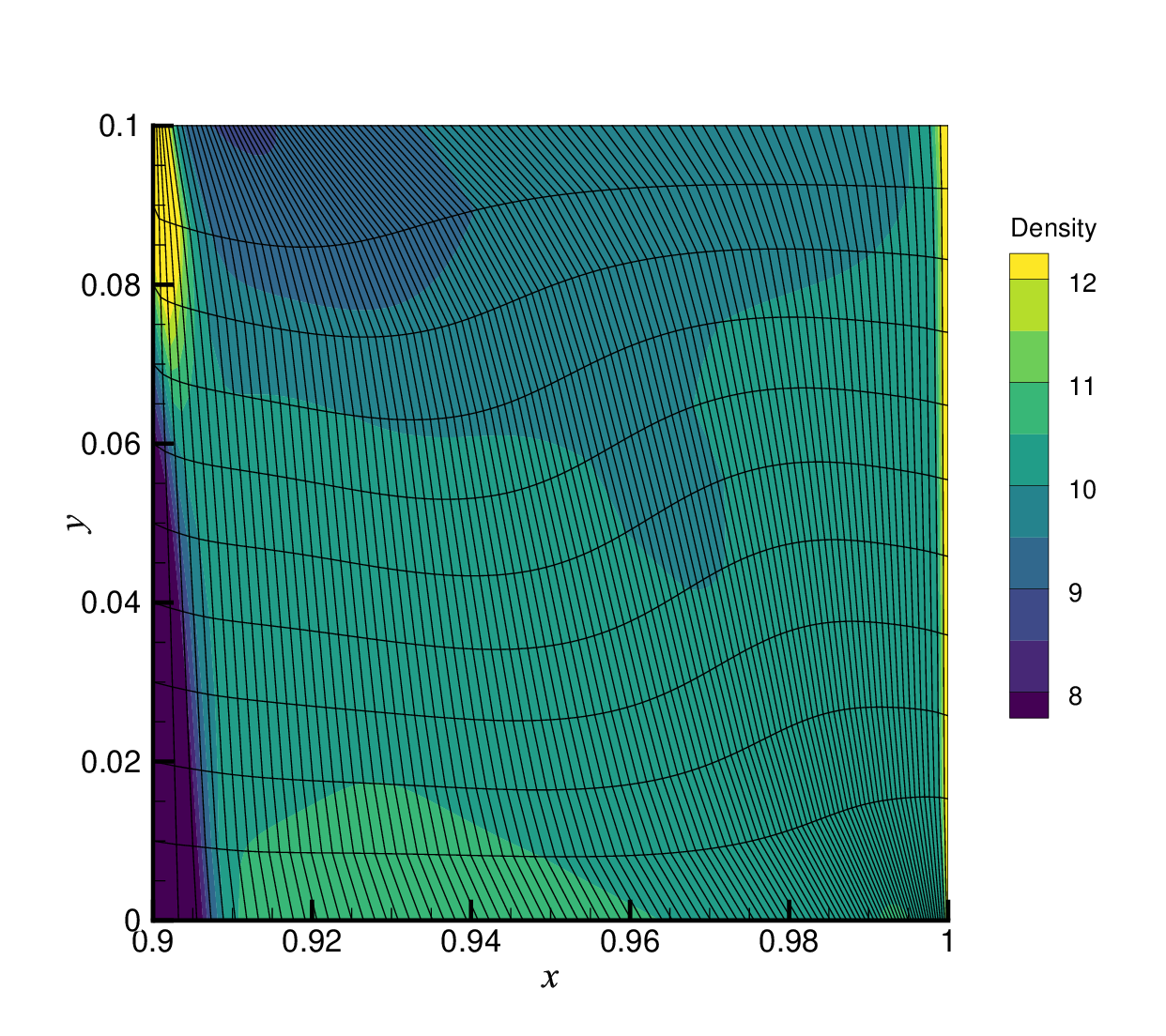}
		\caption{FV at $t=0.9$}
	\end{subfigure}
	\begin{subfigure}{0.45\linewidth}
		\centering
		\includegraphics[width=\textwidth]{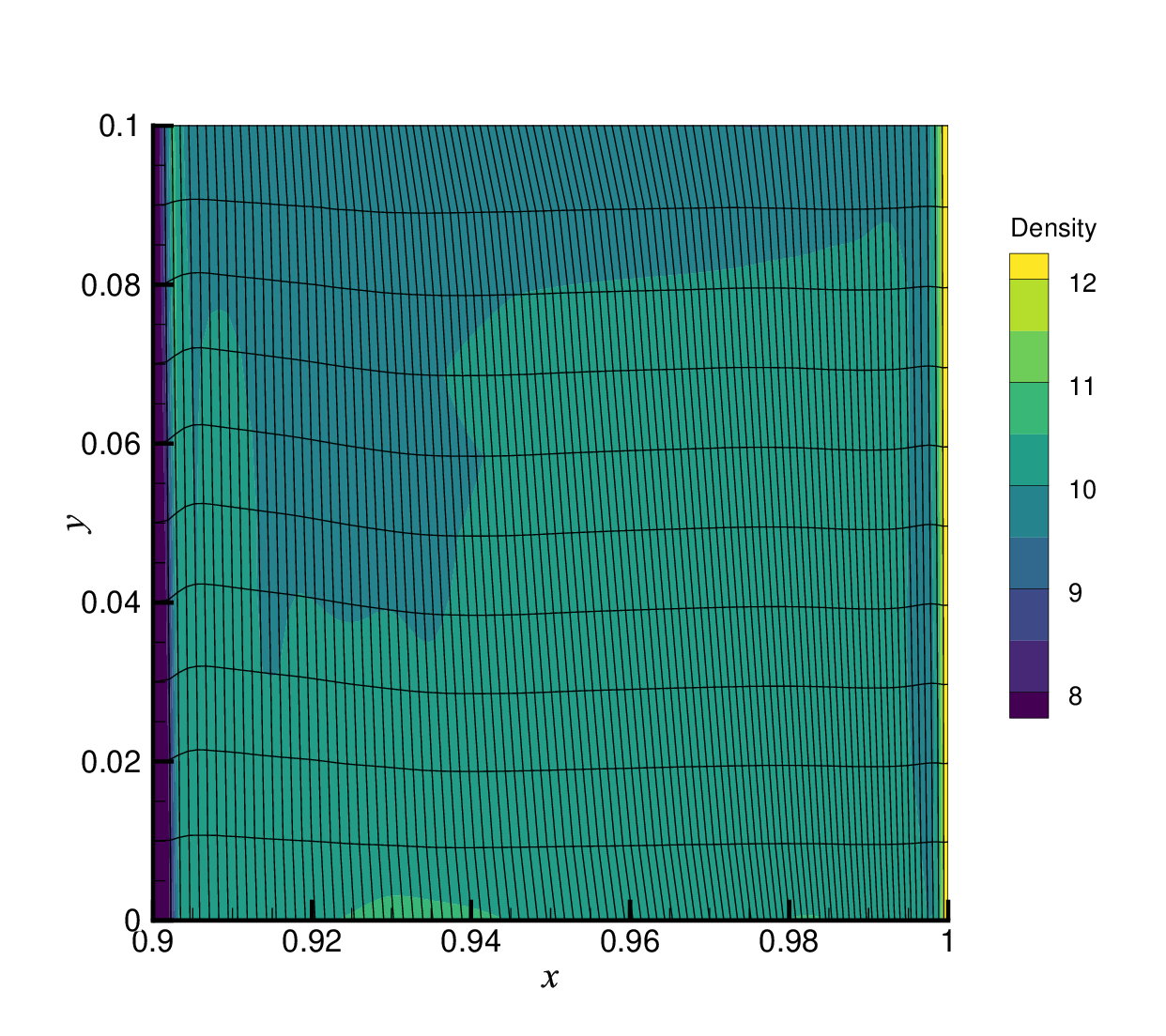}
		\caption{LMCV at $t=0.9$}
	\end{subfigure}
	\caption{The density contours with mesh for Saltzman problem at different time. }
	\label{Fig.Saltzman-mesh}
\end{figure}

\begin{figure}[!htbp]
	\centering  
	\begin{subfigure}{0.43\linewidth}
		\centering
		\includegraphics[width=\textwidth]{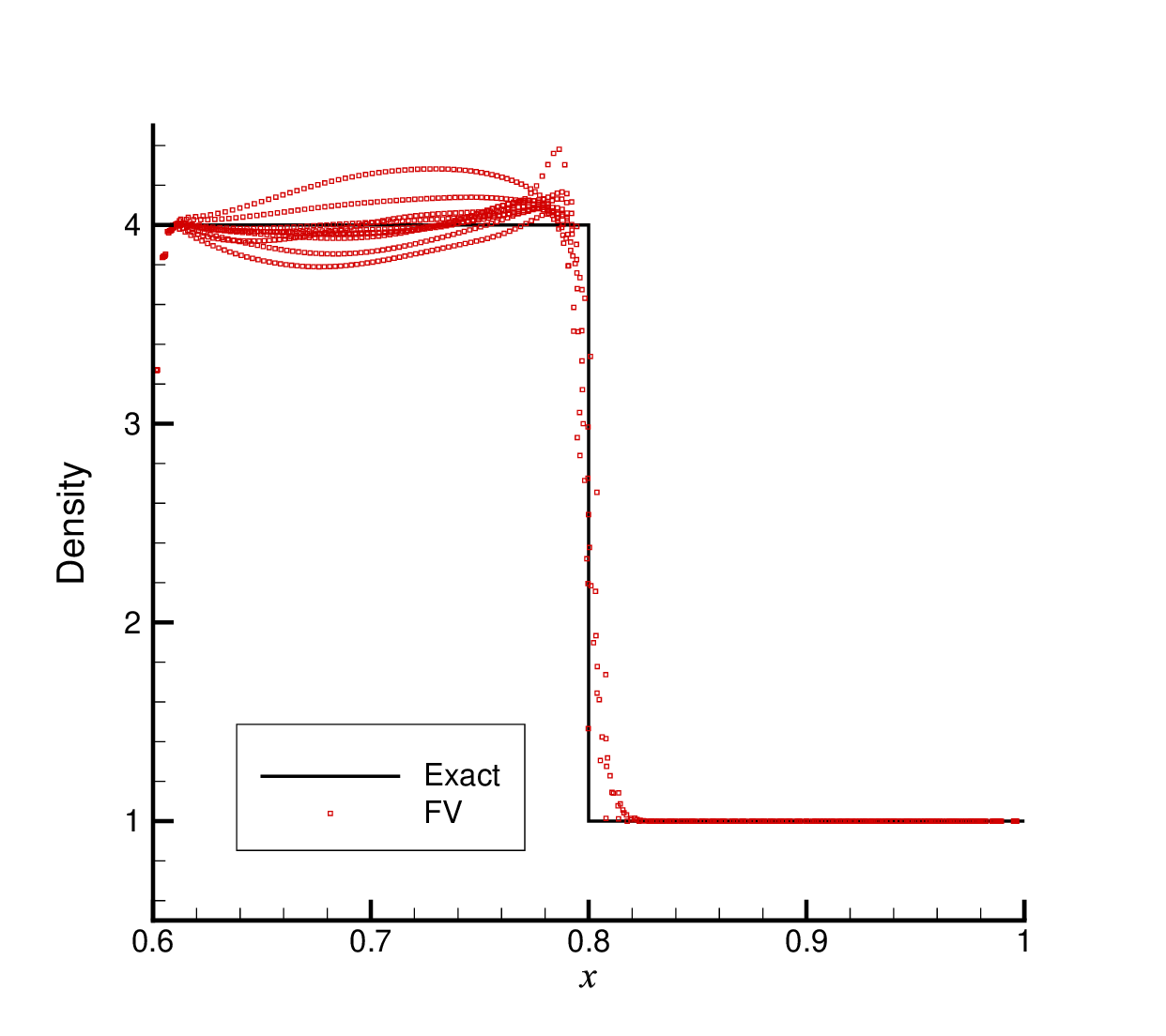}
		\caption{FV at $t=0.6$}
	\end{subfigure}
	\begin{subfigure}{0.43\linewidth}
		\centering
		\includegraphics[width=\textwidth]{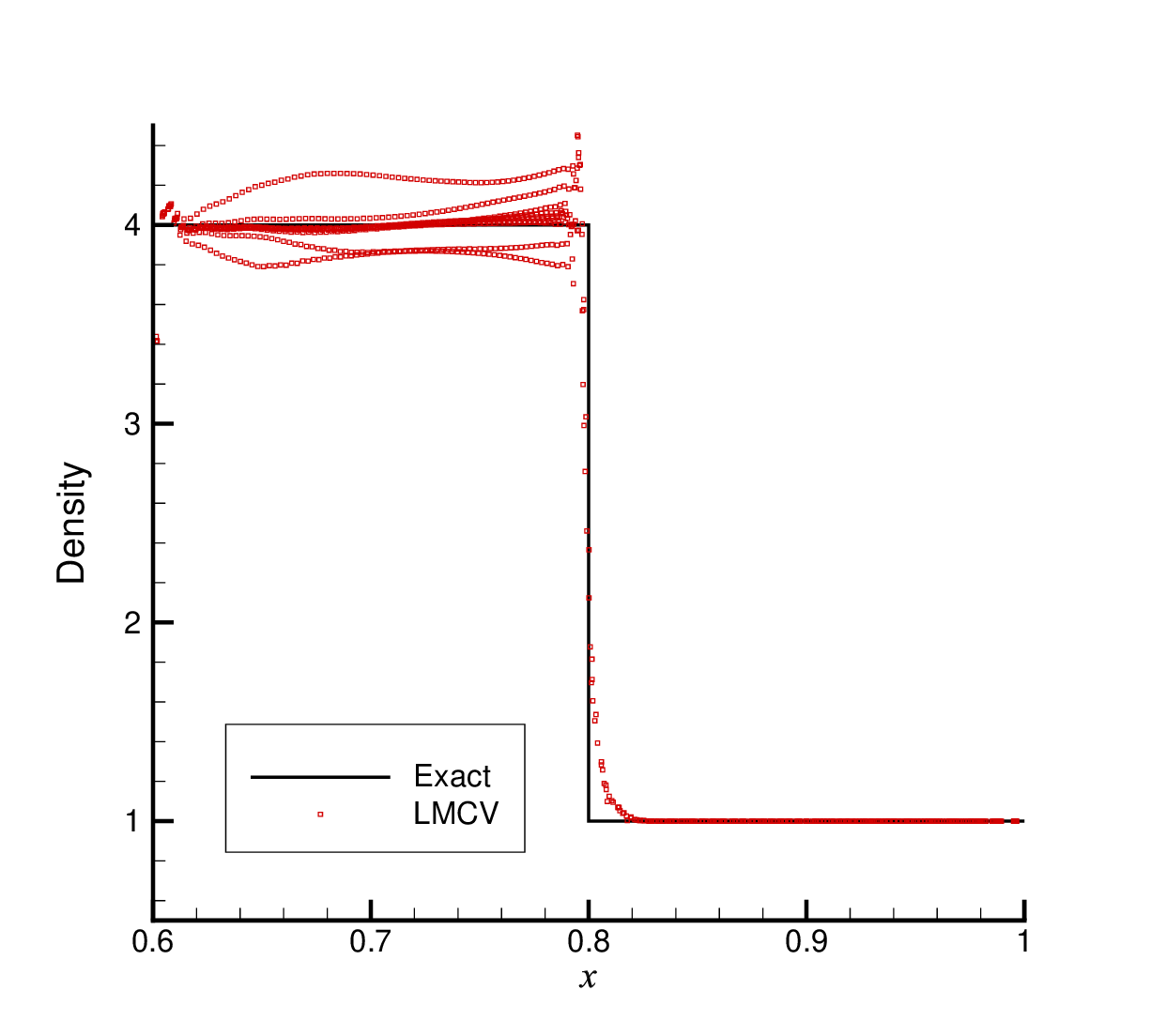}
		\caption{LMCV at $t=0.6$}
	\end{subfigure}
	\begin{subfigure}{0.43\linewidth}
		\centering
		\includegraphics[width=\textwidth]{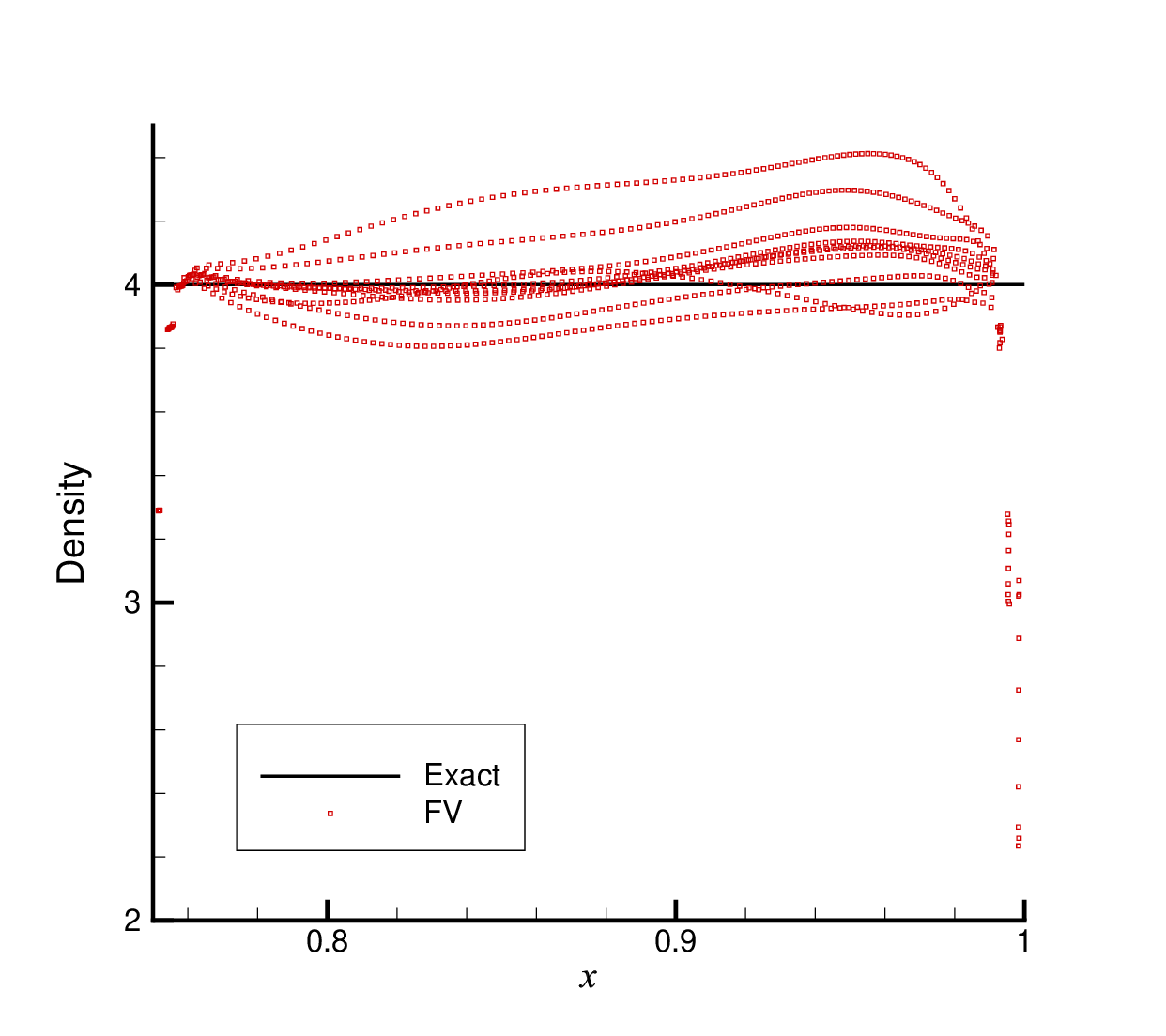}
		\caption{FV at $t=0.75$}
	\end{subfigure}
	\begin{subfigure}{0.43\linewidth}
		\centering
		\includegraphics[width=\textwidth]{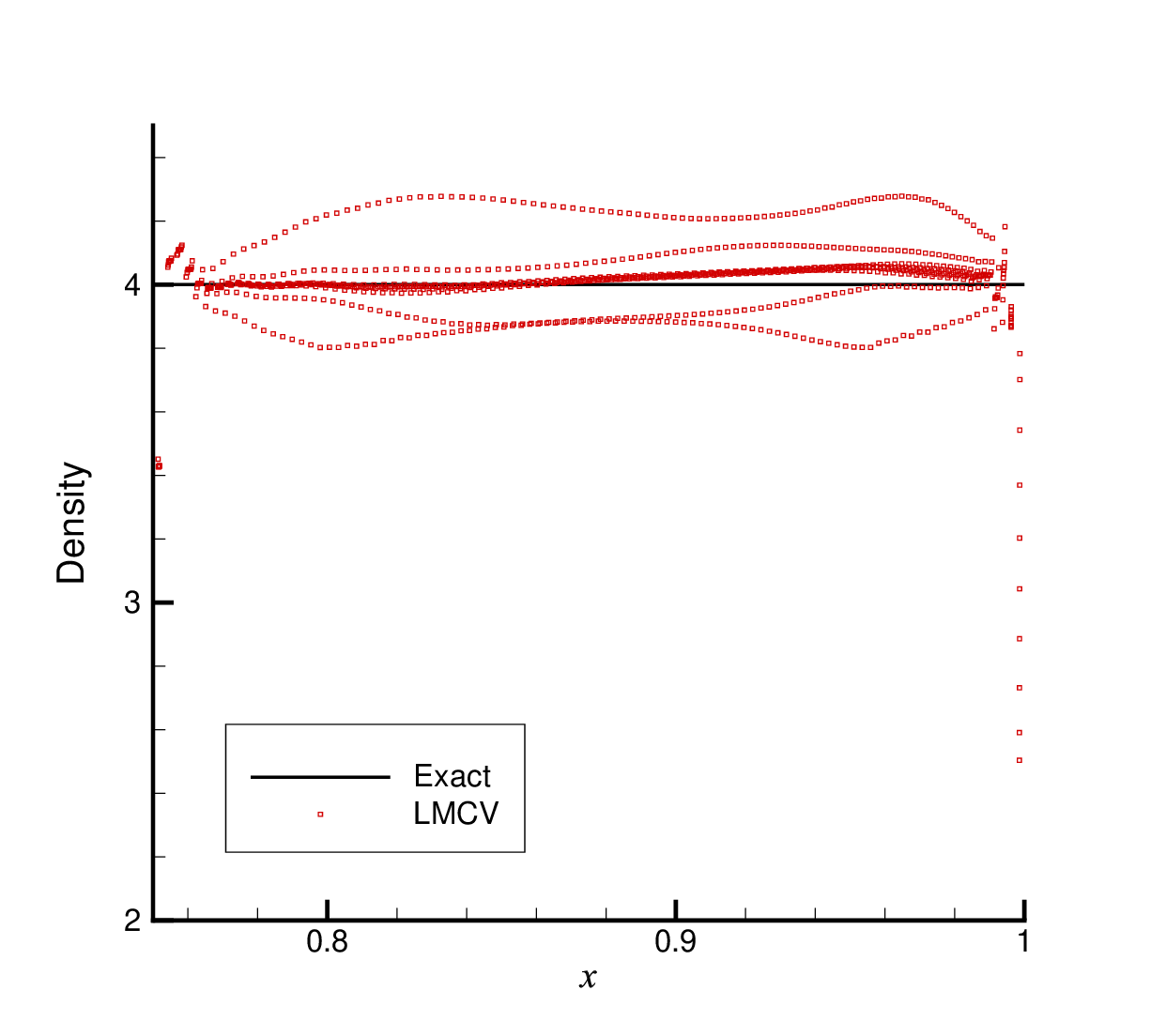}
		\caption{LMCV at $t=0.75$}
	\end{subfigure}
	\begin{subfigure}{0.43\linewidth}
		\centering
		\includegraphics[width=\textwidth]{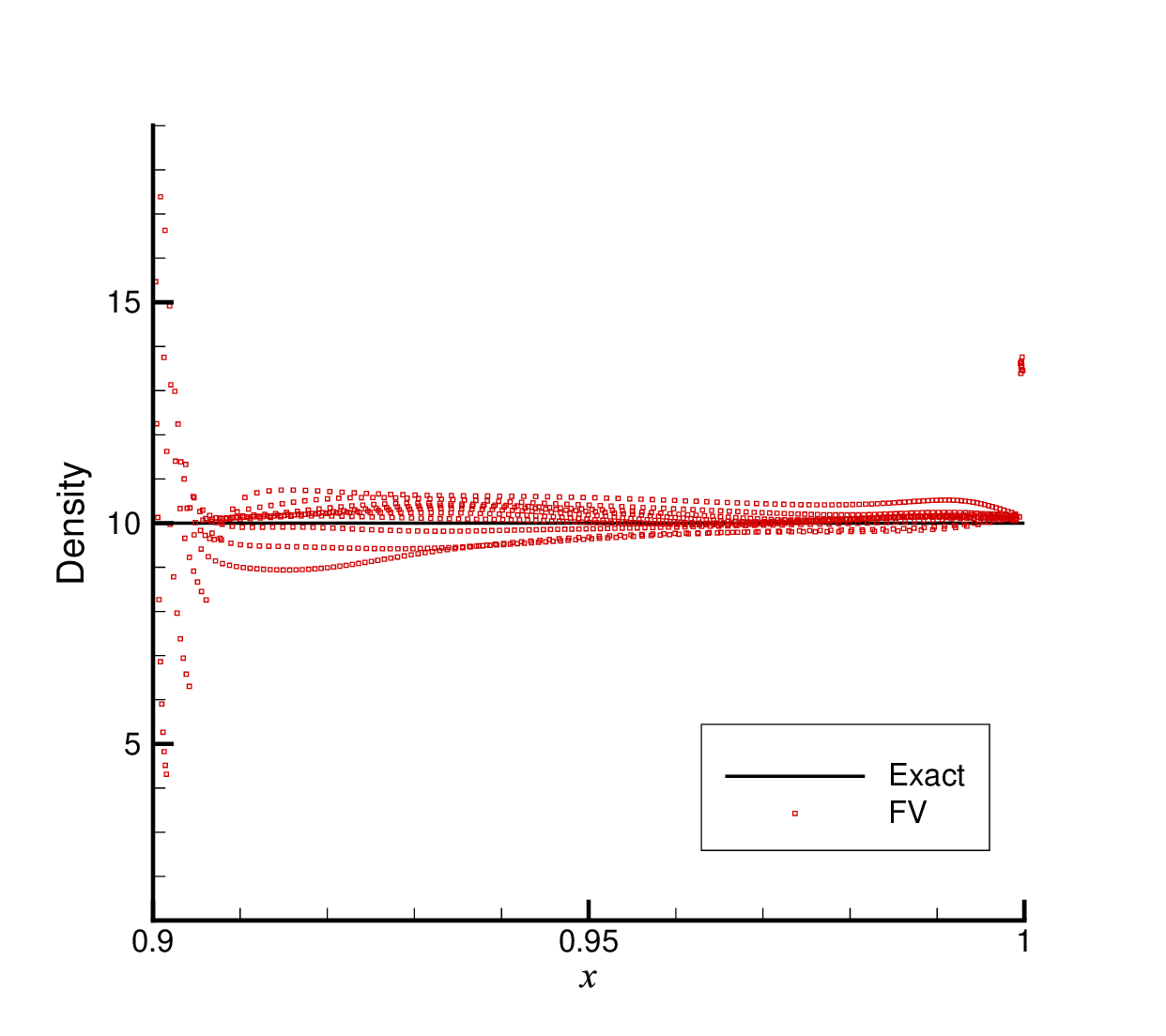}
		\caption{FV at $t=0.9$}
	\end{subfigure}
	\begin{subfigure}{0.43\linewidth}
		\centering
		\includegraphics[width=\textwidth]{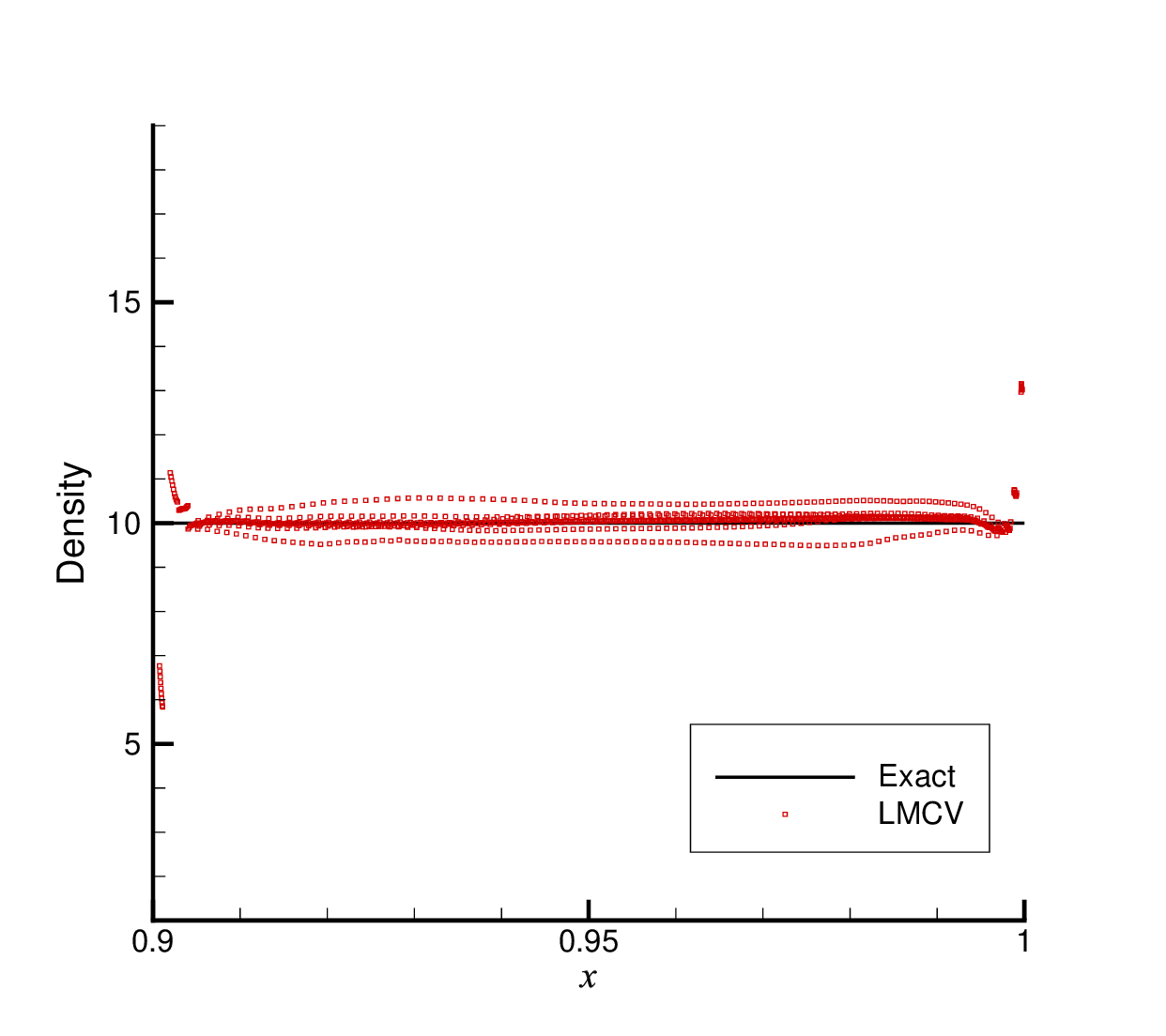}
		\caption{LMCV at $t=0.9$}
	\end{subfigure}
	\caption{The scatter plots of density for Saltzman problem at different time.}
	\label{Fig.Saltzman-x}
\end{figure}

\subsection{Triple point problem}
Finally we consider a three-state 2D Riemann problem, i.e. triple point problem \cite{Loubere2005TestCase}. The computational domain $\Omega = [0,7]\times[0,3]$ is surrounded by rigid walls and split into three regions, and the regions and states is initialized as
\begin{eqnarray*}
\Omega^1 = [0,1]\times[0,3]: & \quad \rho=1, \bV = \boldsymbol{0}, P =1,\\
\Omega^2 = [1,7]\times[0,1.5]: & \quad \rho=1, \bV = \boldsymbol{0}, P =0.1,\\
\Omega^3 = [1,7]\times[1.5,3]: & \quad \rho=0.1, \bV = \boldsymbol{0}, P =0.1,
\end{eqnarray*}
with $\gamma=1.4$.
This problem is designed to access the robustness of a Lagrangian scheme facing significant vorticity. The high pressure from left region $\Omega^1$ drives a shock through the right $\Omega^2 \cup \Omega^3$ and generates a vortex at the triple point where three regions connect. $\Omega$ is discretized by a uniform mesh of $70 \times 30$ size. To deal with the interaction between shocks and the high distortion of cells, not only is the limiting procedure used in this test case with $S_c = 5\times 10^5$ , but also the acoustic impedance is calculated in a more robust manner by Dukowicz solver \cite{Maire-Staggered} as
\begin{eqnarray*}
\alpha_{pt} = \rho_{pt}\brb{c_{pt} + \beta \frac{\gamma+1}{2} (\tbV_e - \bV_{pt}) \cdot \bN_e},
\end{eqnarray*}
where $pt$ denotes a point on any half surface $e$.
User-defined coefficient $\beta$ is set equal to $\beta = 6$ in order to acquire smooth mesh movement and reduce mesh self-intersection.
Solutions at time $t=3$ given by EUCCLHYD and LMCV are shown in Fig. \ref{Fig.triple-rho}. It is evident that greater amount of vorticity is generated for LMCV in the absence of significant overlapping cells, thereby demonstrating the accuracy and robustness of our high order scheme.




\begin{figure}[!htbp]
	\centering  
	\begin{subfigure}{\linewidth}
		\centering
		\includegraphics[width=\textwidth]{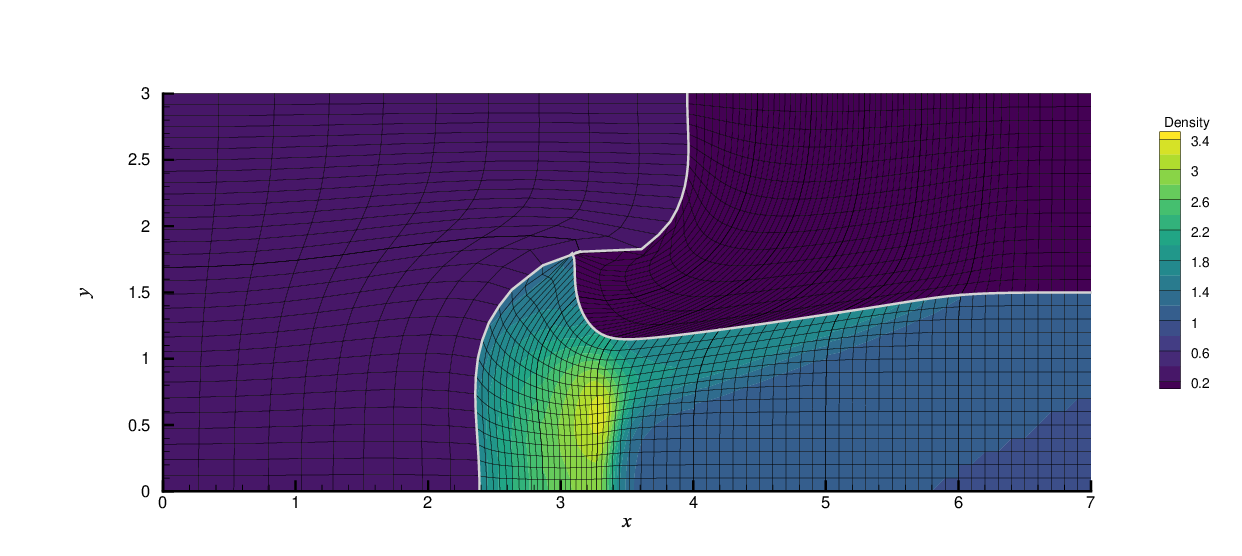}
	\end{subfigure}
	\begin{subfigure}{\linewidth}
		\centering
		\includegraphics[width=\textwidth]{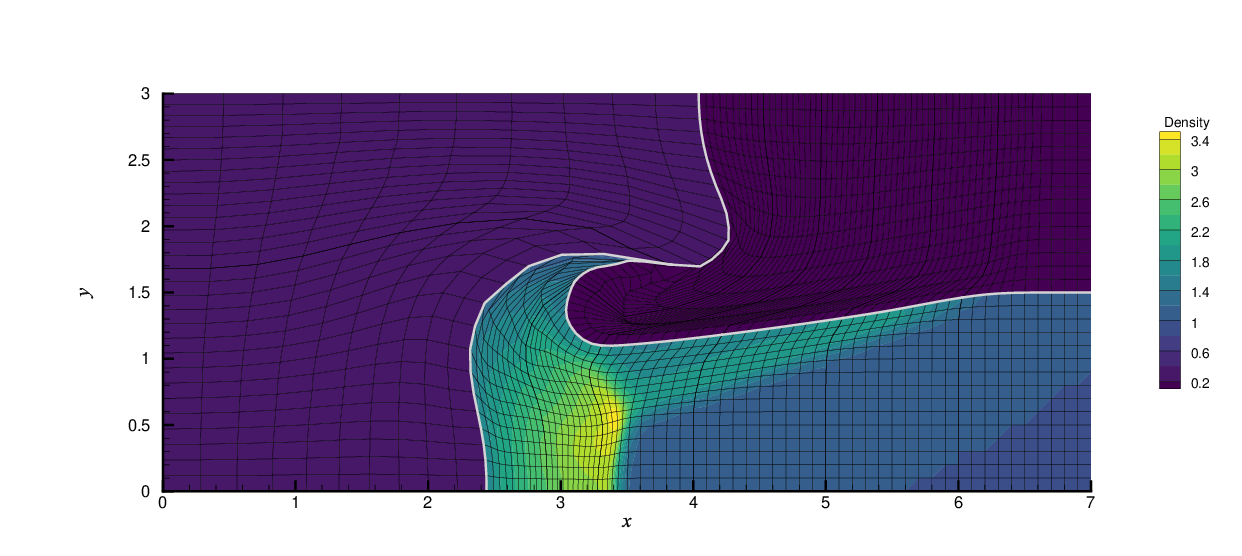}
	\end{subfigure}
	\caption{The density contours with mesh given by EUCCLHYD (left) and LMCV (right) for the triple point problem at time $t=3$, inner boundaries between three regions $\Omega$ are highlighted. }
	\label{Fig.triple-rho}
\end{figure}

\section{Conclusion} \label{chap:Lag:Conclusion}
We presented a novel third-order cell-centered Lagrangian scheme for hydrodynamics called LMCV on quadrilateral meshes using an augmented fourth-order 2D nodal solver. In our method,
the physical field is discretized into moments of two kind (i.e. VIA and PV) and approximated by polynomial reconstruction on reference cells. Based on augmented jump condition and balance condition which take the physical field on cell boundaries into full consideration, our new 2D nodal solver is capable to deliver conservative and accurate fluxes and consistent nodal velocities, lays a solid foundation for the robustness and high accuracy of our scheme.

Further more, a set of challenging test cases are calculated to verify the capacity of LMCV. Firstly, three various smooth tests are performed, the results confirm the three-order accuracy of LMCV across various flow fields and initial mesh configurations. Secondly, four flows with strong shocks are used to demonstrate the robustness of our scheme without losing accuracy facing discontinuities. To the best of our knowledge, LMCV is the first cell-centered Lagrangian scheme which achieves third-order accuracy on quadrilateral meshes without introducing curved edges.

\section*{Acknowledgements}

\section*{Appendix \hypertarget{Apd.2rd}{A}. Proof of specific volume accuracy}
This appendix primarily illustrates the specific volume error resulting from straight edges on quadrilateral mesh. 

For a set of uniformly refined quadrilateral mesh $\brc{M_h}_{h\in\mathbb{R}}$ and a smooth field $\U$ defined on physical domain, we assume that every vertices in $M_h$ flow EXACTLY according the velocity field $\bV$ given by $\U$, that is $\tbV^i_r = \bV(\bx^i_r) = \bV^i_r$. Consider a quadrilateral cell $\omega^i$, let 
$$(\tau^i_{\mathrm{ref}})_t = \frac{1}{m_i}\int_{\partial\omega^i}\bV\cdot\bN\dd l,$$
be the exact change rate of specific volume according to Eqn. \eqref{lag^int}, while the actual change rate of specific volume is 
$$
(\tau^i)_t = \frac{1}{m_i}
\sum_{r=1}^{4} L^i_{r,r+1}\bra{\hf\bV^i_r+\hf\bV^i_{r+1}} \cdot \bN^i_{r,r+1},
$$
since the edge stay straight. Our goal is to prove $$(\tau^i_{\mathrm{ref}})_t - (\tau^i)_t = \cO(h^2).$$ 
To be more specific, mesh $M_h$ in $\brc{M_h}_{h\in\mathbb{R}}$ is defined by its nodes $\bx_{k,l}$ with $k,l\in \mathbb{Z}$, and we assumed that the refinement follows $\bx_{k,l} = \bphi(kh,lh)$ with  $\bphi:\mathbb{R}^2\to\mathbb{R}^2$ is a smooth homeomorphism. 

Firstly, the error at each edge $e=\overline{\bx^i_r\bx^i_{r+1}}$ can be quantified as follows. Let $u(\bx) = \bV(\bx)\cdot \bN^i_{r,r+1}$, $u^i_r=u(\bx^i_r)$ and $u^i_{r+1} = u(\bx^i_{r+1})$, then the error is computed based on trapezoidal rule as  
\begin{eqnarray*}
    & &\int_{e} \bV \cdot \bN\dd l - L^i_{r,r+1}\bra{\hf\bV^i_r+\hf\bV^i_{r+1}} \cdot \bN^i_{r,r+1} \\
    &=& \int_{e} u \dd l - L^i_{r,r+1}\bra{\hf u^i_r +\hf u^i_{r+1}} \\
    &=& -\frac{1}{12}(L^i_{r,r+1})^3\brb{\bT^\top(\nabla^2 u)(\bx^i_{r+\hf})\bT} + \cO\bra{(L^i_{r,r+1})^5}, 
\end{eqnarray*}
where $\bT = \mathbb{T} \bN^i_{r,r+1}$ is the tangent vector with $\mathbb{T}=\begin{bmatrix}
    0 & -1 \\
    1 & 0
\end{bmatrix}$.
For simplicity, we define function $W(\bx,\bN):\mathbb{R}^2\times\mathbb{R}^2 \to \mathbb{R}$ as 
$$
W(\bx,\bN) = -\frac{1}{12}(\mathbb{T}\bN)^\top (\nabla^2 \bV\cdot\bN)(\bx)(\mathbb{T}n),
$$
then we have 
\begin{eqnarray} \label{error_edge}
    \int_{\overline{\bx^i_r\bx^i_{r+1}}} \bV \cdot \bN\dd l - L^i_{r,r+1}\bra{\hf\bV^i_r+\hf\bV^i_{r+1}} \cdot \bN^i_{r,r+1} = (L^i_{r,r+1})^3 W(\bx^i_{r+\hf},\bN^i_{r,r+1}) + \cO\bra{(L^i_{r,r+1})^5}.
\end{eqnarray} 
It should be noticed that $W(\bx,\bN)$ is a smooth function as a result of the smoothness of $\bV$.

In order to fully utilize the advantages of quadrilateral meshes, a new set of notation is introduced for this appendix. Let $\bx_{k+\hf,l}, \bx_{k,l+\hf}$ be the middle points of edge $\overline{\bx_{k,l}\bx_{k+1,l}}$ and $\overline{\bx_{k,l}\bx_{k,l+1}}$. For any smooth physical field $\varphi$, $\varphi_{k,l}$, $\varphi_{k+\hf,l}$ and $\varphi_{k,l+\hf}$ are the field $\varphi$ at point $\bx_{k,l}$, $\bx_{k+\hf,l}$ and $\bx_{k,l+\hf}$ separately. The unit normal vector and length of $\overline{\bx_{k,l}\bx_{k+1,l}}$ is denoted as $\bN_{k+\hf,l}$ and $L_{k+\hf,l}$. $\bN_{k,l+\hf}$ and $L_{k,l+\hf}$ to $\overline{\bx_{k,l}\bx_{k,l+1}}$ is defined similarly (see Fig. \ref{fig_Mij}). 

{\color{black}
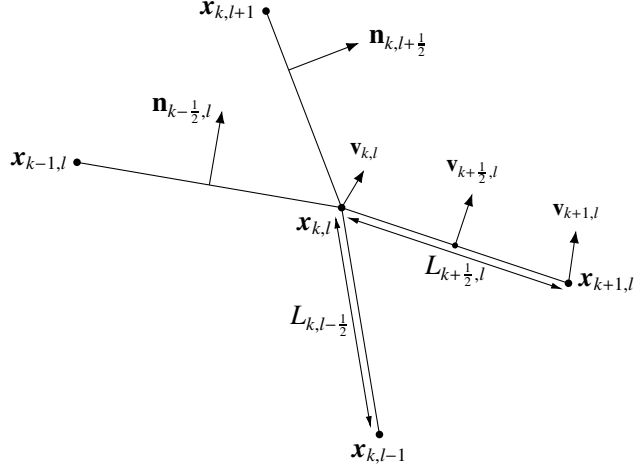
\begin{figure}[!htp]
\centering
\begin{tikzpicture}[scale=1]
\coordinate (M0) at (0,0);
\coordinate (Mi0) at (-3.5,0.6);
\coordinate (Mi1) at (3,-1);
\coordinate (Mj0) at (0.5,-3);
\coordinate (Mj1) at (-1,2.6);

\fill (M0) circle[radius=0.05];
\fill (Mi0) circle[radius=0.05];
\fill (Mi1) circle[radius=0.05];
\fill (Mj0) circle[radius=0.05];
\fill (Mj1) circle[radius=0.05];
\draw (Mi0) node[left] {$\bx_{k-1,l}$}
--(M0) node[below left] {$\bx_{k,l}$}
--(Mi1) node[right] {$\bx_{k+1,l}$};
\draw (Mj0) node[below] {$\bx_{k,l-1}$}
--(M0)
--(Mj1) node[left] {$\bx_{k,l+1}$};

\coordinate (V0) at (0.3,0.5);
\coordinate (V1) at (0.23,0.69);
\coordinate (V2) at (0.1,0.7);
\draw[-{Latex}] (M0) -- ++(V0)
node[font=\small, above] {$\bV_{k,l}$};
\draw[-{Latex}] ($(M0)!0.5!(Mi1)$) -- ++(V1)
node[font=\small, above] {$\bV_{k+\hf,l}$};
\draw[-{Latex}] (Mi1) -- ++(V2)
node[font=\small, above] {$\bV_{k+1,l}$};
\fill ($(M0)!0.5!(Mi1)$) circle[radius=0.04];

\coordinate (dNj0) at ($(M0)!0.1cm!-90:(Mj0)-(M0)$);
\draw[{Latex[width=0.08cm]}-{Latex[width=0.08cm]}]
($(M0)!0.03!(Mj0) + (dNj0)$)
-- ($(M0)!0.97!(Mj0) + (dNj0)$);
\node[left] at ($(M0)!0.5!(Mj0)$) {$L_{k,l-\hf}$};
\coordinate (dNi1) at ($(M0)!0.1cm!-90:(Mi1)-(M0)$);
\draw[{Latex[width=0.08cm]}-{Latex[width=0.08cm]}]
($(M0)!0.03!(Mi1) + (dNi1)$)
-- ($(M0)!0.99!(Mi1) + (dNi1)$);
\node[below] at ($(M0)!0.5!(Mi1)$) {$L_{k+\hf,l}$};

\coordinate (Ni0) at ($(M0)!1cm!-90:(Mi0)-(M0)$);
\coordinate (Nj1) at ($(M0)!1cm!-90:(Mj1)-(M0)$);
\draw[-{Latex}] ($(M0)!0.5!(Mi0)$) -- ++ (Ni0)
node[left] {$\bN_{k-\hf,l}$};
\draw[-{Latex}] ($(M0)!0.7!(Mj1)$) -- ++ (Nj1)
node[right] {$\bN_{k,l+\hf}$};

\end{tikzpicture}
\caption{New notations around $\bx_{k,l}$}\label{fig_Mij}
\end{figure}
}

\begin{lemma}[Geometric properties of $M_h$]\label{lem:Mh_geo}
For a set of uniformly refined quadrilateral mesh $\brc{M_h}_{h\in\mathbb{R}}$ defined above, we have 
\begin{enumerate}
\item $L_{k+\hf,l},\; L_{k,l+\hf}=\cO(h)$,
\item $|L_{k+\hf,l+1} - L_{k+\hf,l}|,\; |L_{k+1,l+\hf}-L_{k,l+\hf}| =  \cO(h^2)$,
\item $|\bN_{k+\hf,l+1}-\bN_{k+\hf,l}|,\; |\bN_{k+1,l+\hf}-\bN_{k,l+\hf}|=\cO(h)$,
\end{enumerate}
\end{lemma}
\begin{proof}[Proof of Lemma \ref{lem:Mh_geo}]
By definition of $M_h$, it can be easily checked for one dimension that
\begin{eqnarray*}
    L_{k+\hf,l} = |\bx_{k+1,l}-\bx_{k,l}| = |\bphi((k+1)h,lh)-\bphi(kh,lh)| = \cO(h), \\
    |L_{k+\hf,l+1} - L_{k+\hf,l}| \le |\bx_{k+1,l+1}-\bx_{k,l+1} -\bx_{k+1,l} + \bx_{k,l}| = \cO(h^2), \\ 
    |\bN_{k+\hf,l+1}-\bN_{k+\hf,l}| = \left|\frac{\bx_{k+1,l+1} - \bx_{k,l+1}}{L_{k+\hf,l+1}} - \frac{\bx_{k+1,l}-\bx_{k,l}}{L_{k+\hf,l}}\right| = \cO(h).
\end{eqnarray*}
Based on symmetry, the statement for other dimension is also true.  
\end{proof}

With the new notation, the error $(\tau^i_{\mathrm{ref}})_t - (\tau^i)_t$ can be formulated in detail. Without losing generality, consider cell $\omega^i$ with $\bx_{k,l},\bx_{k+1,l},\bx_{k,l+1},\bx_{k+1,l+1}$ as its vertices. Recall Eqn. \eqref{error_edge}, we have 
\begin{eqnarray*}
    & &(\tau^i_{\mathrm{ref}})_t - (\tau^i)_t \\
    &=&\frac{1}{m_i}\left[L_{k+\hf,l+1}^3W(\bx_{k+\hf,l+1}, \bN_{k+\hf,l+1}) - L_{k+\hf,l}^3W(\bx_{k+\hf,l},\bN_{k+\hf,l})\right. \\
    & &\left. + L_{k+1,l+\hf}^3W(\bx_{k+1,l+\hf}, \bN_{k+1,l+\hf}) - L_{k,l+\hf}^3W(\bx_{k,l+\hf},\bN_{k,l+\hf}) + \cO(h^5)\right].
\end{eqnarray*}
Noticing that $|(\bx_{k+\hf,l+1},\bN_{k+\hf,l+1})-(\bx_{k+\hf,l},\bN_{k+\hf,l})| = \cO(h)$ according to Lemma \ref{lem:Mh_geo}, which means 
$$W(\bx_{k+\hf,l+1}, \bN_{k+\hf,l+1}) -W(\bx_{k+\hf,l},\bN_{k+\hf,l}) = \cO(h).$$
As a result, we have  
$$L_{k+\hf,l+1}^3W(\bx_{k+\hf,l+1}, \bN_{k+\hf,l+1}) - L_{k+\hf,l}^3W(\bx_{k+\hf,l},\bN_{k+\hf,l}) = \cO(h^4).$$ 
Similarly, we have the same analysis for the other direction. 
With $m_i =\int_{\omega^i}\rho\dd\omega = \cO(h^2)$,  we finally get 
$$
(\tau^i_{\mathrm{ref}})_t - (\tau^i)_t = \cO(h^2).
$$

\section*{Appendix \hypertarget{Apd.4th}{B}. Proof of flux accuracy}
This appendix primarily illustrates the relationship between flux accuracy and the movement of the grid. 

For a set of uniformly refined quadrilateral mesh $\brc{M_h}_{h\in\mathbb{R}}$ and a smooth field $\mathbf{U}$, Simpson's rule provides a 4th-order accurate approximation for the Lagrangian flux across the edge $[r, r+1]$ of cell $\omega^i$:
\begin{eqnarray} \label{eq:reference_flux}
\bF_S = \begin{bmatrix}
-\bra{\ah \bVi_{r} + \fh \bVi_{r+\hf} + \ah \bVi_{r+1}}\cdot \bN^i_{r,r+1} \\
\bra{\ah P^i_{r} + \fh P^i_{r+\hf} + \ah P^i_{r+1}}\bN^i_{r,r+1}           \\
\bra{\ah P^i_{r}\bVi_{r} + \fh P^i_{r+\hf}\bVi_{r+\hf} + \ah P^i_{r+1}\bVi_{r+1}}\cdot \bN^i_{r,r+1}
\end{bmatrix}
\end{eqnarray}
Based on the augmented nodal solver, the Lagrangian flux across the edge $[r, r+1]$ of cell $\omega^i$ is:
\begin{eqnarray}\label{eq:numerical_flux}
\bF_M = \begin{bmatrix}
-\hf \bra{\tbVi_{r}+\tbVi_{r+1}} \cdot \bN^i_{r,r+1} \\
\hf \bra{\tPsi_{r,r+\hf} + \tPsi_{r+\hf,r+1}}
\bN^i_{r,r+1}                                        \\
\hf \bra{\tPsi_{r,r+\hf}\tbVi_{r}
+ \tPsi_{r+\hf,r+1}\tbVi_{r+1}}\cdot \bN^i_{r,r+1}
\end{bmatrix}.
\end{eqnarray}
Our goal is to prove 
$$\mathbf{F}_M - \mathbf{F}_S = \mathcal{O}(h^4).$$
To be more specific, mesh $M_h$ in $\brc{M_h}_{h\in\mathbb{R}}$ is defined by its nodes $\bx_{k,l}$ with $k,l\in \mathbb{Z}$, and we assumed that the refinement follows $\bx_{k,l} = \bphi(kh,lh)$ with  $\bphi:\mathbb{R}^2\to\mathbb{R}^2$ is a smooth homeomorphism. 

The key step to prove is to explore the properties of $\delta\tbV$ as below. 

\begin{lemma}[Properties of $\delta \widetilde{\mathbf{v}}= \widetilde{\mathbf{v}}- \mathbf{v}$] \label{lem:deltaV_properties}
For a set of uniformly refined quadrilateral mesh $\brc{M_h}_{h\in\mathbb{R}}$ and a smooth field $\mathbf{U}$, the nodal velocity correction $\delta \tbV^i_r$ satisfies:
\begin{enumerate}
\item $\delta \tbV^i_r = \mathcal{O}(h^2)$.
\item $\delta \tbV^i_{r+1} - \delta \tbV^i_r = \mathcal{O}(h^3)$.
\item $\bN^i_{r,r+1} \cdot \bra{ \delta \tbV^i_r + \delta \tbV^i_{r+1} } = 2 w^i_{r+\hf} + \mathcal{O}(h^4)$.
\end{enumerate}
\end{lemma}
\begin{proof}[Proof of Lemma \ref{lem:deltaV_properties}]
In order to fully utilize the advantages of quadrilateral meshes, a new set of notation is introduced for the proof of this lemma. Let $\bx_{k+\hf,l}, \bx_{k,l+\hf}$ be the middle points of edge $\overline{\bx_{k,l}\bx_{k+1,l}}$ and $\overline{\bx_{k,l}\bx_{k,l+1}}$. For any smooth physical field $\varphi$, $\varphi_{k,l}$, $\varphi_{k+\hf,l}$ and $\varphi_{k,l+\hf}$ are the field $\varphi$ at point $\bx_{k,l}$, $\bx_{k+\hf,l}$ and $\bx_{k,l+\hf}$ separately. The unit normal vector and length of $\overline{\bx_{k,l}\bx_{k+1,l}}$ is denoted as $\bN_{k+\hf,l}$ and $L_{k+\hf,l}$. $\bN_{k,l+\hf}$ and $L_{k,l+\hf}$ to $\overline{\bx_{k,l}\bx_{k,l+1}}$ is defined similarly (see Fig. \ref{fig_Mij}). 

Subsequently, reconsider Eqn. \eqref{maire_vel_hoo} for node $\bx_{k,l}$ with the smoothness of fields, one can get
\begin{eqnarray} \label{maire_vel^ij}
\sum_{m \in \mathcal{N}(k,l)} L_m \left[ \left( \frac{1}{3}\alpha_{k,l} + \frac{2}{3}\alpha_m \right) (\delta\tbV_{k,l} \cdot \bN_m) - \alpha_m w_m \right] \bN_m = \boldsymbol{0},\quad \mathcal{N}(k,l) = \left\{(k\pm\tfrac{1}{2},l),\ (k,l\pm\tfrac{1}{2})\right\},
\end{eqnarray}
where $w_{k\pm\hf,l} \coloneqq \fh \bra{\bV_{k\pm\hf,l} - \hf \bV_{k,l} - \hf\bV_{k\pm 1,l}}\cdot\bN_{k\pm\hf,l}$ and $w_{k,l\pm\hf} \coloneqq \fh \bra{\bV_{k,l\pm\hf} - \hf \bV_{k,l} - \hf\bV_{k,l\pm 1}}\cdot\bN_{k,l\pm\hf}$. 
Rewriting \eqref{maire_vel^ij} into a linear system as
\begin{eqnarray*}
\mathbb{M}_{k,l} \delta \tbVs_{k,l}-\boldsymbol{b}_{k,l} = \boldsymbol{0},
\end{eqnarray*}
where
\begin{eqnarray*}
\left\{
\begin{aligned}
&\mathbb{M}_{k,l}
= \sum_{m \in \mathcal{N}(k,l)} L_m  \left( \frac{1}{3}\alpha_{k,l} + \frac{2}{3}\alpha_m \right)  \bN_m   \bN_m ^{\top} =\alpha_{k,l}\left\{
L_{kk}
\brb{\bN^{0}_{k,l}\bN^{0,\top}_{k,l} + \cO(h^2)}+
L_{ll}
\brb{\bN^{1}_{k,l}\bN^{1,\top}_{k,l} + \cO(h^2)}\right\},
\\
& \boldsymbol{b}_{k,l} =
\sum_{m \in \mathcal{N}(k,l)}L_{m}\alpha_{m}w_{m}\bN_{m}
= \frac{\alpha_{k,l}}{2}\left\{ 
L_{kk}
\bra{ w_{k+\hf,l}+  w_{k-\hf,l}}
\brb{\bN^{0}_{k,l}+\cO(h^2)}
+ L_{ll}
\bra{ w_{k,l+\hf}+  w_{k,l-\hf}}
\brb{\bN^{1}_{k,l}+\cO(h^2)}\right\},
\end{aligned}
\right.
\end{eqnarray*}
with
\begin{eqnarray*}
L_{kk}=L_{k+\hf,l}+L_{k-\hf,l},\quad
L_{ll}=L_{k,l+\hf}+L_{k,l-\hf},\quad
\bN^0_{k,l} = \frac{\bN_{k+\hf,l} + \bN_{k-\hf,l}}{|\bN_{k+\hf,l} + \bN_{k-\hf,l}|}, \quad
\bN^1_{k,l} = \frac{\bN_{k,l+\hf} + \bN_{k,l-\hf}}{|\bN_{k,l+\hf} + \bN_{k,l-\hf}|}.
\end{eqnarray*}
It can be observed that $\mathbb{M}_{k,l}$ and $\boldsymbol{b}_{k,l}$ consist two independent parts related to $\bN^0_{k,l}$ and $\bN^1_{k,l}$.
Comparing the two components separately, one can get
\begin{eqnarray} \label{dv_3}
\left\{
\begin{aligned}
& \bN_{k,l}^0\cdot \delta \tbVs_{k,l} = \hf\bra{w_{k-\hf,l}+w_{k+\hf,l}} \brb{1+ \mathcal{O}(h^2)}, \\
& \bN_{k,l}^1\cdot \delta \tbVs_{k,l} = \hf\bra{w_{k,l-\hf}+w_{k,l+\hf}} \brb{1+ \mathcal{O}(h^2)},
\end{aligned}
\right.
\end{eqnarray}

With this essential relations, three properties of $\delta\tbVs$ can be found.
Firstly, noticing that $w_{k+\hf,l}, w_{k,l+\hf} = \cO(h^2)$, it is obvious that
\begin{eqnarray*}
\delta\tbVs_{k,l} = \cO(h^2),
\end{eqnarray*}
Then, taking the smoothness into consideration,
\begin{eqnarray*}
\delta\tbVs_{k+1,l} - \delta\tbV_{k,l} = \cO(h^3),\quad
\delta\tbVs_{k,l+1} - \delta\tbV_{k,l} = \cO(h^3)
\end{eqnarray*}
can be readily verified.
At last, adding Eqs. \eqref{dv_3} from two adjacent nodes, $M_{k,l}$ and $M_{k+1,l}$ for instance, we get
\begin{eqnarray*}
\bN^0_{k,l} \cdot \delta \tbVs_{k,l} + \bN^0_{k+1,l} \cdot \delta \tbVs_{k+1,l} &=& \hf\bra{w_{k-\hf,l} + 2w_{k+\hf,l} + w_{k+\frac{3}{2},l}} \brb{1+ \mathcal{O}(h^2)} \\
&=& 2w_{k+\hf,l}\brb{1+ \mathcal{O}(h^2)}.
\end{eqnarray*}
Meanwhile, the left formula is
\begin{eqnarray*}
\bN^0_{k,l} \cdot \delta \tbVs_{k,l} + \bN^0_{k+1,l} \cdot \delta \tbVs_{k+1,l} = \bN_{k+\hf,l} \cdot \bra{\delta\tbV_{k,l} + \delta\tbV_{k+1,l}} \brb{1+ \mathcal{O}(h^2)}.
\end{eqnarray*}
The final property is obtained from the above two equations as
\begin{eqnarray*}
\bN_{k+\hf,l} \cdot \bra{\delta\tbVs_{k,l} + \delta\tbVs_{k+1,l}} = 2w_{k+\hf,l}\brb{1+ \mathcal{O}(h^2)} = 2w_{k+\hf,l} + \cO(h^4).
\end{eqnarray*}
Similar result can be acquired in another direction as
\begin{eqnarray*}
\bN_{k,l+\hf} \cdot \bra{\delta\tbV_{k,l} + \delta\tbV_{k,l+1}} = 2w_{k,l+\hf} + \cO(h^4).
\end{eqnarray*}

In summary, three proved properties of $\delta\tbVs$ are listed as
\begin{gather*}
\delta\tbVs_{k,l} = \cO(h^2), \\
\left\{
\begin{aligned}
\delta\tbVs_{k+1,l} - \delta\tbVs_{k,l} = \cO(h^3) \\
\delta\tbVs_{k,l+1} - \delta\tbVs_{k,l} = \cO(h^3)
\end{aligned}
\right. \, , \\
\left\{
\begin{aligned}
\bN_{k+\hf,l} \cdot \bra{\delta\tbVs_{k,l} + \delta\tbVs_{k+1,l}} = 2w_{k+\hf,l} + \cO(h^4) \\
\bN_{k,l+\hf} \cdot \bra{\delta\tbVs_{k,l} + \delta\tbVs_{k,l+1}} = 2w_{k,l+\hf} + \cO(h^4)
\end{aligned}
\right. \, .
\end{gather*}
which play an essential role in following lemma.
\end{proof}

\begin{lemma}[Flux accuracy from $\delta\tbV$]
For a set of uniformly refined mesh $\brc{M_h}_{h\in\mathbb{R}}$ and a smooth field $\mathbf{U}$, if the nodal velocity correction $\delta \tbV^i_r$ satisfies:
\begin{enumerate}
\item $\delta \tbV^i_r = \mathcal{O}(h^2)$,
\item $\delta \tbV^i_{r+1} - \delta \tbV^i_r = \mathcal{O}(h^3)$,
\item $\bN^i_{r,r+1} \cdot \bra{ \delta \tbV^i_r + \delta \tbV^i_{r+1} } = 2 w^i_{r+\hf} + \mathcal{O}(h^4)$,
\end{enumerate} 
then the flux accuracy satisfies 
$$\mathbf{F}_M - \mathbf{F}_S = \mathcal{O}(h^4).$$
\end{lemma}
\begin{proof}[Proof of Lemma 2]
For start, the difference between first component is computed as
\begin{eqnarray*}
& & \hf \bra{\tbVi_{r}+\tbVi_{r+1}} \cdot \bN^i_{r,r+1}
-  \bra{\ah \bVi_{r} + \fh \bVi_{r+\hf} + \ah \bVi_{r+1}}\cdot \bN^i_{r,r+1} \\
&=& \hf \bra{\delta\tbVi_r + \delta\tbVi_{r+1}} \cdot \bN^i_{r,r+1}
-  \fh \bra{\bVi_{r+\hf} - \hf\bVi_{r} - \hf\bVi_{r+1}} \cdot \bN^i_{r,r+1} \\
&=& \hf \bra{\delta\tbVi_r + \delta\tbVi_{r+1}} \cdot \bN^i_{r,r+1} - w^i_{r+\hf} = \cO(h^4),
\end{eqnarray*}
where $\delta\tbVi_r= \tbVi_r-\bVi_r$.

Recall the jump condition Eqs. \eqref{maire_pres_hoo} for surface $[r,r+1]$ as
\begin{eqnarray*}
\tPsi_{r,r+\hf} &=&
\bra{\frac{1}{3}P^{i}_{r} +\frac{2}{3} P^{i}_{r+\hf}}
- \left[
\bra{\frac{1}{3}\alpha^{i}_{r}+\frac{2}{3}\alpha^{i}_{r+\hf}}\delta\tbVi_r \cdot \bN^i_{r,r+1}
-\alpha^{i}_{r+\hf} w^{i}_{r+\hf}
\right], \\
\tPsi_{r+\hf,r+1} &=&
\bra{\frac{1}{3}P^{i}_{r+1} +\frac{2}{3} P^{i}_{r+\hf}}
- \left[
\bra{\frac{1}{3}\alpha^{i}_{r+1}+\frac{2}{3}\alpha^{i}_{r+\hf}}\delta\tbVi_{r+1}\cdot \bN^i_{r,r+1}
-\alpha^{i}_{r+\hf} w^{i}_{r+\hf}
\right],
\end{eqnarray*}
then the pressure difference can be simplified as
\begin{eqnarray*}
& & \bra{\ah P^i_{r} + \fh P^i_{r+\hf} + \ah P^i_{r+1}}\bN^i_{r,r+1}
- \hf \bra{\tPsi_{r,r+\hf} + \tPsi_{r+\hf,r+1}}
\bN^i_{r,r+1} \\
&=& \hf \left[
\bra{\frac{1}{3}\alpha^{i}_{r}+\frac{2}{3}\alpha^{i}_{r+\hf}}\delta\tbVi_r\cdot \bN^i_{r,r+1}
+ \bra{\frac{1}{3}\alpha^{i}_{r+1}+\frac{2}{3}\alpha^{i}_{r+\hf}}\delta\tbVi_{r+1}\cdot \bN^i_{r,r+1}
\right] -\alpha^{i}_{r+\hf} w^{i}_{r+\hf} \\
&=& \hf \alpha^i_{r+\hf}\bra{\delta\tbVi_r+\delta\tbVi_{r+1}}\cdot \bN^i_{r,r+1} \brb{1+\cO(h^2)} - \alpha^{i}_{r+\hf} w^{i}_{r+\hf} \\
&=& \alpha^i_{r+\hf}w^i_{r+\hf} \brb{1+\cO(h^2)} - \alpha^{i}_{r+\hf} w^{i}_{r+\hf} = \cO(h^4).
\end{eqnarray*}

At last, the energy term of $\bF_M$ is
\begin{eqnarray*}
& & \hf \bra{\tPsi_{r,r+\hf}\tbVi_{r}
+ \tPsi_{r+\hf,r+1}\tbVi_{r+1}}\cdot \bN^i_{r,r+1} \\
&=& \hf \left[
\bra{\frac{1}{3}P^{i}_{r} +\frac{2}{3} P^{i}_{r+\hf}}
- \bra{\frac{1}{3}\alpha^{i}_{r}+\frac{2}{3}\alpha^{i}_{r+\hf}}
\delta\tbVi_r\cdot \bN^i_{r,r+1}
+\alpha^{i}_{r+\hf} w^{i}_{r+\hf}
\right] \tbVi_{r}\cdot \bN^i_{r,r+1} \\
&+& \hf \left[
\bra{\frac{1}{3}P^{i}_{r+1} +\frac{2}{3} P^{i}_{r+\hf}}
- \bra{\frac{1}{3}\alpha^{i}_{r+1}+\frac{2}{3}\alpha^{i}_{r+\hf}}\delta\tbVi_{r+1}\cdot \bN^i_{r,r+1}
+\alpha^{i}_{r+\hf} w^{i}_{r+\hf} \right]
\tbVi_{r+1}\cdot \bN^i_{r,r+1} \\
&=& \underbrace{\bra{\frac{1}{6}P^{i}_{r} +\frac{1}{3} P^{i}_{r+\hf}}\tbVi_{r}\cdot \bN^i_{r,r+1} + \bra{\frac{1}{6}P^{i}_{r+1} +\frac{1}{3} P^{i}_{r+\hf}}\tbVi_{r+1}\cdot\bN^i_{r,r+1}}_\mathrm{I} \\
& & - \alpha^i_{r+\hf}
\bra{\hf\delta\tbVi_r\cdot\bN^i_{r,r+1} + \hf\delta\tbVi_r\cdot\bN^i_{r,r+1}}
\bra{\hf\tbVi_r\cdot\bN^i_{r,r+1} + \hf\tbVi_r\cdot\bN^i_{r,r+1}}\brb{1+\cO(h^2)}\\
& & + \alpha^i_{r+\hf} w^i_{r+\hf}\bra{\hf\tbVi_r\cdot\bN^i_{r,r+1} + \hf\tbVi_r\cdot\bN^i_{r,r+1}} \\
&=& \mathrm{I} - \alpha^i_{r+\hf} w^i_{r+\hf}\bra{\hf\tbVi_r\cdot\bN^i_{r,r+1} + \hf\tbVi_r\cdot\bN^i_{r,r+1}}\brb{1+\cO(h^2)} + \alpha^i_{r+\hf} w^i_{r+\hf}\bra{\hf\tbVi_r\cdot\bN^i_{r,r+1} + \hf\tbVi_r\cdot\bN^i_{r,r+1}} \\
&=& \mathrm{I} + \cO(h^4).
\end{eqnarray*}
Meanwhile,
\begin{eqnarray*}
\mathrm{I} &=& \bra{\frac{1}{6}P^{i}_{r} +\frac{1}{3} P^{i}_{r+\hf}}\bV^i_{r}\cdot \bN^i_{r,r+1}
+ \bra{\frac{1}{6}P^{i}_{r+1} +\frac{1}{3} P^{i}_{r+\hf}}\bV^i_{r+1}\cdot\bN^i_{r,r+1} \\
&+& \bra{\frac{1}{6}P^{i}_{r} +\frac{1}{3} P^{i}_{r+\hf}}\delta\tbVi_{r}\cdot \bN^i_{r,r+1}
+ \bra{\frac{1}{6}P^{i}_{r+1} +\frac{1}{3} P^{i}_{r+\hf}}\delta\tbVi_{r+1}\cdot\bN^i_{r,r+1} \\
&=& \bra{\frac{1}{6}P^{i}_{r} +\frac{1}{3} P^{i}_{r+\hf}}\bV^i_{r}\cdot \bN^i_{r,r+1}
+ \bra{\frac{1}{6}P^{i}_{r+1} +\frac{1}{3} P^{i}_{r+\hf}}\bV^i_{r+1}\cdot\bN^i_{r,r+1} \\
&+& \hf P^i_{r+\hf}\bra{\delta\tbVi_{r}+\delta\tbVi_{r+1}}\cdot \bN^i_{r,r+1} \brb{1+\cO(h^2)} \\
&=&\bra{\frac{1}{6}P^{i}_{r} +\frac{1}{3} P^{i}_{r+\hf}}\bV^i_{r}\cdot \bN^i_{r,r+1}
+ \bra{\frac{1}{6}P^{i}_{r+1} +\frac{1}{3} P^{i}_{r+\hf}}\bV^i_{r+1}\cdot\bN^i_{r,r+1} \\
&+& \frac{2}{3}P^i_{r+\hf} \bra{\bV^i_{r+\hf} - \hf \bV^i_r -\hf \bV^i_{r+1}}\cdot \bN^i_{r,r+1} + \cO(h^4) \\
&=& \bra{\frac{1}{6}P^{i}_{r}  \bV^i_{r}
+ \frac{2}{3}P^i_{r+\hf} \bV^i_{r+\hf}
+ \frac{1}{6}P^{i}_{r+1} \bV^i_{r+1}}\cdot\bN^i_{r,r+1} + \cO(h^4).
\end{eqnarray*}
Together, the difference of energy term is
\begin{eqnarray*}
\bra{\ah P^i_{r}\bVi_{r} + \fh P^i_{r+\hf}\bVi_{r+\hf} + \ah P^i_{r+1}\bVi_{r+1}}\cdot \bN^i_{r,r+1} - \hf \bra{\tPsi_{r,r+\hf}\tbVi_{r}
+ \tPsi_{r+\hf,r+1}\tbVi_{r+1}}\cdot \bN^i_{r,r+1} = \cO(h^4).
\end{eqnarray*}

Finally, it is proved that the Lagrangian flux given by our augmented nodal solver has the $4$-th order accuracy.
\end{proof}

\section*{Appendix \hypertarget{Apd.PV}{C}. Derivates reconstruction} 
In order to reconstruct derivates $\hat{\U}_x, \hat{\U}_y$ and one-sided derivates $\hat{\U}_{x,L}, \hat{\U}_{x,R}, \hat{\U}_{y,L}, \hat{\U}_{y,R}$ used in Section \ref{chap:Lag:LMCV:PVflux}, it is sufficient to reconstruct for each scalar component $\varphi$. Let $\mathcal{N}(\bx^i_r)$ be the index set of cells adjacent to $\bx^i_r$, then a least squares problem is introduced as
\begin{eqnarray*}
\min_{\nabla\varphi, \nabla^2 \varphi} \sum_{j\in \mathcal{N}(\bx^i_r)}|D\varphi(\bx^j_c) - \nabla_{\bx} \varphi^j(\bx^i_r)|^2
\end{eqnarray*}
where $\varphi^j$ is the high order reconstruction function of cell $\omega^j$ defined in Section \ref{chap:Lag:LMCV:space}, and
\begin{eqnarray*}
D\varphi(\bx) = \nabla \varphi + \nabla^2 \varphi (\bx - \bx^i_r),
\end{eqnarray*}
where $\nabla\varphi \in \mathbb{R}^2$ and $\nabla^2 \varphi\in M_2(\mathbb{R})$ are undetermined coefficients of the least squares problem, which gives $(\hat{\varphi}_x,\hat{\varphi}_y)^{\top} = \nabla \varphi$ after solved.

One-sided derivates are obtained by a much straightforward way, since they only participate in numerical viscosity. Taking $\hat{\varphi}_{x,L}, \hat{\varphi}_{x,R}$ for example, weight averages on both sides are considered as
\begin{eqnarray*}
\hat{\varphi}_{x,L} =\frac{\displaystyle
\sum_{j \in \mathcal{N}_{x,L}(\bx^i_r) } \frac{\partial \varphi^j}{\partial x}(\bx^i_r)\, (\bx^j_c-\bx^i_r)\cdot \mathbf{e}_0
}
{\displaystyle
\sum_{j \in \mathcal{N}_{x,L}(\bx^i_r) } (\bx^j_c-\bx^i_r)\cdot \mathbf{e}_0
}, \\
\hat{\varphi}_{x,R} =\frac{\displaystyle
\sum_{j \in \mathcal{N}_{x,R}(\bx^i_r) } \frac{\partial \varphi^j}{\partial x}(\bx^i_r)\,  (\bx^j_c-\bx^i_r)\cdot \mathbf{e}_0
}
{\displaystyle
\sum_{j \in \mathcal{N}_{x,R}(\bx^i_r) } (\bx^j_c-\bx^i_r)\cdot \mathbf{e}_0
},
\end{eqnarray*}
where $\mathbf{e}_0 = (1,0)^{\top}$ and 
\begin{gather*}
\mathcal{N}_{x,L} = \brc{j,\;(\bx^j_c-\bx^i_r)\cdot \mathbf{e}_0 < 0 }\cap\mathcal{N}(\bx^i_r), \\
\mathcal{N}_{x,R} = \brc{j, \;(\bx^j_c-\bx^i_r)\cdot \mathbf{e}_0 > 0 }\cap\mathcal{N}(\bx^i_r),
\end{gather*}
which indicate the left and right neighborhood of $\bx^i_r$. 

In the same way, $\hat{\varphi}_{y,L}$ and $\hat{\varphi}_{y,R}$ are defined as
\begin{eqnarray*}
\hat{\varphi}_{y,L} =\frac{\displaystyle
\sum_{\mathcal{N}_{y,L}(\bx^i_r)} \frac{\partial \varphi^j}{\partial y}(\bx^i_r)\, (\bx^j_c-\bx^i_r)\cdot \mathbf{e}_1
}
{\displaystyle
\sum_{\mathcal{N}_{y,L}(\bx^i_r)} (\bx^j_c-\bx^i_r)\cdot \mathbf{e}_1
}, \\
\hat{\varphi}_{y,R} =\frac{\displaystyle
\sum_{\mathcal{N}_{y,R}(\bx^i_r)} \frac{\partial \varphi^j}{\partial y}(\bx^i_r)\,  (\bx^j_c-\bx^i_r)\cdot \mathbf{e}_1
}
{\displaystyle
\sum_{\mathcal{N}_{y,R}(\bx^i_r)} (\bx^j_c-\bx^i_r)\cdot \mathbf{e}_1
}
\end{eqnarray*}
with $\mathbf{e}_1 = (0,1)^{\top}$ and 
\begin{gather*}
\mathcal{N}_{y,L} = \brc{j,\;(\bx^j_c-\bx^i_r)\cdot \mathbf{e}_1 < 0 }\cap\mathcal{N}(\bx^i_r), \\
\mathcal{N}_{y,R} = \brc{j, \;(\bx^j_c-\bx^i_r)\cdot \mathbf{e}_1 > 0 }\cap\mathcal{N}(\bx^i_r).
\end{gather*}

\end{document}